\documentclass[a4paper,onecolumn,english]{article}

\usepackage[T1]{fontenc}
\usepackage[utf8]{inputenc}
\usepackage{babel}
\usepackage{pdflscape}
\usepackage{enumitem}
\usepackage{amsmath}
\usepackage{amsfonts}
\usepackage{amssymb}
\usepackage{amsthm}
\usepackage{thmtools}
\usepackage{dsfont}
\usepackage{bm}
\usepackage{relsize}
\usepackage[title]{appendix}
\usepackage[left=2.5cm,right=2.5cm,top=2.5cm,bottom=2.5cm]{geometry}
\usepackage[colorlinks=true,
            linkcolor = blue,
            urlcolor  = blue,
            citecolor = red,
			breaklinks]{hyperref}
\allowdisplaybreaks

\numberwithin{equation}{section}

\makeatletter
	\let\save@mathaccent\mathaccent
	\newcommand*\if@single[3]{%
	  \setbox0\hbox{${\mathaccent"0362{#1}}^H$}%
	  \setbox2\hbox{${\mathaccent"0362{\kern0pt#1}}^H$}%
	  \ifdim\ht0=\ht2 #3\else #2\fi
	  }
	\newcommand*\rel@kern[1]{\kern#1\dimexpr\macc@kerna}
	\newcommand*\widebar[1]{\@ifnextchar^{{\wide@bar{#1}{0}}}{\wide@bar{#1}{1}}}
	\newcommand*\wide@bar[2]{\if@single{#1}{\wide@bar@{#1}{#2}{1}}{\wide@bar@{#1}{#2}{2}}}
	\newcommand*\wide@bar@[3]{%
	  \begingroup
	  \def\mathaccent##1##2{%
	    \let\mathaccent\save@mathaccent
	    \if#32 \let\macc@nucleus\first@char \fi
	    \setbox\z@\hbox{$\macc@style{\macc@nucleus}_{}$}%
	    \setbox\tw@\hbox{$\macc@style{\macc@nucleus}{}_{}$}%
	    \dimen@\wd\tw@
	    \advance\dimen@-\wd\z@
	    \divide\dimen@ 3
	    \@tempdima\wd\tw@
	    \advance\@tempdima-\scriptspace
	    \divide\@tempdima 10
	    \advance\dimen@-\@tempdima
	    \ifdim\dimen@>\z@ \dimen@0pt\fi
	    \rel@kern{0.6}\kern-\dimen@
	    \if#31
	      \overline{\rel@kern{-0.6}\kern\dimen@\macc@nucleus\rel@kern{0.4}\kern\dimen@}%
	      \advance\dimen@0.4\dimexpr\macc@kerna
	      \let\final@kern#2%
	      \ifdim\dimen@<\z@ \let\final@kern1\fi
	      \if\final@kern1 \kern-\dimen@\fi
	    \else
	      \overline{\rel@kern{-0.6}\kern\dimen@#1}%
	    \fi
	  }%
	  \macc@depth\@ne
	  \let\math@bgroup\@empty \let\math@egroup\macc@set@skewchar
	  \mathsurround\z@ \frozen@everymath{\mathgroup\macc@group\relax}%
	  \macc@set@skewchar\relax
	  \let\mathaccentV\macc@nested@a
	  \if#31
	    \macc@nested@a\relax111{#1}%
	  \else
	    \def\gobble@till@marker##1\endmarker{}%
	    \futurelet\first@char\gobble@till@marker#1\endmarker
	    \ifcat\noexpand\first@char A\else
	      \def\first@char{}%
	    \fi
	    \macc@nested@a\relax111{\first@char}%
	  \fi
	  \endgroup
	}
	
	\def\@seccntformat#1{\@ifundefined{#1@cntformat}%
	   {\csname the#1\endcsname\quad}  % default
	   {\csname #1@cntformat\endcsname}% enable individual control
	}
	\let\oldappendix\appendix %% save current definition of \appendix
	\renewcommand\appendix{%
	    \oldappendix
	    \newcommand{\section@cntformat}{\appendixname$:\;$~\thesection}
	}
\makeatother

\newcommand{\eps}{\varepsilon}
\newcommand{\mb}[1]{\mathbf{#1}}
\newcommand*{\ccdot}{\kern-.12em\cdot\kern-.12em}
\newcommand{\supp}{\mathrm{supp}}
\newcommand{\warrow}{\overset{w}{\longrightarrow}}
\newcommand{\varrow}{\overset{v}{\longrightarrow}}
\newcommand{\conv}[1]{
	\mskip-1.8\thinmuskip\smash{\begin{array}[t]{c} \mathlarger{\ast} \\[-8pt] \mathsmaller{#1} \end{array}}\mskip-1.2\thinmuskip
	}
\newcommand{\convdiam}[1]{
	\mskip-1.8\thinmuskip\smash{\begin{array}[t]{c} \diamond \\[-8pt] \mathsmaller{#1} \end{array}}\mskip-1.2\thinmuskip
	}
\newcommand{\convstar}[1]{
	\mskip-2.2\thinmuskip\smash{\begin{array}[t]{c} \bm{\star} \\[-8pt] \mathsmaller{#1} \end{array}}\mskip-1.7\thinmuskip
	}
\newcommand\restrict[1]{\raisebox{-.5ex}{$|$}_{#1}}

\newtheoremstyle{slplain}% name
  {0.4cm}% Space above
  {0.4cm}% Space below
  {\upshape}% Body font
  {}%Indent amount (empty = no indent, \parindent = para indent)
  {\bfseries}%  Thm head font
  {.}%       Punctuation after thm head
  { }%      Space after thm head: " " = normal interword space;
  {}%       Thm head spec
  
\newtheoremstyle{itplain}% name
    {0.4cm}% Space above
    {0.4cm}% Space below
    {\itshape}% Body font
    {}%Indent amount (empty = no indent, \parindent = para indent)
    {\bfseries}%  Thm head font
    {.}%       Punctuation after thm head
    { }%      Space after thm head: " " = normal interword space;
    {}%       Thm head spec
  
\declaretheorem[style=slplain,numberwithin=section]{definition}
\declaretheorem[style=slplain,sibling=definition]{example}
\declaretheorem[style=slplain,sibling=definition]{remark}

\declaretheorem[style=itplain,sibling=definition]{theorem}

\declaretheorem[style=itplain,sibling=definition]{proposition}
\declaretheorem[style=itplain,sibling=definition]{lemma}
\declaretheorem[style=itplain,sibling=definition]{corollary}

\AtBeginDocument{
	\addtolength{\abovedisplayskip}{0.125ex}
	\addtolength{\abovedisplayshortskip}{0.125ex}
	\addtolength{\belowdisplayskip}{0.125ex}
	\addtolength{\belowdisplayshortskip}{0.125ex}
}

\renewenvironment{abstract}{%
\noindent\hfill\begin{minipage}{0.92\textwidth}
\rule{\textwidth}{1pt}}
{\par\noindent\rule{\textwidth}{1pt}\end{minipage}\hfill}

\let\OLDthebibliography\thebibliography
\renewcommand\thebibliography[1]{
  \OLDthebibliography{#1}
  \setlength{\parskip}{0pt}
  \setlength{\itemsep}{3pt plus 0.3ex}
}

\title{\bfseries Product formulas and convolutions for two-dimensional Laplace-Beltrami operators: beyond the trivial case\\[0.3cm]}

\author{Rúben Sousa
\thanks{Corresponding author. CMUP, Departamento de Matemática, Faculdade de Ciências, Universidade do Porto, Rua do Campo Alegre 687, 4169-007 Porto, Portugal. Email: \texttt{rubensousa@fc.up.pt}}
\and
Manuel Guerra \thanks{ISEG – School of Economics and Management, Universidade de Lisboa; REM – Research in Economics and Mathematics, CEMAPRE, Rua do Quelhas 6, 1200-781 Lisbon, Portugal. Email: \texttt{mguerra@iseg.ulisboa.pt}}
\and
Semyon Yakubovich \thanks{CMUP, Departamento de Matemática, Faculdade de Ciências, Universidade do Porto, Rua do Campo Alegre 687, 4169-007 Porto, Portugal. Email: \texttt{syakubov@fc.up.pt}}
\\[0.3cm]
}
\date{\today\\}

\begin{document}

\maketitle

\begin{abstract}
	 \small
	 \parbox{\linewidth}{\vspace{-2pt}
\begin{center} \bfseries Abstract \vspace{-8pt} \end{center}
	 
	 \-\ \quad 
	 We introduce the notion of a family of convolution operators associated with a given elliptic partial differential operator. Such a convolution structure is shown to exist for a general class of Laplace-Beltrami operators on two-dimensional manifolds endowed with cone-like metrics. This structure gives rise to a convolution semigroup representation for the Markovian semigroup generated by the Laplace-Beltrami operator.
	 
	 \-\ \quad 
	 In the particular case of the operator $\mathcal{L} = \partial_x^2 + {1 \over 2x} \partial_x + {1 \over x} \partial_\theta^2$ on $\mathbb{R}^+ \times \mathbb{T}$, we deduce the existence of a convolution structure for a two-dimensional integral transform whose kernel and inversion formula can be written in closed form in terms of confluent hypergeometric functions. The results of this paper can be interpreted as a natural extension of the theory of one-dimensional generalized convolutions to the framework of multiparameter eigenvalue problems.
	 \vspace{5pt}
	 
     \-\ \quad \textbf{Keywords:}  Laplace-Beltrami operator, generalized convolution, product formula, eigenfunction expansion, convolution semigroup, multiparameter eigenvalue problem. \vspace{6.5pt}
     }
\end{abstract}

\vspace{8pt}

\begingroup
\let\clearforchapter\relax
\section{Introduction}
\endgroup

The problem of constructing generalized convolutions associated with a given Sturm-Liouville operator $\mathcal{L}$ has been widely studied in the literature. Such generalized convolutions, whose defining property is that they should trivialize the eigenfunction expansion of $\mathcal{L}$ in the same way as the ordinary convolution trivializes the Fourier transform, allow one to develop the basic notions of harmonic analysis in parallel with the standard theory. Among other applications, this construction yields a convolution semigroup representation for the kernel of the heat semigroup $\{e^{-t\mathcal{L}}\}_{t \geq 0}$; in other words, it enables us to interpret the diffusion process generated by $\mathcal{L}$ as a generalized Lévy process. It is known that such convolution structures exist for a wide class of Sturm-Liouville operators on bounded or unbounded intervals \cite{bloomheyer1994,sousaetal2020,sousaetal2019b}.

It is quite natural to wonder if one can also construct generalized convolutions associated with elliptic operators in spaces of dimension $d>1$. Indeed, the heat semigroup $\{e^{-t\Delta}\}_{t \geq 0}$ generated by the Laplacian on $\mathbb{R}^d$ is a convolution semigroup with respect to the ordinary convolution on $\mathbb{R}^d$, and this suggests that it may be possible to devise a similar property for other elliptic operators on subsets of $\mathbb{R}^d$ or on $d$-dimensional Riemannnian manifolds. However, this task turns out to be considerably more difficult than in the one-dimensional setting, as we now explain. 

Let $M$ be a Riemannian manifold and $m$ a positive measure on $M$. Consider a self-adjoint elliptic partial differential operator $\mathcal{L}$ on $L^2(M,m)$ whose eigenfunction expansion (cf.\ \cite{garding1954}, \cite[Theorem XIV.6.6]{dunfordschwartz1963}) is of the form $(\mathcal{F}h)_k(\lambda) = \int_M h(\xi) \, \omega_{k,\lambda}(\xi) \, m(d\xi)$ ($\lambda \in \mathbb{R},$\, $k=0,1,\ldots$). The basic requirement that a generalized convolution $\bm{\star}$ associated with $\mathcal{L}$ should satisfy is the following \cite{sousaetal2019b,trimeche1997,volkovich1988}: $\bm{\star}$ is a bilinear operator on the space of finite complex measures on $M$ such that for any two such measures $\mu, \nu$ we have
\begin{equation} \label{eq:coneconv_intro_convtriv}
\int_M \omega_{k,\lambda}\, d(\mu\bm{\star} \nu) = \Bigl(\int_M \omega_{k,\lambda} \, d\mu\Bigr) \ccdot \Bigl(\int_M \omega_{k,\lambda} \, d\nu \Bigr) \qquad (\lambda \in \mathbb{R},\,  k=0,1,\ldots).
\end{equation}
In particular, the (generalized) eigenfunctions $\omega_{k,\lambda}$ should satisfy the product formula
\begin{equation} \label{eq:coneconv_intro_prodform}
\omega_{k,\lambda}(\xi_1) \, \omega_{k,\lambda}(\xi_2) = \int_M \omega_{k,\lambda}\, d\pi_{\xi_1,\xi_2}
\end{equation}
where the measures $\pi_{\xi_1,\xi_2} = \delta_{\xi_1} \bm{\star} \delta_{\xi_2}$ do not depend on $\lambda$ and $k$. Usually one also requires that these are positive measures, so that the convolution is positivity-preserving. In certain specific cases where the $\omega_{k,\lambda}$ can be written in terms of classical special functions, a closed-form expression for the product formula measures $\pi_{\xi_1,\xi_2}$ has been determined \cite{berezansky1998,bloomheyer1994,koornwinderschwartz1997,laine1980,nessibitrimeche1997,sousaetal2019a,trimeche1995}; however, the generality of such results is severely limited. In the one-dimensional setting, a general theorem on the existence of such product formulas has been established via a PDE technique \cite{sousaetal2019b,zeuner1992} which consists in studying the properties of the hyperbolic equation $\mathcal{L}_{\xi_1} u = \mathcal{L}_{\xi_2} u$ which is satisfied by the product of the eigenfunctions (here $\mathcal{L}_\xi$ denotes the differential operator acting on the variable $\xi$). But this technique does not extend to the multidimensional case, because the equation $\mathcal{L}_{\xi_1} u = \mathcal{L}_{\xi_2} u$ becomes ultrahyperbolic and the corresponding Cauchy problem becomes ill-posed (see e.g.\ \cite{courant1962}). The challenge is therefore to understand which extra assumptions need to be imposed to make it possible to prove the existence of the product formula \eqref{eq:coneconv_intro_prodform} for the eigenfunctions of general elliptic partial differential operators. (Here the meaning of `general' is that one should include operators for which closed-form expressions for the eigenfunctions are not available.)

The goal of this paper is to construct product formulas and convolutions for elliptic operators on two-dimensional manifolds $M$ which admit \emph{separation of variables} in the sense that the eigenfunctions are of the form $\omega_{k,\lambda}(\xi) = \psi_{k,\lambda}(x)\, \phi_{k,\lambda}(y)$\, ($\xi = (x,y) \in M$).

To outline the ideas behind our construction, we first observe that if $M = M_1 \times M_2$ is a product of Riemannian manifolds endowed with the product metric (so that the Laplace-Beltrami operator on $M_1 \times M_2$ obviously admits separation of variables) and if there exists a convolution for the Laplace-Beltrami operator on both $M_1$ and $M_2$, then we can trivially define a convolution associated with the Laplace-Beltrami operator on $M$ by taking the product of the convolutions on $M_1$ and $M_2$ (see Example \ref{exam:coneconv_exam_conetriv}). Here we focus on manifolds of the form $M = \mathbb{R}_0^+ \times M_2$\, ($\mathbb{R}_0^+ = [0,\infty)$)\, where instead of the product metric structure we consider Riemannian structures defined by the so-called \emph{cone-like metrics}, i.e.\ possibly singular metrics of the form $g = dx^2 + A(x)^2 g_{M_2}$; note that this metric structure is a generalization of the metric cone \cite{cheeger1980}. In the body of the paper we restrict ourselves to settings where $M_2$ is one-dimensional, but this is not an essential restriction. The Laplace-Beltrami operator of $(M,g)$ is 
\[
\Delta = \partial_x^2 + {A'(x) \over A(x)} \partial_x + {1 \over A(x)^2} \Delta_2
\]
where $\Delta_2$ is the Laplace-Beltrami operator of $M_2$. The operator $\Delta$ admits separation of variables: its eigenfunctions $\omega_{k,\lambda}$ can be written as the product of eigenfunctions of $\Delta_2$ and eigenfunctions of Sturm-Liouville operators of the form $\Delta_{1,k} := {d^2 \over dx^2} + {A'(x) \over A(x)} {d \over dx} - {\eta_k \over A(x)^2}$, where $\eta_k \geq 0$ is a separation constant. One of our contributions is to determine conditions on the function $A$ which ensure that a product formula exists for the eigenfunctions of each of the operators $\Delta_{1,k}$. If the eigenfunctions of $\Delta_2$ also admit a product formula, then this gives rise to a product formula for $\omega_{k,\lambda}$ which is analogous to \eqref{eq:coneconv_intro_prodform}, but with the difference that in general the measure of the product formula also depends on the separation parameter $k$.

In fact, this dependence on $k$ shows that the problem of constructing a convolution for $\Delta$ has a positive answer provided that the operator $\bm{\star}$ (for which \eqref{eq:coneconv_intro_convtriv} should hold) is allowed to depend on $k$. This naturally leads to the notion of a \emph{family of convolutions} associated with a given elliptic operator. These are in fact families of hypergroups, because each of its convolutions endows the space $M$ with a hypergroup structure, as defined by Jewett \cite{bloomheyer1994,jewett1975}. As we will see, many properties of the convolution algebras determined by one-dimensional (generalized) convolutions can be extended to the families of convolutions discussed here. In particular, such families provide a natural generalization of the convolution semigroup property for the heat semigroup generated by the Laplace-Beltrami operator (see Corollary \ref{cor:coneconv_rbm_convsemigroup}).

As noted above, our approach is applicable to a general class of elliptic operators whose eigenfunctions are generally not expressible in terms of special functions. In addition, it allows us to recover as particular cases the existence of product formulas and convolutions for certain (eigen)functions of hypergeometric type. Our main example (Example \ref{exam:coneconv_exam_conehalf}) is the elliptic operator $\mathcal{L} = \partial_x^2 + {1 \over 2x} \partial_x + {1 \over x} \partial_\theta^2$ on $(0,\infty) \times \mathbb{T}$, which belongs to a family of Laplace-Beltrami operators on conic surfaces that has been the object of recent studies \cite{boscainprandi2016,boscainneel2019+,prandietal2018}. We show that the eigenfunctions $\omega_{k,\lambda}$ of $\mathcal{L}$ can be written in terms of confluent hypergeometric (or Whittaker) functions, and that the eigenfunction expansion $(\mathcal{F}h)_k(\lambda) = \int_{(0,\infty) \times \mathbb{T}} h(x,\theta) \, \omega_{k,\lambda}(x,\theta) \, \sqrt{x}\, dx d\theta$ admits a closed-form inversion formula. The proof of this inversion formula relies on an apparently little known result on the spectral measure of the generator of a drifted Bessel process. Using our general construction, we obtain the existence of positivity-preserving convolutions $\convstar{k}$ associated with the two-dimensional integral transform $\mathcal{F}$ via the linearisation property $(\mathcal{F}(h \convstar{k}g))_k = (\mathcal{F}h)_k \ccdot (\mathcal{F}g)_k$; these convolutions give rise to a decomposition of the law of the diffusion process generated by $\mathcal{L}$ in terms of the laws of drifted Bessel processes.

Our assumption that the elliptic operator $\mathcal{L}$ admits separation of variables is equivalent to requiring that the eigenfunction equation $-\mathcal{L}u = \lambda u$ can be reduced to a multiparameter Sturm-Liouville eigenvalue problem (see Remark \ref{rmk:coneconv_prodform_multiparam}). The main contribution of this paper can therefore be restated as follows: \emph{for a general class of multiparameter Sturm-Liouville problems, it is possible to construct associated generalized convolution structures in which the basic notions of harmonic analysis can be developed in analogy with the classical theory.} We hope that the results of this study are a first step towards a more general theory of convolutions associated with multiparameter eigenvalue problems.

The structure of the paper is the following. Section \ref{sec:eigenexp} gives background on the eigenfunction expansion of Laplace-Beltrami operators on $M = \mathbb{R}_0^+ \times \mathbb{T}$ endowed with cone-like metrics and on the spectral representation of the heat kernel. The product formula for eigenfunctions of the Laplace-Beltrami operator is established in Section \ref{sec:prodformconv}, where we also define the associated family of convolutions and describe the main continuity and mapping properties of the convolution structure. In Section \ref{sec:infdivsemigr} we introduce the notion of infinitely divisible measures and convolution semigroups and we discuss the convolution semigroup properties of operator semigroups determined by the Laplace-Beltrami operator. Some special cases where the convolution is related with functions of hypergeometric type are presented in Section \ref{sec:examples}. Finally, in Section \ref{sec:ellipticR2} we show that the construction of the previous sections can be developed on manifolds of the form $M = \mathbb{R}_0^+ \times I$, where $I$ is an interval of the real line. The Appendix collects some known results on one-dimensional generalized convolutions which are used throughout the paper.

\medskip

\paragraph{Notation.} In the sequel, $\mathbb{R}^+$ and $\mathbb{R}_0^+$ stand for the positive and nonnegative numbers respectively. We denote by $\mathrm{C}(E)$ the space of continuous complex-valued functions on a given space $E$;\, $\mathrm{C}_\mathrm{b}(E)$, $\mathrm{C}_0(E)$, $\mathrm{C}_\mathrm{c}(E)$ and $\mathrm{C}^k(E)$ denote, respectively, its subspaces of bounded continuous functions, of continuous functions vanishing at infinity, of continuous functions with compact support and of $k$ times continuously differentiable functions. For $-\infty \leq a < b \leq \infty$, $\mathrm{AC_{loc}}(a,b)$ is the space of locally absolutely continuous functions on the interval $(a,b)$. For a given measure $\mu$ on a measurable space $E$, $L^p(E,\mu)$ ($1 \leq p \leq \infty$) is the Lebesgue space of complex-valued $p$-integrable functions with respect to $\mu$. The indicator function of a subset $B \subset E$ is denoted by $\mathds{1}_B(\cdot)$. The set of all probability (respectively, finite positive, finite complex) Borel measures on $E$ is denoted by $\mathcal{P}(E)$ (respectively, $\mathcal{M}_+(E)$, $\mathcal{M}_\mathbb{C}(E)$), and we denote by $\delta_x$ the Dirac measure at a point $x$. For given measures $\mu, \mu_1, \mu_2, \ldots \in \mathcal{M}_\mathbb{C}(E)$, we write $\mu_n \warrow \mu$ (respectively $\mu_n \varrow \mu$) if the measures $\mu_n$ converge weakly (resp.\ vaguely) to $\mu$ as $n \to \infty$, i.e.\ if $\lim_{n \to \infty} \int_E g(\xi) \mu_n(d\xi) = \int_E g(\xi) \mu(d\xi)$ for all $g \in \mathrm{C}_\mathrm{b}(E)$ (resp.\ for all $g \in \mathrm{C}_0(E)$).
\vspace{10pt}

\section{The eigenfunction expansion of the Laplace-Beltrami operator} \label{sec:eigenexp}

Throughout the paper we consider the (possibly singular) Riemannian manifold $(M,g)$, where $M = \mathbb{R}_0^+ \times \mathbb{T}$ (with $\mathbb{T} = \mathbb{R} / \mathbb{Z}$) and the ($\mathrm{C}^1$, possibly non-smooth) Riemannian metric is given by
\begin{equation} \label{eq:coneconv_Riemmetric}
g = dx^2 + A(x)^2 d\theta^2 \qquad (0 \leq x < \infty, \; \theta \in \mathbb{T})
\end{equation}
where the function $A$ is such that
\begin{equation} \label{eq:coneconv_Ahyp}
A \in \mathrm{C}(\mathbb{R}_0^+) \cap \mathrm{C}^1(\mathbb{R}^+), \quad A(x) > 0\text{ for } x>0, \quad \int_0^1 {dx \over A(x)} < \infty, \quad {A' \over A} \text{ is nonnegative and decreasing}.
\end{equation}
The Riemannian volume form on $M$ is $d \omega= \sqrt{\mathrm{det}\, g}\, dx d\theta=A(x) dx d\theta$.
Thus, the Riemannian gradient of a function $u: M \longrightarrow \mathbb C$ is 
\[
\nabla u = \left( \partial_x u , \frac 1 {A^2} \partial_\theta u \right),
\]
and the Laplace-Beltrami operator is
\begin{equation} \label{eq:coneconv_laplbeltr}
\Delta = \mathrm{div}\circ \nabla = \partial_x^2 + \frac{A^\prime (x)}{A(x)} \partial_x + \frac 1{A(x)^2} \partial_\theta^2 .
\end{equation}
To introduce the closure of $\Delta$ with reflecting boundary at $x=0$, we proceed in the standard way (see e.g. \cite{schmudgen2012}).
Consider the Sobolev space $H^1(M) \equiv H^1(M,\omega) = \left\{ u \in L^2 (M, \omega) \mid \nabla u \in L^2 (M, \omega)\right\}$, and the sesquilinear form $\mathcal{E} :H^1(M) \times H^1(M) \longrightarrow \mathbb C$, defined as
\begin{equation}
\label{Eq bilinear form}
\mathcal{E}(u,v) =
\left\langle \nabla u, \nabla v \right\rangle_{L^2(M, \omega)} =
\int_M \left( \partial_x u \mskip0.7\thinmuskip \overline{\partial_x v} + \frac 1{A(x)^2} \partial_\theta u \mskip0.7\thinmuskip \overline{\partial_\theta v} \right) d\omega .
\end{equation}
It is clear from \eqref{Eq bilinear form} that $\mathcal{E}$ is symmetric, positive semidefinite and,  since its graph norm
$\left\| u \right\|_{\Gamma (\mathcal{E})}^2 = \left\| u \right\|_{L^2(M,\omega)}^2 + \mathcal{E}(u,u) $ coincides with the norm of $H^1(M)$, it is a closed sesquilinear form.
Let
\[
\mathcal D_N = \bigl\{ u \in H^1(M) \bigm| \exists v \in L^2(M,\omega) \text{ such that } \langle \nabla u, \nabla z \rangle = - \langle v, z \rangle \text{ for all } z \in H^1(M)\bigr\}.
\]
The mapping $u \mapsto \Delta_Nu = v$ is a well defined linear operator in the domain $\mathcal D_N$.
It follows from the above that $ \left( \Delta_N, \mathcal D_N \right)$ is a self-adjoint operator in $L^2(M,\omega)$ (\cite{schmudgen2012}, Theorem 10.7), called the {\it Neumann Laplacian}.
It is an extension of the Laplace-Beltrami operator defined in a domain of smooth functions satisfying the reflective boundary condition at $x=0$
\[
\left( A \mskip0.7\thinmuskip \partial_x u \right)(0, \theta) = 0 \qquad \forall \theta \in \mathbb T.
\]

We use the notations $L^p(M) = L^p(M, \omega)$,\, $L^p(A) = L^p(\mathbb R^+, A(x) dx)$, and consider the Fourier decomposition
\begin{equation}
\label{Eq Fourier decomposition}
L^2(M) = \bigoplus_{k \in \mathbb Z} H_k,
\qquad
H_k = \bigl\{ e^{i2k\pi \theta} v(x) \bigm| v \in L^2(A) \bigr\},
\end{equation}
where $H_k$ are regarded as Hilbert spaces with inner product $\left\langle e^{i2k\pi \theta }u, e^{i2k\pi \theta }v \right\rangle_{H_k} = \left\langle u, v \right\rangle_{L^2(A)}$.
The direct sum is also regarded as a Hilbert space with inner product
$\left\langle \{ u_k \}, \{ v_k \} \right\rangle_{\bigoplus\! H_k} = \smash{\sum\limits_{k \in \mathbb Z} \left\langle u_k, v_k \right\rangle_{H_k}}$, such that for 
$u(x,\theta) =\sum\limits_{k \in \mathbb Z} e^{i2k\pi \theta}u_k(x)$, 
$v(x,\theta)=\sum\limits_{k \in \mathbb Z} e^{i2k\pi \theta}v_k(x)$,
we have
\begin{align*}
\left\langle u, v \right\rangle_{L^2(M)} & =
\sum_{k \in \mathbb Z} \left\langle u_k,v_k \right\rangle_{L^2(A)}
\\
\mathcal{E} (u,v) & = \sum_{k \in \mathbb Z} \mathcal{E}_k (u_k,v_k),
\end{align*}
where $\mathcal{E}_k(u_k,v_k) = \int_{\mathbb R^+} \bigl( u_k^\prime (x) \overline{v_k^\prime (x)} + \frac{(2k\pi)^2}{A(x)^2}u_k(x)\overline{v_k(x)} \mskip0.7\thinmuskip \bigr) A(x) dx$ are sesquilinear forms with domains
\[
\mathcal D( \mathcal{E}_k ) =
\Bigl\{
u \in L^2(A)\cap \mathrm{AC}_{\mathrm{loc}}(\mathbb R^+) \Bigm| \frac{2k\pi}{A}u \in L^2(A), \ u^\prime \in L^2(A) 
\Bigr\}.
\]
Thus, we obtain the decomposition, compatible with \eqref{Eq Fourier decomposition}:
\[
H^1(M) = \bigoplus_{k \in \mathbb Z} \mathcal D(\mathcal{E}_k) .
\]
It can be checked that the forms $\mathcal{E}_k$ are symmetric, positive, and closed.
Therefore, a similar argument allows us to construct self-adjoint realizations of the Sturm-Liouville operators $\Delta_k u (x) = u^{\prime \prime} (x) + \frac{A^\prime (x)}{A(x)} u^\prime(x) - \frac{(2k\pi)^2}{A(x)^2}u(x) $, whose domain is
\[
\mathcal D( \Delta_k) = \bigl\{ 
u \in L^2(A) \bigm| u, u^\prime \in \mathrm{AC}_{\mathrm{loc}}(\mathbb R^+), \ \Delta_k u \in L^2(A), \ (Au^\prime)(0) = 0 \bigr\},
\]
for $k \in \mathbb Z$.
This provides a decomposition of the Neumann Laplacian:
\begin{equation}
\label{Eq Fourier decomposition DN Laplacian}
\mathcal D_N = \bigoplus_{k \in \mathbb Z} \mathcal D(\Delta_k) ,
\qquad
\Delta_N \biggl( \, \sum_{k \in \mathbb Z} e^{i2k\pi \theta} u_k(x) \biggr) =
\sum_{k \in \mathbb Z} e^{i2k\pi \theta} \Delta_k u_k(x) .
\end{equation}

The first step towards the construction of convolution operators associated with $\Delta$ is the following characterization of the separable solutions of the eigenfunction equation $-\Delta u = \lambda u$ with reflecting boundary condition at $x=0$.

\begin{lemma} \label{lem:coneconv_wsol}
For each $(k, \lambda) \in \mathbb{Z} \times \mathbb{C}$, there exists a unique solution $w_{k,\lambda} \in H_{k,\infty} := \bigl\{e^{i2k\pi\theta} v(x) \bigm| v \in \mathrm{C}(\mathbb{R}_0^+)\bigr\}$ of the boundary value problem
\begin{equation} \label{eq:coneconv_wsol_bvp}
-\Delta u = \lambda u, \qquad\;\; u(0,\theta) = e^{i2k\pi\theta}, \qquad\;\; u^{[1]}(0,\theta) = 0
\end{equation}
where $u^{[1]}(x,\theta) = A(x) \, (\partial_x u)(x,\theta)$. Moreover, $\lambda \mapsto w_{k,\lambda}(x,\theta)$ is, for each fixed $(x,\theta) \in M$ and $k \in \mathbb{Z}$, an entire function of exponential type.
\end{lemma}

\begin{proof}
Clearly, any such solution must be of the form $w_{k,\lambda}(x,\theta) = e^{i2k\pi\theta} v(x)$, where $v$ is a solution of 
\begin{equation} \label{eq:coneconv_wsol_pf1}
-\Delta_k v(x) = \lambda v(x), \qquad v(0) = 1, \qquad (Av')(0) = 0.
\end{equation}
The result therefore follows from a standard existence and uniqueness theorem for solutions of Sturm-Liouville boundary value problems (Lemma \ref{lem:SLappendix_boundaryvaluepb}).
\end{proof}

The unique solution of \eqref{eq:coneconv_wsol_pf1} will be denoted by $v_{k,\lambda}(x)$, so that $w_{k,\lambda}(x,\theta) = e^{i2k\pi\theta} v_{k,\lambda}(x)$. Throughout the paper we will make frequent use of the normalized form of the function $v_{k,\lambda}$ which is described in the following lemma (whose proof is elementary):

\begin{lemma} \label{lem:coneconv_sol_normaliz}
Define $\widetilde{v}_{k,\lambda}(x) := {v_{k,\lambda}(x) \over \zeta_k(x)}$, where $\zeta_k(x) := \cosh\bigl(2k\pi \int_0^x \! {dy \over A(y)}\bigr)$. Then $\widetilde{v}_{k,\lambda}(\bm{\cdot})$ is a solution of
\begin{equation} \label{eq:coneconv_wsol_pf2}
\ell_k(\widetilde{v}) = \lambda \widetilde{v}, \qquad \widetilde{v}(0) = 1, \qquad (B_k\widetilde{v}')(0) = 0
\end{equation}
where
\begin{equation} \label{eq:coneconv_wsol_pf3}
\ell_k(g) := -{1 \over B_k}(B_k g')', \qquad B_k(x) := A(x) \, \zeta_k(x)\rule[-.3\baselineskip]{0pt}{0.85\baselineskip}^2.
\end{equation}
Moreover, we have ${B_k' \over B_k} = \eta + \phi$, where $\eta(x) = {4k\pi \over A(x)} \tanh(2k\pi \int_0^x {dy \over A(y)}) \geq 0$ and the functions $\phi = {A' \over A}$ and $\psi := {1 \over 2} \eta' - {1 \over 4} \eta^2 + {B_k' \over 2B_k} \eta = {(2k\pi)^2 \over A^2}$ are both decreasing and nonnegative.
\end{lemma}

The final assertion of the above lemma implies, in particular, that the assumption \eqref{eq:SLappendix_mainassump} of the Appendix holds for the coefficients $p = r = B_k$ of the Sturm-Liouville operator $\ell_k$.

In what follows, to lighten the notation, points of $M$ are denoted by $\bm{\xi} = (x,\theta)$, $\bm{\xi_1} = (x_1, \theta_1)$, etc.

It is well-known that the classical Weyl-Titchmarsh-Kodaira theory of eigenfunction expansions of Sturm-Liouville operators can be generalized to elliptic partial differential operators on higher-dimensional spaces, see e.g.\ \cite{garding1954}, \cite[Theorem XIV.6.6]{dunfordschwartz1963}. As remarked in \cite[p.\ 1713]{dunfordschwartz1963}, the knowledge about the boundary conditions satisfied by the kernels of the eigenfunction expansion is much smaller in the (general) multidimensional case, when compared to the one-dimensional setting. However, in the special case where separation of variables can be applied to the eigenvalue problem for the elliptic operator and therefore the eigenvalue equation reduces to a system of ordinary differential equations, further information on the eigenfunction expansion can be obtained from the theory of multiparameter eigenvalue problems. This connection will be further discussed in Remark \ref{rmk:coneconv_prodform_multiparam} below.

In particular, the Fourier decomposition \eqref{Eq Fourier decomposition DN Laplacian}, combined with the eigenfunction expansion of the Sturm-Liouville operator $-\Delta_k$, gives rise to an eigenfunction expansion of $(\Delta_N,\mathcal{D}_N)$ in terms of the separable solutions $w_{k,\lambda}$ defined in Lemma \ref{lem:coneconv_wsol}:

\begin{proposition} \label{prop:coneconv_Feigenexp}
There exists a sequence of locally finite positive Borel measures $\bm{\rho_k}$ on $\mathbb{R}_0^+$ such that the map $h \mapsto \bm{\mathcal{F}} h$, where
\begin{equation} \label{eq:coneconv_Feigenexp_dir}
(\bm{\mathcal{F}} h)_k(\lambda) := \int_M h(\bm{\xi}) \, w_{-k,\lambda}(\bm{\xi}) \, \omega(d\bm{\xi}) \qquad \bigl(k \in \mathbb{Z},\, \lambda \geq 0\bigr),
\end{equation}
is an isometric isomorphism $\bm{\mathcal{F}}: L^2(M) \longrightarrow \bigoplus_{k \in \mathbb{Z}} L^2(\mathbb{R}_0^+, \bm{\rho_k})$ whose inverse is given by
\begin{equation} \label{eq:coneconv_Feigenexp_inv}
(\bm{\mathcal{F}}^{-1} \{\varphi_k\})(\bm{\xi}) = \sum_{k \in \mathbb{Z}} \int_{\mathbb{R}_0^+} \varphi_k(\lambda) \, w_{k,\lambda}(\bm{\xi}) \, \bm{\rho_k}(d\lambda).
\end{equation}
The convergence of the integral in \eqref{eq:coneconv_Feigenexp_dir} is understood with respect to the norm of $L^2(\mathbb{R}_0^+,\bm{\rho_k})$, and the convergence of the inner integrals and the series in \eqref{eq:coneconv_Feigenexp_inv} is understood with respect to the norm of $L^2(M)$. Moreover, the operator $\bm{\mathcal{F}}$ is a spectral representation of $(\Delta_N, \mathcal{D}_N)$ in the sense that
\begin{align}
\label{eq:coneconv_Feigenexp_domain} & \mathcal{D}_N = \biggl\{h \in L^2(M) \biggm| \sum_{k \in \mathbb{Z}} \int_{\mathbb{R}_0^+} \lambda^2 \bigl|(\bm{\mathcal{F}} h)_k(\lambda)\bigr|^2 \bm{\rho_k}(d\lambda) < \infty \biggr\}\\[2pt]
\label{eq:coneconv_Feigenexp_spidentity}
& \bigl(\bm{\mathcal{F}} (-\Delta_N h)\bigr)_{\!k} (\lambda) = \lambda \ccdot (\bm{\mathcal{F}} h)_k(\lambda), \qquad\quad h \in \mathcal{D}_N, \; k \in \mathbb{Z}.
\end{align}
\end{proposition}

\begin{proof}
Let $h \in L^2(M)$. By Fubini's theorem, $h(x,\bm{\cdot}) \in L^2(\mathbb{T})$ for a.e.\ $x \in \mathbb{R}^+$. For these points $x$ we have
\begin{equation} \label{eq:coneconv_Feigenexp_pf1}
h(x,\theta) = \sum_{k \in \mathbb{Z}} \widehat{h}_k(x) \, e^{i2k\pi\theta}, \qquad \text{where } \widehat{h}_k(x) := \int_0^1 e^{-i2k\pi\vartheta} h(x,\vartheta) d\vartheta,
\end{equation}
the series converging in the norm of $L^2(\mathbb{T})$. It is straightforward that $\widehat{h}_k \in L^2(A)$ for all $k \in \mathbb{Z}$, and therefore the function $\widehat{h}_k$ can be represented in terms of the eigenfunction expansion of the Sturm-Liouville operator $\Delta_k$ (Proposition \ref{prop:SLappendix_eigenexp}): denoting the spectral measure of $\Delta_k$ by $\bm{\rho_k}$, we have
\[
\widehat{h}_k(x) = \! \int_{\mathbb{R}_0^+} \bigl(\mathcal{F}_{\Delta_k} \widehat{h}_k\bigr)(\lambda) \, v_{k,\lambda}(x) \, \bm{\rho_k}(d\lambda), \quad \text{ where } \;\; \bigl(\mathcal{F}_{\Delta_k} \widehat{h}_k\bigr)(\lambda) := \! \int_0^\infty \widehat{h}_k(y) \, v_{k,\lambda}(y) \, A(y) dy
\]
the integrals converging in the norms of $L^2(A)$ and $L^2(\bm{\rho_k}) \equiv L^2(\mathbb{R}_0^+,\bm{\rho_k})$ respectively.

By definition of $\widehat{h}_k$, we have $\bigl(\mathcal{F}_{\Delta_k} \widehat{h}_k\bigr)(\lambda) = \int_M h(\bm{\xi}) \, w_{-k,\lambda}(\bm{\xi}) \, \omega(d\bm{\xi}) \equiv (\bm{\mathcal{F}} h)_k(\lambda)$, with equality in the $L^2(\bm{\rho_k})$-sense. Furthermore, by a dominated convergence argument it is clear that $\widehat{h}_k(x) e^{i2k\pi\theta} = \int_{\mathbb{R}_0^+} \bigl(\mathcal{F}_{\Delta_k} \widehat{h}_k\bigr)(\lambda) \, w_{k,\lambda}(\bm{\xi}) \, \bm{\rho_k}(d\lambda)$ with equality in the $L^2(M)$-sense; therefore,
\[
h(x,\theta) = \sum_{k \in \mathbb{Z}} \int_{\mathbb{R}_0^+} \bigl(\mathcal{F}_{\Delta_k} \widehat{h}_k\bigr)(\lambda) \, w_{k,\lambda}(\bm{\xi}) \, \bm{\rho_k}(d\lambda)
\]
proving the inversion formula \eqref{eq:coneconv_Feigenexp_dir}--\eqref{eq:coneconv_Feigenexp_inv}. Finally, the fact that the integral operator $\bm{\mathcal{F}}$ is isometric follows from the identities 
\[
\| h \|_{L^2(M)}^2 = \sum_{k \in \mathbb{Z}} \bigl\| \widehat{h}_k \bigr\|_{L^2(A)}^2 = \sum_{k \in \mathbb{Z}} \bigl\| \mathcal{F}_{\Delta_k} \widehat{h}_k \bigr\|_{L^2(\bm{\rho_k})}^2 = \bigl\| \{(\bm{\mathcal{F}}h)_k\} \bigr\|_{\bigoplus\mskip-0.8\thinmuskip L^2(\bm{\rho_k})}^2
\]
where the first and second steps follow from the isometric properties of the classical Fourier series and the eigenfunction expansion of $\Delta_k$, respectively.

It only remains to justify the identities \eqref{eq:coneconv_Feigenexp_domain}--\eqref{eq:coneconv_Feigenexp_spidentity}. Using \eqref{Eq Fourier decomposition DN Laplacian} we obtain
\[
\bigl(\bm{\mathcal{F}} (-\Delta_N h)\bigr)_{\!k} (\lambda) = \bigl(\mathcal{F}_{\Delta_k} (\widehat{-\Delta_N h})_k\bigr) (\lambda) = \bigl(\mathcal{F}_{\Delta_k} (-\Delta_k \, \widehat{h}_k)\bigr) (\lambda) = \lambda \ccdot \bigl(\bm{\mathcal{F}} h\bigr)_{\!k} (\lambda), \qquad h \in \mathcal{D}_N
\]
which proves \eqref{eq:coneconv_Feigenexp_spidentity}. Now, let $h \in L^2(M)$ be such that $\sum_{k \in \mathbb{Z}} \int_{\mathbb{R}_0^+} \lambda^2 \bigl|(\bm{\mathcal{F}} h)_k(\lambda)\bigr|^2 \bm{\rho_k}(d\lambda) < \infty$. We know that $\mathcal{F}_{\Delta_k}$ is a spectral representation of $\Delta_k$ (Proposition \ref{prop:SLappendix_eigenexp}), which means in particular that
\[
\mathcal{D}(\Delta_k) = \biggl\{u \in L^2(A) \biggm| \int_{\mathbb{R}_0^+} \lambda^2 |(\mathcal{F}_{\Delta_k} u)(\lambda)|^2 \bm{\rho_k}(d\lambda) < \infty\biggr\}.
\]
Consequently, we have $(\widehat{h}_k)_{k \in \mathbb{Z}} \in \bigoplus_{k \in \mathbb{Z}} \mathcal{D}(\Delta_k)$ and $h = \sum_{k \in \mathbb{Z}} e^{i2k\pi\theta} \widehat{h}_k \in \mathcal{D}_N$. Conversely, if $h \in \mathcal{D}_N$ then by \eqref{eq:coneconv_Feigenexp_spidentity} we have $\{\lambda\ccdot (\bm{\mathcal{F}}h)_k(\lambda)\} \in \bigoplus_{k \in \mathbb{Z}} L^2(\bm{\rho_k})$, and we conclude that \eqref{eq:coneconv_Feigenexp_domain} holds.
\end{proof}

Since $\Delta_N$ is a negative self-adjoint operator, it is the infinitesimal generator of a strongly continuous semigroup in $L^2(M)$, denoted by $\left\{ e^{t \Delta_N} \right\}_{t \geq 0}$.
For any real-valued $u \in H^1(M)$, $|u| \in H^1(M)$ and $\mathcal{E} (|u|,|u|) \leq \mathcal{E}(u,u)$.
Therefore, $e^{t \Delta_N}$ is positivity-preserving for every $t \geq 0$.
Further, $|u| \wedge 1 \in H^1(M) $ and $\mathcal{E} \left( |u| \wedge 1 , |u| \wedge 1 \right) \leq \mathcal{E} \left(|u|,|u| \right)$. Thus, for every $p \in [1,+\infty]$ the subspace $L^2(M) \cap L^p(M)$ is invariant under $e^{t \Delta_N}$, for every $t \geq 0$.
The semigroup $\left\{ e^{t \Delta_N} \right\}_{t \geq 0}$ can be extended into a strongly continuous contraction semigroup in $L^p(M)$ (see e.g.\ \cite[Sections 1.3--1.4]{davies1989}).
In other words, $e^{t \Delta_N}$ is a strongly continuous Markov semigroup in $L^p(M)$ for every $p \in [1,\infty]$.
The analogous statement holds for the semigroup $\left\{ e^{t \Delta_k} \right\}_{t \geq 0}$ in $L^p(A)$, for every $k \in \mathbb Z$.

\begin{proposition} \label{prop:coneconv_heatkern_spectrep}
Assume that the action of $e^{t\Delta_N}$ on $L^2(M)$ is given by a symmetric heat kernel, i.e.\ there exists a measurable function $p: \mathbb{R}^+ \times M \times M \longrightarrow \mathbb{R}_0^+$ such that:
\begin{enumerate}[itemsep=0pt,topsep=4pt]
\item[\textbf{I.}] For all $t,s > 0$ and $\bm{\xi_1}, \bm{\xi_2} \in M$,
\[
p(t,\bm{\xi_1},\bm{\xi_2}) = p(t,\bm{\xi_2},\bm{\xi_1}) \quad \text{ and } \quad p(t+s,\bm{\xi_1},\bm{\xi_2}) = \int_M p(t,\bm{\xi_1},\bm{\xi_3}) \,  p(s,\bm{\xi_3}, \bm{\xi_2}) \, \omega(d\bm{\xi_3});
\]
\item[\textbf{II.}] For $t>0$, $h \in L^2(M)$ and $\omega$-a.e.\ $\bm{\xi_1} \in M$,
\[
(e^{t\Delta_N} h)(\bm{\xi_1}) = \int_M h(\bm{\xi_2}) \, p(t,\bm{\xi_1},\bm{\xi_2}) \, \omega(d\bm{\xi_2}).
\]
\end{enumerate} 
Then, for $t > 0$ and $\omega$-a.e.\ $\bm{\xi_1}, \bm{\xi_2} \in M$, the heat kernel admits the spectral representation
\begin{equation} \label{eq:coneconv_rbm_spectralrep}
p(t,\bm{\xi_1},\bm{\xi_2}) = \sum_{k \in \mathbb{Z}} \int_{\mathbb{R}_0^+} e^{-t\lambda} w_{k,\lambda}(\bm{\xi_1}) \, w_{-k,\lambda}(\bm{\xi_2}) \, \bm{\rho_k}(d\lambda)
\end{equation}
where the integral and the sum are absolutely convergent.
\end{proposition}

\begin{proof}
Fix $t > 0$. It follows from condition I that
\[
\int_M p(t,\bm{\xi_1},\bm{\xi_2})^2 \, \omega(d\bm{\xi_2}) =p(2t,\bm{\xi_1},\bm{\xi_1}) < \infty \qquad (\bm{\xi_1} \in M),
\]
meaning in particular that $p(t,\bm{\xi_1},\bm{\cdot}) \in L^2(M)$ for all $\bm{\xi_1} \in M$. Moreover, by the spectral representation property \eqref{eq:coneconv_Feigenexp_domain}--\eqref{eq:coneconv_Feigenexp_spidentity} we have $[\bm{\mathcal{F}} (e^{t\Delta_N} h)]_k (\lambda) = e^{-t\lambda} (\bm{\mathcal{F}} h)_k(\lambda)$ for all $h \in L^2(M)$, hence
\begin{equation} \label{eq:coneconv_heatkern_spectrep_pf1}
\begin{aligned}
& \sum_{k\in \mathbb{Z}} \int_{\mathbb{R}_0^+} (\bm{\mathcal{F}}h)_k(\lambda)\, [\bm{\mathcal{F}}p(t,\bm{\xi_1},\bm{\cdot})]_k(\lambda) \, \bm{\rho_k}(d\lambda) = (e^{t\Delta_N}h)(\bm{\xi_1}) \\
& \qquad \;\; = \bm{\mathcal{F}}^{-1}\bigl\{e^{-t \kern.08em \bm{\cdot}}(\bm{\mathcal{F}}h)_k(\bm{\cdot})\bigr\}(\bm{\xi_1}) = \sum_{k\in \mathbb{Z}} \int_{\mathbb{R}_0^+} e^{-t\lambda} (\bm{\mathcal{F}}h)_k(\lambda) \, w_{k,\lambda}(\bm{\xi_1}) \, \bm{\rho_k}(d\lambda)
\end{aligned}
\end{equation}
for $\omega$-a.e.\ $\bm{\xi_1} \in M$. Since $h \in L^2(M)$ is arbitrary, from \eqref{eq:coneconv_heatkern_spectrep_pf1} we deduce that $\{e^{-t\lambda} w_{k,\lambda}(\bm{\xi_1})\} = \{[\bm{\mathcal{F}}p(t,\bm{\xi_1},\bm{\cdot})]_k(\lambda)\} \in \bigoplus_{k \in \mathbb{Z}} L^2(\bm{\rho_k})$ for $\omega$-a.e.\ $\bm{\xi_1} \in M$. Therefore
\[
\sum_{k \in \mathbb{Z}} \int_{\mathbb{R}_0^+} e^{-t\lambda} |w_{k,\lambda}(\bm{\xi_1})|^2 \, \bm{\rho_k}(d\lambda) = \sum_{k \in \mathbb{Z}} \bigl\| e^{-t\lambda/2} w_{k,\lambda}(\bm{\xi_1}) \bigr\|_{L^2(\bm{\rho_k})}^2 < \infty
\]
and it follows (by the Cauchy-Schwarz inequality) that the right-hand side of \eqref{eq:coneconv_rbm_spectralrep} is absolutely convergent for $\omega$-a.e.\ $\bm{\xi_1}, \bm{\xi_2} \in M$. Moreover, the isometric property of $\bm{\mathcal{F}}$ yields
\[
p(t,\bm{\xi_1},\bm{\xi_2}) = \bigl\langle p(t/2,\bm{\xi_1},\bm{\cdot}), p(t/2,\bm{\xi_2},\bm{\cdot}) \bigr\rangle_{\!L^2(M)} = \sum_{k \in \mathbb{Z}} \bigl\langle e^{-t\lambda/2} w_{k,\lambda}(\bm{\xi_1}), e^{-t\lambda/2} w_{k,\lambda}(\bm{\xi_2}) \bigr\rangle_{L^2(\bm{\rho_k})}
\]
and therefore the identity \eqref{eq:coneconv_rbm_spectralrep} holds for $\omega$-a.e.\ $\bm{\xi_1}, \bm{\xi_2} \in M$.
\end{proof}

\begin{corollary} \label{cor:coneconv_wsol_eigeneq}
If the assumptions of Proposition \ref{prop:coneconv_heatkern_spectrep} are satisfied, then for $t \geq 0$, $k \in \mathbb{Z}$, $\lambda \in \supp(\bm{\rho_k})$ and $\omega$-a.e.\ $\bm{\xi_1} \in M$ we have
\[
e^{-t\lambda} w_{k,\lambda}(\bm{\xi_1}) = \int_M w_{k,\lambda}(\bm{\xi_2}) \, p(t,\bm{\xi_1},\bm{\xi_2}) \, \omega(d\bm{\xi_2}).
\]
\end{corollary}

\begin{proof}
Fix $t \geq 0$ and $k \in \mathbb{Z}$. Notice that
\[
p(t,\bm{\xi_1},\bm{\xi_2}) = \sum_{j \in \mathbb{Z}} e^{i2j\pi (\theta_1 - \theta_2)} p_{\Delta_j\!}(t,x_1,x_2),
\]
where $p_{\Delta_j\!}(t,x_1,x_2) = \int_{\mathbb{R}_0^+} e^{-t\lambda\,} v_{j,\lambda}(x_1) \, v_{j,\lambda}(x_2) \bm{\rho_j}(d\lambda)$ is the heat kernel for the semigroup $\{e^{t\Delta_j}\}$ on $L^2(\mathbb{R}^+,A(x)dx)$\, (Proposition \ref{prop:SLappendix_semigr_heatkern}), and the sum converges absolutely. Hence for $\lambda \in \supp(\bm{\rho_k})$ and $\omega$-a.e.\ $\bm{\xi_1} \in M$ we can write
\begin{align*}
& \int_M w_{k,\lambda}(\bm{\xi_2}) \, p(t,\bm{\xi_1},\bm{\xi_2}) \, \omega(d\bm{\xi_2}) \\
& \qquad = \int_M e^{i2k\pi\theta_2} v_{k,\lambda}(x_2) \sum_{j \in \mathbb{Z}} e^{i2j\pi(\theta_1-\theta_2)} p_{\Delta_j\!}(t,x_1,x_2) \, A(x_2) dx_2 d\theta_2 \\
& \qquad = \int_0^\infty v_{k,\lambda}(x_2) \sum_{j \in \mathbb{Z}} e^{i2j\pi\theta_1} \! \int_0^1 e^{i2(k-j)\pi\theta_2} d\theta_2 \, p_{\Delta_j\!}(t,x_1,x_2) A(x_2) dx_2 \\
& \qquad = e^{i2k\pi\theta_1} \int_0^\infty v_{k,\lambda}(x_2) \, p_{\Delta_k\!}(t,x_1,x_2) A(x_2) dx_2 \\
& \qquad = e^{i2k\pi\theta_1} \zeta_k(x_1) \int_0^\infty \widetilde{v}_{k,\lambda}(x_2) \int_{\mathbb{R}_0^+} e^{-t\lambda_0\,} \widetilde{v}_{k,\lambda_0}(x_1) \, \widetilde{v}_{k,\lambda_0}(x_2) \bm{\rho_k}(d\lambda_0) \, B_k(x_2) dx_2 \\
& \qquad = e^{i2k\pi\theta_1} e^{-t\lambda} v_{k,\lambda}(x_1) \\
& \qquad =  e^{-t\lambda} w_{k,\lambda}(\bm{\xi_1}).
\end{align*}
The second to last equality follows from the eigenfunction expansion of the Sturm-Liouville operator $\ell_k$ defined in Lemma \ref{lem:coneconv_sol_normaliz}, considering that the double integral can be recognized as $\mathcal{F}_{\ell_k}[\mathcal{F}_{\ell_k}^{-1} e^{-t\bm{\cdot}\,} \widetilde{v}_{k,\bm{\cdot}}(x_1)](\lambda)$, where $(\mathcal{F}_{\ell_k} g)(\lambda) :=  \int_{\mathbb{R}^+} g(y) \, \widetilde{v}_{k,\lambda}(y) \, B_k(y) dy \equiv (\mathcal{F}_{\Delta_k}(\zeta_k \ccdot g))(\lambda)$. It should also be noted that
\[
e^{-t\lambda} \in L^2(\bm{\rho_k}), \qquad \mathcal{F}_{\ell_k}^{-1} e^{-t\bm{\cdot}} \in L^1(\mathbb{R}^+,\, B_k(x)dx), \qquad \widetilde{v}_{k,\lambda} \in \mathrm{C}_\mathrm{b}(\mathbb{R}_0^+)
\]
(cf.\ Proposition \ref{prop:SLappendix_semigr_heatkern} and Lemma \ref{lem:SLappendix_solprops}(a)), and therefore the second to last equality, which holds initially for $\bm{\rho_k}$-a.e.\ $\lambda \geq 0$, can be extended by continuity to all $\lambda \in \supp(\bm{\rho_k})$.
\end{proof}

\begin{remark} \label{rmk:coneconv_heatkern_existsuffcond}
The two results above depend on the assumption that the heat kernel exists. In the general framework of metric measure spaces, the existence of the heat kernel for the strongly continuous Markovian semigroup $\{e^{t \mb{L}}\}_{t \geq 0}$ determined by a given sesquilinear form is equivalent to the ultracontractivity property $\|e^{t \mb{L}} h\|_{\infty} \leq \gamma(t)\|h\|_{L^1}$, where $\gamma$ is a positive left-continuous function on $\mathbb{R}^+$ (see \cite[Theorem 3.1]{barlowetal2009}). Besides this, there is an extensive body of work on geometric conditions which ensure the existence of the heat kernel. For instance, it is well-known that a heat kernel exists if the generator $\mb{L}$ is the Dirichlet or Neumann Laplacian on a smooth bounded Euclidean domain \cite{davies1989,basshsu1991}. We refer to \cite{grigoryantelcs2012} and references therein for a discussion of this problem in the context of metric measure spaces.

Here we are interested in the case of the Neumann Laplacian on $(M,g)$, $\mb{L} = (\Delta_N,\mathcal{D}_N)$. We will not discuss the general case introduced in \eqref{eq:coneconv_Riemmetric}--\eqref{eq:coneconv_Ahyp}, but we note that the heat kernel exists provided that the metric $g$ is smooth and nondegenerate:

\emph{Let $\mathring{M} = \mathbb{R}^+ \times \mathbb{T}$. If $A$ belongs to $\mathrm{C}^\infty(\mathbb{R}_0^+)$ and $A(0) > 0$, then there exists a heat kernel $p(t,x,y) \in \mathrm{C}^\infty(\mathbb{R}^+ \times \mathring{M} \times \mathring{M})$ satisfying the assumptions of Proposition \ref{prop:coneconv_heatkern_spectrep}.}

Indeed, the stated assumption on the function $A$ ensures that we can regard $(M,g)$ as a submanifold of the complete smooth Riemannian manifold $(\mathbb{R} \times \mathbb{T},\, \widetilde{g})$, where $\widetilde{g} = dx^2 + \widetilde{A}(x)^2 d\theta^2$ and $\widetilde{A} \in \mathrm{C}^\infty(\mathbb{R})$ is a positive extension of the function $A$. Therefore, the above claim is a consequence of the fact that the Laplace-Beltrami operator with Neumann boundary conditions on a domain of a complete Riemannian manifold admits a heat kernel \cite[Theorem 1.1]{choullietal2015}.
\end{remark}

\section{Product formulas and convolutions} \label{sec:prodformconv}

In this section, we establish the product formula for the eigenfunctions $w_{k,\lambda}$, define the associated family of convolutions, and describe the main continuity and mapping properties of the convolution structure.
Following the results in Section \ref{sec:eigenexp},  the contents of this section rely heavily on properties of the Sturm-Liouville operators $\ell_k$, studied in previous works  \cite{sousaetal2020,sousaetal2019b}.
For convenience, the main results used in the current paper are presented in appendix and proofs in this section refer to them when necessary.

We start with the following product formula.

\begin{proposition}[Product formula for $w_{k,\lambda}$] \label{prop:coneconv_prodform}
For each $k \in \mathbb{N}_0$ and $\bm{\xi_1}, \bm{\xi_2} \in M$ there exists a positive measure $\bm{\gamma}_{k,\bm{\xi_1}, \bm{\xi_2}}$ on $M$ such that the product $w_{k,\lambda}(\bm{\xi_1}) \, w_{k,\lambda}(\bm{\xi_2})$ admits the integral representation
\begin{equation} \label{eq:coneconv_prodform}
w_{k,\lambda}(\bm{\xi_1}) \, w_{k,\lambda}(\bm{\xi_2}) = \int_M w_{k,\lambda}(\bm{\xi_3})\, \bm{\gamma}_{k,\bm{\xi_1},\bm{\xi_2}}(d\bm{\xi_3}), \qquad \bm{\xi_1},\bm{\xi_2} \in M, \; \lambda \in \mathbb{C}.
\end{equation}
\end{proposition}

Since $w_{-k,\lambda} = \overline{w_{k,\lambda}}$, this result trivially extends to all $k \in \mathbb{Z}$.

\begin{proof}
Fix $k \in \mathbb{N}_0$. Recall that $w_{k,\lambda}(x,\theta) = e^{i2k\pi\theta} \zeta_k(x) \, \widetilde{v}_{k,\lambda}(x)$, where $\widetilde{v}_{k,\lambda}$ is a solution of \eqref{eq:coneconv_wsol_pf2}. We saw in Lemma \ref{lem:coneconv_sol_normaliz} that the operator $\ell_k$ satisfies assumption \eqref{eq:SLappendix_mainassump}, hence we can apply the existence theorem for Sturm-Liouville type product formulas (Theorem \ref{thm:SLappendix_prodform}) and conclude that there exists a family of measures $\{\pi_{x_1,x_2}^{[k]}\}_{x_1,x_2 \in \mathbb{R}_0^+} \subset \mathcal{P}(\mathbb{R}_0^+)$ with $\supp(\pi_{x_1,x_2}^{[k]}) \subset [|x_1-x_2|,x_1+x_2]$ and such that
\[
\widetilde{v}_{k,\lambda}(x_1) \, \widetilde{v}_{k,\lambda}(x_2) = \int_{\mathbb{R}_0^+} \widetilde{v}_{k,\lambda} \, d\pi_{x_1,x_2}^{[k]} \qquad (x_1, x_2 \in \mathbb{R}_0^+, \; \lambda \in \mathbb{C}).
\]
Consequently, the product formula \eqref{eq:coneconv_prodform} holds for the positive measures $\bm{\gamma}_{k,\bm{\xi_1}, \bm{\xi_2}}$ defined by
\begin{equation} \label{eq:coneconv_gammaconvmeas_def}
\bm{\gamma}_{k,\bm{\xi_1}, \bm{\xi_2}}(d\bm{\xi_3}) =  {\zeta_k(x_1) \zeta_k(x_2) \over \zeta_k(x_3)} \bm{\nu}_{k,\bm{\xi_1}, \bm{\xi_2}}(d\bm{\xi_3})
\end{equation}
where $\bm{\nu}_{k,\bm{\xi_1}, \bm{\xi_2}} := \pi_{x_1,x_2}^{[k]} \otimes \delta_{\theta_1 + \theta_2}$.
\end{proof}

It follows at once from the definition that the measures $\bm{\nu}_{k,\bm{\xi_1}, \bm{\xi_2}}$ are probability measures on $M$. Let us introduce the convolution operators determined by these measures:

\begin{definition} \label{def:coneconv_convk}
Let $k \in \mathbb{Z}$ and $\mu, \nu \in \mathcal{M}_{\mathbb{C}}(M)$. The measure
\[
(\mu \conv{k} \nu)(\bm{\cdot}) = \int_M \int_M \bm{\nu}_{k,\bm{\xi_1}, \bm{\xi_2}}(\bm{\cdot}) \, \mu(d\bm{\xi_1}) \, \nu(d\bm{\xi_2})
\]
is called the \emph{$\Delta_k$-convolution} of the measures $\mu$ and $\nu$.
\end{definition}

It is easy to check that the $\Delta_k$-convolution has the following property:

\begin{proposition}
The space $(\mathcal{M}_\mathbb{C}(M),\conv{k})$, equipped with the total variation norm $\|\mu\| = |\mu|(M)$, is a commutative Banach algebra over $\mathbb{C}$ whose identity element is the Dirac measure $\delta_{(0,0)}$. Moreover, the subset $\mathcal{P}(M)$ is closed under the $\Delta_k$-convolution.
\end{proposition}

Another fundamental property of the $\Delta_k$-convolution is its connection with the $\Delta$-Fourier transform defined as follows:

\begin{definition}
Let $\mu \in \mathcal{M}_{\mathbb{C}}(M)$. The \emph{$\Delta$-Fourier transform} of the measure $\mu$ is the function defined by the integral
\[
(\bm{\mathcal{F}}\mu)(k,\lambda) = \int_M {w_{-k,\lambda}(\bm{\xi}) \over \zeta_k(x)} \, \mu(d\bm{\xi}), \qquad k \in \mathbb{Z}, \; \lambda \geq 0.
\]
\end{definition}

It follows from Lemma \ref{lem:SLappendix_solprops}(a) that $\|{w_{-k,\lambda} \over \zeta_k}\|_\infty \leq 1$, hence $(\bm{\mathcal{F}}\mu)(k,\lambda)$ is well-defined for all $\mu \in \mathcal{M}_{\mathbb{C}}(M)$ and $(k, \lambda) \in \mathbb{Z} \times \mathbb{R}_0^+$.

\begin{proposition}
Let $\mu, \nu \in \mathcal{M}_{\mathbb{C}}(M)$. We have 
\begin{equation} \label{eq:coneconv_transf_convtrivial}
\bigl(\bm{\mathcal{F}}(\mu \conv{k} \nu)\bigr)(k,\lambda) = (\bm{\mathcal{F}}\mu)(k,\lambda) \ccdot (\bm{\mathcal{F}}\nu)(k,\lambda) \qquad \text{for all } k \in \mathbb{Z} \text{ and } \lambda \geq 0.
\end{equation}
Moreover, for fixed $k \in \mathbb{Z}$ we have
\[
(\bm{\mathcal{F}}\alpha)(k,\bm{\cdot}) = (\bm{\mathcal{F}}\mu)(k,\bm{\cdot}) \ccdot (\bm{\mathcal{F}}\nu)(k,\bm{\cdot}) \quad \text{ if and only if } \quad \widehat{\alpha}_k = \widehat{\mu}_k \convdiam{k} \widehat{\nu}_k
\]
where $\widehat{\tau}_k$ ($\tau = \alpha, \mu, \nu$) is the complex measure on $\mathbb{R}_0^+$ defined by
\[
\widehat{\tau}_k(J) = \int_M e^{-i2k\pi\theta} \, \mathds{1}_J(x) \, \tau(d\bm{\xi})
\]
and $\widehat{\mu}_k \convdiam{k} \widehat{\nu}_k(\bm{\cdot}) := \int_{\mathbb{R}_0^+} \int_{\mathbb{R}_0^+} \pi_{x_1,x_2}^{[k]}(\bm{\cdot}) \, \widehat{\mu}_k(dx_1) \, \widehat{\nu}_k(dx_2)$ (here $\pi_{x_1,x_2}^{[k]} \in \mathcal{P}(\mathbb{R}_0^+)$ are the measures from the proof of Proposition \ref{prop:coneconv_prodform}).
\end{proposition}

\begin{proof}
Applying the product formula \eqref{eq:coneconv_prodform}, we obtain
\begin{align*}
\bigl(\bm{\mathcal{F}}(\mu \conv{k} \nu)\bigr)(k,\lambda) & = \int_M {w_{-k,\lambda}(\bm{\xi}) \over \zeta_k(x)} \, (\mu \conv{k} \nu)(d\bm{\xi}) \\
& = \int_M \int_M \int_M {w_{-k,\lambda}(\bm{\xi_3}) \over \zeta_k(x_3)} \, (\delta_{\bm{\xi_1}} \conv{k} \delta_{\bm{\xi_2}})(d\bm{\xi_3}) \, \mu(d\bm{\xi_1}) \nu(d\bm{\xi_2})\\
& = \int_M \int_M {w_{-k,\lambda}(\bm{\xi_1}) \over \zeta_k(x_1)} {w_{-k,\lambda}(\bm{\xi_2}) \over \zeta_k(x_2)} \, \mu(d\bm{\xi_1}) \nu(d\bm{\xi_2}) \: = \: (\bm{\mathcal{F}}\mu)(k,\lambda) \ccdot (\bm{\mathcal{F}}\nu)(k,\lambda)
\end{align*}
so that \eqref{eq:coneconv_transf_convtrivial} holds. Since $\convdiam{k}$ is the convolution on $\mathbb{R}_0^+$ associated with the Sturm-Liouville operator $\ell_k = -{1 \over B_k}{d \over dx}(B_k {d \over dx})$, the second statement is a consequence of the corresponding property of one-dimensional generalized convolutions (Proposition \ref{prop:SLappendix_convtransl_props}(a)).
\end{proof}

Next we summarize some useful properties of the $\Delta$-Fourier transform.

\begin{proposition} \label{prop:coneconv_fourmeas_props}
The $\Delta$-Fourier transform $\bm{\mathcal{F}}\mu$ of $\mu \in \mathcal{M}_\mathbb{C}(M)$ has the following properties:
\begin{enumerate}[itemsep=0pt,topsep=4pt]
\item[\textbf{(i)}] For each $k \in \mathbb{Z}$, $(\bm{\mathcal{F}}\mu)(k,\bm{\cdot})$ is continuous on $\mathbb{R}_0^+$. Moreover, if a family of measures $\{\mu_j\} \subset \mathcal{M}_\mathbb{C}(M)$ is tight and uniformly bounded, then $\{(\bm{\mathcal{F}}\mu_j) (k,\bm{\cdot})\}$ is equicontinuous on $\mathbb{R}_0^+$.

\item[\textbf{(ii)}] Each measure $\mu \in \mathcal{M}_\mathbb{C}(M)$ is uniquely determined by $\bm{\mathcal{F}}\mu$.

\item[\textbf{(iii)}]
If $\{\mu_n\}$ is a sequence of measures belonging to $\mathcal{M}_+(M)$, $\mu \in \mathcal{M}_+(M)$, and $\mu_n \warrow \mu$, then for each $k \in \mathbb{Z}$ we have
\[
(\bm{\mathcal{F}}\mu_n)(k,\bm{\cdot}) \xrightarrow[\,n \to \infty\,]{} (\bm{\mathcal{F}}\mu)(k,\bm{\cdot}) \qquad \text{uniformly on compact sets.}
\]

\item[\textbf{(iv)}] Suppose that $\lim_{x \to \infty} A(x) = \infty$. If $\{\mu_n\}$ is a sequence of measures belonging to $\mathcal{M}_+(M)$ whose $\Delta$-Fourier transforms are such that
\[
(\bm{\mathcal{F}}\mu_n)(k,\lambda) \xrightarrow[\,n \to \infty\,]{} f(k,\lambda) \qquad \text{pointwise in } (k,\lambda) \in \mathbb{Z} \times \mathbb{R}_0^+
\]
for some real-valued function $f$ such that $f(0,\bm{\cdot})$ is continuous at a neighbourhood of zero, then $\mu_n \warrow \mu$ for some measure $\mu \in \mathcal{M}_+(M)$ such that $\bm{\mathcal{F}}\mu \equiv f$.
\end{enumerate}
\end{proposition}

\begin{proof}
\textbf{\emph{(i)}} We have
\[
(\bm{\mathcal{F}}\mu)(k,\lambda) = \int_{\mathbb{R}_0^+} \widetilde{v}_{k,\lambda}(x)\, \widehat{\mu}_k(dx) \equiv (\mathcal{F}_{\ell_k}\,\widehat{\mu}_k) (\lambda)
\]
where 
\[
(\mathcal{F}_{\ell_k} \nu)(\lambda) = \int_{\mathbb{R}_0^+} \widetilde{v}_{k,\lambda}(x) \, \nu(dx), \qquad \nu \in \mathcal{P}(\mathbb{R}_0^+)
\]
is the generalized Fourier transform of measures determined by the Sturm-Liouville operator $\ell_k$, as defined in \eqref{eq:SLappendix_fourmeas_def}. Therefore, the result follows from the corresponding property of one-dimensional convolutions (Proposition \ref{prop:SLappendix_fourmeas_props}(i)). \\[-8pt]

\textbf{\emph{(ii)}} Let $\mu \in \mathcal{M}_\mathbb{C}(M)$ be such that $(\bm{\mathcal{F}}\mu)(k,\lambda) = 0$ for all $k \in \mathbb{Z}$ and $\lambda \geq 0$. Let $f \in \mathrm{C}_\mathrm{c}(\mathbb{R}_0^+)$ and $g \in \mathrm{C}^1(\mathbb{T})$. Recalling that the Fourier series $g(\theta) = \sum_{k \in \mathbb{Z}} \langle g, e^{-i2k\pi\bm{\cdot}} \rangle \, e^{i2k\pi\theta}$ converges absolutely and uniformly \cite[Theorem 1.4.2]{dymmckean1972}, we get
\begin{align*}
\int_M f(x) \, g(\theta) \, \mu(d(x,\theta)) & = \int_M f(x) \sum_{k \in \mathbb{Z}} \langle g, e^{-i2k\pi\bm{\cdot}} \rangle \, e^{i2k\pi\theta} \, \mu(d(x,\theta)) \\
& = \sum_{k \in \mathbb{Z}} \langle g, e^{-i2k\pi\bm{\cdot}} \rangle \int_{\mathbb{R}_0^+} f(x) \,\widehat{\mu}_{-k}(dx) \\
& = 0
\end{align*}
where the last equality holds because, by Proposition \ref{prop:SLappendix_fourmeas_props}(ii), $(\bm{\mathcal{F}}\mu)(k,\bm{\cdot}) \equiv 0$ implies that $\widehat{\mu}_k = 0$. By the Stone-Weierstrass theorem (see \cite[Section 38]{simmons1963} and also \cite[Corollary 15.3]{klenke2014}), this implies that $\mu$ is the zero measure. \\[-8pt]

\textbf{\emph{(iii)}} This follows directly from Proposition \ref{prop:SLappendix_fourmeas_props}(iii). \\[-8pt]

\textbf{\emph{(iv)}} Let us show that $\{\mu_n\}$ is tight. Fix $\eps > 0$. Given that $f(0,\bm{\cdot})$ is continuous near zero, we can choose $\delta > 0$ such that
\[
\biggl| {1 \over \delta} \int_0^{2\delta} \bigl(f(0,0) - f(0,\lambda)\bigr)d\lambda \biggr| < \eps.
\]
Furthermore, according to Lemma \ref{lem:SLappendix_solprops}(b) we have $\lim_{x \to \infty} \widetilde{v}_{0,\lambda}(x) = 0$ for all $\lambda > 0$; we can therefore pick $0 < \beta < \infty$ such that
\[
\int_0^{2\delta} \bigl( 1-w_{0,\lambda}(\bm{\xi}) \bigr) d\lambda \geq \delta \qquad \text{for all } (x, \theta) \in (\beta, \infty) \times \mathbb{T}.
\]
We now compute
\begin{align*}
\mu_n\bigl([\beta,\infty) \times \mathbb{T}) & \leq {1 \over \delta} \int_{[\beta,\infty) \times \mathbb{T}} \int_0^{2\delta} \bigl( 1-w_{0,\lambda}(\bm{\xi}) \bigr) d\lambda\, \mu_n(d\bm{\xi}) \\
& \leq {1 \over \delta} \int_{[a,\infty) \times \mathbb{T}} \int_0^{2\delta} \bigl( 1-w_{0,\lambda}(\bm{\xi}) \bigr) d\lambda\, \mu_n(d\bm{\xi}) \\
& = {1 \over \delta} \int_0^{2\delta} \bigl((\bm{\mathcal{F}}\mu_n)(0,0) - (\bm{\mathcal{F}}\mu_n)(0,\lambda)\bigr) d\lambda
\end{align*}
so that, using dominated convergence, we obtain
\[
\limsup_{n \to \infty} \mu_n([\beta,\infty)\times \mathbb{T}) \leq {1 \over \delta} \limsup_{n \to \infty}\! \int_0^{2\delta} \bigl((\bm{\mathcal{F}}\mu_n)(0,0) - (\bm{\mathcal{F}}\mu_n)(0,\lambda)\bigr) d\lambda = {1 \over \delta} \int_0^{2\delta} \bigl(f(0,0) - f(0,\lambda)\bigr) d\lambda < \eps
\]
where $\eps$ is arbitrary, showing that $\{\mu_n\}$ is tight. 

Since $\{\mu_n\}$ is also uniformly bounded, Prohorov's theorem ensures that given a subsequence $\{\mu_{n_k}\}$, there exists a further subsequence $\{\mu_{n_{k_j}}\}$ and a measure $\mu \in \mathcal{M}_+(M)$ for which we have $\mu_{n_{k_j}} \warrow \mu$. By part (iii) we must have $(\bm{\mathcal{F}}\mu)(k,\lambda) = f(k,\lambda)$ for all $(k,\lambda) \in \mathbb{Z} \times \mathbb{R}_0^+$. It then follows from (ii) that all such subsequences (and hence the sequence $\{\mu_n\}$ itself) converge weakly to a unique measure $\mu$.
\end{proof}

In the sequel we will always assume that $\lim_{x \to \infty} A(x) = \infty$.

\begin{corollary} \label{cor:coneconv_weakcont}
For each $k \in \mathbb{Z}$, the mapping $(\mu,\nu) \mapsto \mu \conv{k} \nu$ is continuous in the weak topology.
\end{corollary}

\begin{proof}
We have
\begin{equation} \label{eq:coneconv_weakcont_pf1}
\bigl(\bm{\mathcal{F}}(\delta_{\bm{\xi_1}} \conv{k} \delta_{\bm{\xi_2}})\bigr) (j,\lambda) = \int_M e^{-i2j\pi\theta_3} {v_{j,\lambda}(x_3) \over \zeta_j(x_3)} (\delta_{\bm{\xi_1}} \conv{k} \delta_{\bm{\xi_2}})(d\bm{\xi_3}) = e^{-i2j\pi(\theta_1+\theta_2)}\! \int_{\mathbb{R}_0^+} {v_{j,\lambda}(x_3) \over \zeta_j(x_3)} (\delta_{x_1} \convdiam{k} \delta_{x_2})(dx_3).
\end{equation}
Proposition \ref{prop:SLappendix_convtransl_props}(b) ensures that $(x_1,x_2) \mapsto \delta_{x_1} \convdiam{k} \delta_{x_2}$ is continuous in the weak topology, hence the expression in the right hand side is a continuous function of $(\bm{\xi_1},\bm{\xi_2})$. It then follows from Proposition \ref{prop:coneconv_fourmeas_props}(iv) that $(\bm{\xi_1},\bm{\xi_2}) \mapsto \delta_{\bm{\xi_1}} \conv{k} \delta_{\bm{\xi_2}}$ is continuous in the weak topology.

Let $h \in \mathrm{C}_\mathrm{b}(M)$ and $\mu_n, \nu_n \in \mathcal{M}_\mathbb{C}(M)$ with $\mu_n \warrow \mu$ and $\nu_n \warrow \nu$. We have just seen that $\int_M h(x_3)\, (\delta_{\bm{\xi_1}} \conv{k} \delta_{\bm{\xi_2}})(d\bm{\xi_3})$ is a continuous function of $(\bm{\xi_1},\bm{\xi_2})$; consequently,
\begin{align*}
\lim_n \int_M \int_M \biggl( \int_M h(\bm{\xi_3})\, (\delta_{\bm{\xi_1}} \conv{k} \delta_{\bm{\xi_2}})(d\bm{\xi_3}) \biggr) \mu_n(dx) \nu_n(dy) = \int_M \int_M \biggl( \int_M h(\bm{\xi_3})\, (\delta_{\bm{\xi_1}} \conv{k} \delta_{\bm{\xi_2}})(d\bm{\xi_3}) \biggr) \mu(dx) \nu(dy)
\end{align*}
which means (since $h$ is arbitrary) that $\mu_n \conv{k} \nu_n \to \mu \conv{k} \nu$.
\end{proof}

The operator $\bm{\mathcal{T}}_k^\mu$ defined by the integral 
\[
(\bm{\mathcal{T}}_k^\mu h)(\bm{\xi}) := \int_M h \, d(\delta_{\bm{\xi}} \conv{k} \mu)
\]
is said to be the \emph{$\Delta_k$-translation} by the measure $\mu \in \mathcal{M}_{\mathbb{C}}(M)$. The next result summarizes its mapping properties. For brevity, we write $L_k^p := L^p(M, B_k(x) dx d\theta)$.

\begin{proposition} \label{prop:coneconv_transl_props} \,
\begin{enumerate}[itemsep=0pt,topsep=1pt]
\item[\textbf{(a)}] If $h \in \mathrm{C}_\mathrm{b}(M)$, then $\bm{\mathcal{T}}_k^\mu h \in \mathrm{C}_\mathrm{b}(M)$ for all $\mu  \in \mathcal{M}_\mathbb{C}(M)$.

\item[\textbf{(b)}] If $h \in \mathrm{C}_0(M)$, then $\bm{\mathcal{T}}_k^\mu h \in \mathrm{C}_0(M)$ for all $\mu \in \mathcal{M}_\mathbb{C}(M)$.

\item[\textbf{(c)}] Let $1 \leq p \leq \infty$, $\mu \in \mathcal{M}_+(M)$ and $h \in L_k^p$. The $\Delta_k$-translation $(\bm{\mathcal{T}}_k^\mu h)(x)$ is a Borel measurable function of $x \in M$, and we have
\begin{equation} \label{eq:coneconv_gentransl_Lpcont}
\|\bm{\mathcal{T}}_k^\mu h\|_{L_k^p} \leq \|\mu\| \ccdot \|h\|_{L_k^p}.
\end{equation}
\item[\textbf{(d)}] Let $p_1,p_2 \in [1, \infty]$ such that ${1 \over p_1} + {1 \over p_2} \geq 1$, and write $\bm{\mathcal{T}}_k^{\bm{\xi}} := \bm{\mathcal{T}}_k^{\delta_{\bm{\xi}}}$\, ($\bm{\xi} \in M\mskip-0.8\thinmuskip$). For $h \in L_k^{p_1}$ and $g \in L_k^{p_2}$, the $\Delta_k$-convolution 
\[
(h \conv{k} g)(\bm{\xi}) = \int_M (\bm{\mathcal{T}}_k^{\bm{\xi_1}} h)(\bm{\xi})\, g(\bm{\xi_1})\, B_k(x_1) dx_1 d\theta_1
\]
is well-defined and, for $s \in [1, \infty]$ defined by ${1 \over s} = {1 \over p_1} + {1 \over p_2} - 1$, it satisfies
\[
\| h \conv{k} g \|_{L_k^s} \leq \| h \|_{L_k^{p_1}} \| g \|_{L_k^{p_2}}
\]
(in particular, $h \conv{k} g \in L_k^s$). 
\end{enumerate}
\end{proposition}

\begin{proof}
\textbf{\emph{(a)}} This is an immediate consequence of Corollary \ref{cor:coneconv_weakcont}. \\[-8pt]

\textbf{\emph{(b)}} It follows from \eqref{eq:coneconv_weakcont_pf1} that for all $j \in \mathbb{Z}$ and $\lambda > 0$ we have
\begin{equation} \label{eq:coneconv_vagueconvg_pf}
\bigl(\bm{\mathcal{F}}(\delta_{\bm{\xi}} \conv{k} \mu)\bigr)(j,\lambda) = \int_M e^{-i2j\pi(\theta + \theta_1)}\! \int_{\mathbb{R}_0^+} {v_{j,\lambda}(x_3) \over \zeta_j(x_3)} (\delta_x \convdiam{k} \delta_{x_1})(dx_3) \, \mu(d\bm{\xi_1}) \longrightarrow 0 \qquad \text{as } x \to \infty
\end{equation}
where the last step follows from dominated convergence and the fact that $\delta_x \convdiam{k} \delta_{x_1} \varrow \bm{0}$ as $x \to \infty$, where $\bm{0}$ denotes the zero measure (Proposition \ref{prop:SLappendix_convtransl_props}(d)). It follows from \eqref{eq:coneconv_vagueconvg_pf} and similar reasoning as in \cite[Remark 4.7]{sousaetal2020} that $\delta_{\bm{\xi}} \conv{k} \mu \varrow \bm{0}$ as $x \to \infty$, so that (b) holds. \\[-8pt]

\textbf{\emph{(c)}} In the  case $p = \infty$, the proof is straightforward. Let $1 \leq p < \infty$. Suppose first that $h(x,\theta) = f(x) g(\theta)$ and observe that
\[
(\bm{\mathcal{T}}_k^{(x_2,\theta_2)} h)(x_1,\theta_1) = (\mathcal{T}_{\ell_k}^{x_2} f)(x_1) \ccdot (\mathcal{T}_\mathbb{T}^{\theta_2} g)(\theta_1)
\]
where $\mathcal{T}_{\ell_k}^x$ is the generalized translation associated with the Sturm-Liouville operator $\ell_k$ and $(\mathcal{T}_\mathbb{T}^{\theta_2} g)(\theta_1) := g(\theta_1 + \theta_2)$ is the ordinary translation on the torus. We have $\|\mathcal{T}_{\ell_k}^x f\|_{L^p(\mathbb{R}^+,\, B_k(x)dx)} \leq \|f\|_{L^p(\mathbb{R}^+,\, B_k(x)dx)}$ (cf.\ Proposition \ref{prop:SLappendix_convtransl_props}(e)), and therefore
\[
\| \bm{\mathcal{T}}_k^{(x_2,\theta_2)} h \|_{L_k^p} = \|\mathcal{T}_{\ell_k}^x f\|_{L^p(\mathbb{R}^+,\, B_k(x)dx)} \|\mathcal{T}_\mathbb{T}^{\theta_2} g\|_{L^p(\mathbb{T})} \leq \|f\|_{L^p(\mathbb{R}^+,\, B_k(x)dx)} \|g\|_{L^p(\mathbb{T})} = \|h\|_{L_k^p}.
\]
Since the linear span of indicator functions of compact rectangles $I \times J \subset \mathbb{R}_0^+ \times \mathbb{T}$ is dense in $L_k^p$, we have $\|\bm{\mathcal{T}}_k^{(x,\theta)} h\|_{L_k^p} \leq \|h\|_{L_k^p}$ for all $h \in L_k^p$ and $(x,\theta) \in M$, showing that \eqref{eq:coneconv_gentransl_Lpcont} holds for Dirac measures $\mu = \delta_{(x,\theta)}$. The result can be extended to all $\mu \in \mathcal{M}_+(M)$ (and $1 \leq p < \infty$) by using Minkowski's integral inequality.
\\[-8pt]

\textbf{\emph{(d)}} The proof relies on part (c) and the same reasoning as in the classical case; see e.g.\ the proof of Proposition 1.III.5 of \cite{trimeche1997}.
\end{proof}

In the next statement we show that if a heat kernel exists for the heat semigroup $\{e^{t\Delta_N}\}$, then the functions $e^{t\Delta_N} w_{k,\lambda} = e^{-t\lambda} w_{k,\lambda}$ also admit a product formula whose measures do not depend on the spectral parameter $\lambda$ and, moreover, are absolutely continuous with respect to $\omega$.

\begin{proposition} \label{prop:coneconv_prodform_extended}
Assume that the action of $e^{t\Delta_N}$ on $L^2(M)$ is given by a symmetric heat kernel satisfying conditions I and II of Proposition \ref{prop:coneconv_heatkern_spectrep}. Let $\bm{\gamma}_{t,k,\bm{\xi_1},\bm{\xi_2}}$ be the positive measure defined by
\[
\bm{\gamma}_{t,k,\bm{\xi_1},\bm{\xi_2}}(d\bm{\xi_3}) = \int_M \bm{\gamma}_{k,\bm{\xi_4},\bm{\xi}_2}(d\bm{\xi_3}) \, p(t,\bm{\xi_1},\bm{\xi_4}) \omega(d\bm{\xi_4}).
\]
Then, the product $e^{-t\lambda} w_{k,\lambda}(\bm{\xi_1}) \, w_{k,\lambda}(\bm{\xi_2})$ admits the integral representation
\[
e^{-t\lambda} w_{k,\lambda}(\bm{\xi_1}) \, w_{k,\lambda}(\bm{\xi_2}) = \int_M w_{k,\lambda}(\bm{\xi_3})\, \bm{\gamma}_{t,k,\bm{\xi_1},\bm{\xi_2}}(d\bm{\xi_3}) \qquad (t \geq 0, \; \bm{\xi_1},\bm{\xi_2} \in M, \; \lambda \in \supp(\bm{\rho_k})).
\]
\end{proposition}

\begin{proof}
By direct calculation we get
\begin{align*}
\int_M w_{k,\lambda}(\bm{\xi_3})\, \bm{\gamma}_{t,k,\bm{\xi_1},\bm{\xi_2}}(d\bm{\xi_3}) & = \int_M \int_M w_{k,\lambda}(\bm{\xi_3}) \bm{\gamma}_{k,\bm{\xi_4},\bm{\xi}_2}(d\bm{\xi_3}) \, p(t,\bm{\xi_1},\bm{\xi_4}) \omega(d\bm{\xi_4}) \\
& = e^{-t\lambda} w_{k,\lambda}(\bm{\xi_1}) \, w_{k,\lambda}(\bm{\xi_2})
\end{align*}
where, by Proposition \ref{prop:coneconv_prodform} and Corollary \ref{cor:coneconv_wsol_eigeneq}, the last equality holds for $t \geq 0$, $\lambda \in \supp(\bm{\rho_k})$, $\bm{\xi_2} \in M$ and $\omega$-a.e.\ $\bm{\xi_1} \in M$. Using the symmetry relation $\int_M w_{k,\lambda}(\bm{\xi_3})\, \bm{\gamma}_{t,k,\bm{\xi_1},\bm{\xi_2}}(d\bm{\xi_3}) = \int_M w_{k,\lambda}(\bm{\xi_3})\, \bm{\gamma}_{t,k,\bm{\xi_2},\bm{\xi_1}}(d\bm{\xi_3})$, the identity extends by continuity to all $\bm{\xi_1}, \bm{\xi_2} \in M$. (The given symmetry can be deduced by noting that, by Propositions \ref{prop:coneconv_Feigenexp}--\ref{prop:coneconv_heatkern_spectrep} and Proposition \ref{prop:SLappendix_convtransl_props}(c), we have for $g \in \mathrm{C}_\mathrm{c}^2(\mathbb{R}^+)$
\begin{align*}
\int_M e^{i2k\pi\theta_3} g(x_3) \bm{\gamma}_{t,k,\bm{\xi_1},\bm{\xi_2}}(d\bm{\xi_3}) & = \int_M \int_{\mathbb{R}_0^+} w_{k,\lambda}(\bm{\xi_4}) \, w_{k,\lambda}(\bm{\xi_2}) \, (\mathcal{F}_{\Delta_k} g)(\lambda) \bm{\rho_k}(d\lambda) \, p(t,\bm{\xi_1}, \bm{\xi_4}) \omega(d\bm{\xi_4}) \\
& = \int_{\mathbb{R}_0^+} e^{-t\lambda} w_{k,\lambda}(\bm{\xi_1}) \, w_{k,\lambda}(\bm{\xi_2}) \, (\mathcal{F}_{\Delta_k} g)(\lambda) \bm{\rho_k}(d\lambda) \end{align*}
and, therefore, $(\widehat{\rule{0pt}{0.5\baselineskip}\bm{\gamma}_{t,k,\bm{\xi_1},\bm{\xi_2}}})_{-k} = (\widehat{\rule{0pt}{0.5\baselineskip}\bm{\gamma}_{t,k,\bm{\xi_2},\bm{\xi_1}}})_{-k}$.)
\end{proof}

\section{Infinitely divisible measures and convolution semigroups} \label{sec:infdivsemigr}

In this section we develop the basic notions of divisibility of measures with respect to the convolution algebras $(\mathcal{M}_\mathbb{C}(M),\conv{k})$. As in the classical theory, these will be seen to induce a Lévy-Khintchine type representation and to a convolution semigroup representation for the reflected Brownian motion on $(M,g)$. First we present the following basic definitions:

\begin{definition}
\begin{itemize}[itemsep=0pt,topsep=4pt]
\item[]
\item The set $\mathcal{P}_{k,\mathrm{id}}$ of \emph{$\Delta_k$-infinitely divisible measures} is defined by
\begin{equation} \label{eq:coneconv_infdiv_def}
\mathcal{P}_{k,\mathrm{id}} = \bigl\{ \mu \in \mathcal{P}(M) \bigm| \text{for all } n \in \mathbb{N} \text{ there exists } \nu_n \in \mathcal{P}(M) \text{ such that } \mu = (\nu_n)^{\mskip -0.5\thinmuskip *_{\raisebox{-1.5pt}{$\scriptscriptstyle \!k$}} n} \bigr\}
\end{equation}
where $(\nu_n)^{\mskip -0.5\thinmuskip *_{\raisebox{-1.5pt}{$\scriptscriptstyle \!k$}} n}$ denotes the $n$-fold $\Delta_k$-convolution of $\nu_n$ with itself.

\item The \emph{$\Delta_k$-Poisson measure} associated with $\nu \in \mathcal{M}_+(M)$ is
\[
\mb{e}_k(\nu): = e^{-\|\nu\|} \sum_{n=0}^\infty {\nu^{\mskip -0.5\thinmuskip *_{\raisebox{-1.5pt}{$\scriptscriptstyle \!k$}} n} \over n!}
\]
(the infinite sum converging in the weak topology).

\item A measure $\mu \in \mathcal{P}_{k,\mathrm{id}}$ is called a \emph{$\Delta_k$-Gaussian measure} if the measures $\nu_n$ in \eqref{eq:coneconv_infdiv_def} are such that
\[
\lim_{n \to \infty} n \ccdot \nu_n(M \setminus V) = 0 \; \text{ for every open set } V \text{ containing } (0,0).
\]
\end{itemize}
\end{definition}

It is easy to check that, for $\nu \in \mathcal{M}_+(M)$,
\begin{equation} \label{eq:coneconv_poisson_transf}
\int_M e^{-i2j\pi\theta\,} \widetilde{v}_{k,\lambda}(x) \, \mb{e}_k(\nu)(d\bm{\xi}) = \exp\biggl( \int_M \bigl[ e^{-i2j\pi\theta} \, \widetilde{v}_{k,\lambda}(x) - 1 \bigr] \nu(d\bm{\xi}) \biggr), \qquad (j, \lambda) \in \mathbb{Z} \times \mathbb{R}_0^+.
\end{equation}
(This is an equivalent characterization of $\Delta_k$-Poisson measures, because by \cite[Theorem 2.2.4]{bloomheyer1994} each measure $\mu \in \mathcal{M}_\mathbb{C}(M)$ is characterized by the integrals $\int_M e^{-i2j\pi\theta\,} \widetilde{v}_{k,\lambda}(x) \mu(d\bm{\xi})$.) More generally, if the positive measure $\nu$ is (possibly) unbounded and the equality \eqref{eq:coneconv_poisson_transf} holds for some measure $\mb{e}_k(\nu) \in \mathcal{P}(M)$, then we will also say that $\mb{e}_k(\nu)$ is a \emph{$\Delta_k$-Poisson measure} associated with $\nu$.

\begin{definition}
A family $\{\mu_t\}_{t\geq 0} \subset \mathcal{P}(M)$ is called a \emph{$\Delta_k$-convolution semigroup} if it satisfies the conditions
\[
\mu_s \conv{k} \mu_t = \mu_{s+t} \text{ for all } s, t \geq 0, \qquad \mu_0 = \delta_{(0,0)} \qquad \text{ and } \;\; \mu_t \warrow \delta_{(0,0)} \text{ as } t \downarrow 0.
\]
The $\Delta_k$-convolution semigroup $\{\mu_t\}_{t \geq 0}$ is said to be \emph{Gaussian} if $\mu_1$ is a $\Delta_k$-Gaussian measure.
\end{definition}

A measure $\mu \in \mathcal{M}_\mathbb{C}(M)$ is said to be \emph{symmetric} if $\mu(B) = \mu(\check{B})$ for all Borel subsets $B \subset M$, where $\check{B}$ is the image of $B$ under the mapping $(x,\theta) \mapsto (x,1-\theta)$. One can show that for each symmetric measure $\mu \in \mathcal{P}_{k,\mathrm{id}}$ there exists a unique $\Delta_k$-convolution semigroup $\{\mu_t\}_{t \geq 0}$ such that $\mu_1 = \mu$; consequently, there is a one-to-one correspondence between symmetric $\Delta_k$-infinitely divisible measures and symmetric $\Delta_k$-convolution semigroups. (The proof is similar to that of the corresponding result for the ordinary convolution on the torus, see also \cite[Theorem 5.3.4]{bloomheyer1994}.)

It follows from Proposition \ref{prop:SLappendix_hypergroup} that the convolution algebra $(M,\conv{k})$ is a product hypergroup in the sense of \cite[Definition 1.5.29]{bloomheyer1994}. We can therefore use a general result on infinitely divisible measures on commutative hypergroups \cite[Theorems 4.4 and 4.7]{rentzsch1998} to obtain the following Lévy-Khintchine type representation for symmetric $\Delta_k$-infinitely divisible measures (and for the corresponding convolution semigroups):

\begin{proposition}
Any symmetric measure $\mu \in \mathcal{P}_{k,\mathrm{id}}$ can be represented as
\[
\mu = \gamma \conv{k} \mb{e}_k(\nu)
\]
where $\mb{e}_k(\nu)$ is the $\Delta_k$-Poisson measure associated with the $\sigma$-finite positive measure $\nu = \smash{\lim\limits_{t \downarrow 0} ({1 \over t} \mu_t) \restrict{M \setminus (0,0)}}$ and $\gamma$ is a $\Delta_k$-Gaussian measure.

The representation is unique, i.e.\ if $\mu = \widetilde{\gamma} \conv{k} \mb{e}_k(\widetilde{\nu})$ for a $\sigma$-finite positive measure $\widetilde{\nu}$ and a Gaussian measure $\widetilde{\gamma}$, then $\nu = \widetilde{\nu}$ and $\gamma = \widetilde{\gamma}$.
\end{proposition}

It is easy to show (cf.\ \cite[Proposition 2.1]{rentzschvoit2000}) that each $\Delta_k$-convolution semigroup gives rise to a Markovian contraction semigroup of operators:

\begin{proposition}
Let $\{\mu_t\}$ be a $\Delta_k$-convolution semigroup. Then
\[
(T_t h)(\bm{\xi}) := (\bm{\mathcal{T}}_k^{\mu_t}h)(\bm{\xi}) = \int_M h \, d(\delta_{\bm{\xi}} \conv{k} \mu_t)
\]
defines a conservative Feller semigroup on $\mathrm{C}_0(M)$ such that the identity $T_t \bm{\mathcal{T}}_k^\nu f = \bm{\mathcal{T}}_k^\nu T_t f$ holds for all $t \geq 0$ and $\nu \in \mathcal{M}_\mathbb{C}(M)$. The restriction $\bigl\{T_t\restrict{\mathrm{C}_\mathrm{c}(M)}\bigr\}$ can be extended to a strongly continuous contraction semigroup $\{T_t^{(p)}\}$ on the space $L^p(M)$\, ($1 \leq p < \infty$). Moreover, the operators $T_t^{(p)}$ are given by $T_t^{(p)} f = \bm{\mathcal{T}}_k^{\mu_t} f$\, ($f \in L^p(M)$).
\end{proposition}

Next we show that the heat semigroup generated by $\Delta_N$ is of the convolution semigroup type, in the sense that its action can be represented in terms of integrals with respect to Gaussian $\Delta_k$-convolution semigroups:

\begin{proposition} \label{prop:coneconv_rbm_convsemigroup}
For $k \in \mathbb{Z}$, let $\mathrm{m}_0 \in \mathcal{M}_\mathbb{C}(M)$ be an absolutely continuous measure with respect to $\omega$ whose density function $q_{\mathrm{m}_0}$ belongs to $L^2(M) \cap L^1(M, \zeta_k \ccdot \omega)$, and such that $(\widehat{\mathrm{m}_0})_j = 0$ for each $j \neq k$. Then there exists a Gaussian $\Delta_k$-convolution semigroup $\{\mu_t^k\}_{t \geq 0}$ such that
\begin{equation} \label{eq:coneconv_gentransl_BMlevy}
\int_M (e^{t\Delta_N} h)(\bm{\xi})\, \mathrm{m}_0(d\bm{\xi}) = \int_M {h(\bm{\xi}) \over \zeta_k(x)} \, \bigl(\mu_t^k \conv{k} (\zeta_k \ccdot \mathrm{m}_0)\bigr)(d\bm{\xi}) \qquad \bigl(h \in L^2(M), \; t \geq 0\bigr).
\end{equation}
\end{proposition}

\begin{proof}
For $t > 0$, let $\mu_t^k = \alpha_t^k \otimes \delta_0$, where $\{\alpha_t^k\}_{t \geq 0}$ is the $\convdiam{k}$-Gaussian convolution semigroup generated by $\ell_k$ (Proposition \ref{prop:SLappendix_convsemigr_props}(a)). We recall from the proof of Corollary \ref{cor:coneconv_wsol_eigeneq} that we have $e^{-t\lambda} \in L^2(\bm{\rho_k})$ and $\alpha_t^k(dx) = (\mathcal{F}_{\ell_k}^{-1} e^{-t\bm{\cdot}})(x) B_k(x) dx$, where $\mathcal{F}_{\ell_k}^{-1} e^{-t\bm{\cdot}} \in L^1(\mathbb{R}^+,\, B_k(x)dx)$.

Our first claim is that the measure ${1 \over \zeta_k} (\mu_t^k \conv{k} (\zeta_k \ccdot \mathrm{m}_0))$ is absolutely continuous with respect to $\omega$ and that its density function $q_{\mu_t^k,\mathrm{m}_0}$ belongs to $L^2(M)$. Note first that, by assumption, $(\widehat{\mathrm{m}_0})_j = 0$ for $j \neq k$, and therefore (e.g.\ by Proposition \ref{prop:coneconv_fourmeas_props}(ii)) $\mathrm{m}_0 = (\widehat{\mathrm{m}_0})_k \otimes \phi_k$, where $\phi_k$ is the measure on $\mathbb{T}$ defined by $\phi_k(d\theta) = e^{i2k\pi \theta} d\theta$. We thus have 
\[
\mu_t^k \conv{k} (\zeta_k \ccdot \mathrm{m}_0) = (\alpha_t^k \convdiam{k} (\zeta_k \ccdot (\widehat{\mathrm{m}_0})_k)) \otimes \phi_k.
\]
The absolute continuity assumption on $\mathrm{m}_0$ implies that $(\widehat{\mathrm{m}_0})_k(dx) = (\widehat{q_{\mathrm{m}_0}})_k(x) A(x) dx$ with $(\widehat{q_{\mathrm{m}_0}})_k \in L^2(A)$, so we can now use the properties of the convolution $\convdiam{k}$ (see Proposition \ref{prop:SLappendix_convtransl_props}(f)) to conclude that ${1 \over \zeta_k} (\alpha_t^k \convdiam{k} (\zeta_k \ccdot (\widehat{\mathrm{m}_0})_k))$ is also absolutely continuous with respect to $A(x) dx$ with density belonging to $L^2(A)$, and this proves the claim.

Let $h \in L^2(M)$. Combining the above with Proposition \ref{prop:coneconv_Feigenexp}, we may now compute
\begin{align*}
\int_M (e^{t\Delta_N} h)(\bm{\xi})\, \mathrm{m}_0(d\bm{\xi}) & = \bigl\langle e^{t\Delta_N} h, \overline{q_{\mathrm{m}_0}} \bigr\rangle_{L^2(M)} \, = \,\sum_{j \in \mathbb{Z}} \bigl\langle \bm{\mathcal{F}}(e^{t\Delta_N} h)_j, (\bm{\mathcal{F}} \, \overline{q_{\mathrm{m}_0}})_j \bigr\rangle_{L^2(\bm{\rho_j})} \\
& = \bigl\langle e^{-t\bm{\cdot}}(\bm{\mathcal{F}}h)_{-k}, \, (\bm{\mathcal{F}} \, \overline{q_{\mathrm{m}_0}})_{-k} \bigr\rangle_{L^2(\bm{\rho_k})} \, = \, \Bigl\langle (\bm{\mathcal{F}}h)_{-k}, \, (\bm{\mathcal{F}} \mu_t^{\smash{k}})(-k,\bm{\cdot}) \, (\bm{\mathcal{F}} \, \overline{q_{\mathrm{m}_0}})_{-k} \Bigr\rangle_{L^2(\bm{\rho_k})} \\
& = \sum_{j \in \mathbb{Z}} \bigl\langle (\bm{\mathcal{F}}h)_j, \, (\bm{\mathcal{F}} \, \overline{q_{\mu_t^k,\mathrm{m}_0}})_j \bigr\rangle_{L^2(\bm{\rho_j})} \, = \, \bigl\langle h, \overline{q_{\mu_t^k, \mathrm{m}_0}} \bigr\rangle_{L^2(M)} \\
& = \int_M {h(\bm{\xi}) \over \zeta_k(x)} \, \bigl(\mu_t^k \conv{k} (\zeta_k \ccdot \mathrm{m}_0)\bigr)(d\bm{\xi})
\end{align*}
so that \eqref{eq:coneconv_gentransl_BMlevy} holds.
\end{proof}

As observed in Section \ref{sec:eigenexp}, the sesquilinear form $\mathcal{E}$ associated with the heat semigroup $e^{t\Delta_N}$ is a nonnegative, closed, Markovian symmetric form defined on $H^1(M) \times H^1(M)$; in other words, $(\mathcal{E}, H^1(M))$ is a \emph{Dirichlet form} on $L^2(M)$. One can also check (cf.\ \cite{boscainprandi2016,fukushimaetal2011}) that the Dirichlet form $(\mathcal{E}, H^1(M))$ is \emph{regular}, that is, $H^1(M) \cap \mathrm{C}_\mathrm{c}(M)$ is dense both in $H^1(M)$ with respect to the norm $\|u\|_{H^1(M)} = \sqrt{\mathcal{E}(u,u) + \|u\|_{L^2(M)}}\,$ and in $\mathrm{C}_\mathrm{c}(M)$ with respect to the sup norm. Therefore, by a basic result from the theory of Dirichlet forms \cite[Theorem 7.2.1]{fukushimaetal2011}, there exists a Hunt process with state space $M$ whose transition semigroup $\{P_t\}_{t \geq 0}$ is such that $P_t u$ is, for all $u \in \mathrm{C}_\mathrm{c}(M)$, a quasi-continuous version of $e^{t\Delta_N} u$. (A Hunt process is essentially a strong Markov process whose paths are right-continuous and quasi-left-continuous; for details we refer to \cite[Appendix A.2]{fukushimaetal2011}.)

Accordingly, \eqref{eq:coneconv_gentransl_BMlevy} can be rewritten as 
\begin{equation} \label{eq:coneconv_gentransl_BMlevy_v2}
\mathbb{E}_{\mathrm{m}_0}[h(W_t)] = \int_M h\, d(\widetilde{\mskip 0.5\thinmuskip\mu\,}\phantom{\!}_t^k \convstar{k} \mathrm{m}_0), \qquad \bigl(h \in L^2(M), \; t \geq 0\bigr)
\end{equation}
where:
\begin{itemize}[itemsep=0pt,topsep=4pt]
\item $\{W_t\}_{t \geq 0}$ is the \emph{reflected Brownian motion} on the manifold $(M,g)$, i.e.\ $\{W_t\}$ is the Hunt process on $M$ determined by the regular Dirichlet form $(\mathcal{E}, H^1(M))$;
\item $\mathbb{E}_{\mathrm{m}_0}$ is the expectation operator of the process with initial distribution $\mathrm{m}_0 \in \mathcal{M}_\mathbb{C}(M)$ (defined as $\mathbb{E}_{\mathrm{m}_0}[h(W_t)] := \int_M \mathbb{E}_{\bm{\xi}} [h(W_t)] \mathrm{m}_0(d\bm{\xi})$, where $\mathbb{E}_{\bm{\xi}}$ is the usual expectation operator for the process started at the point $\bm{\xi}$);
\item The convolution $\convstar{k}$ is defined by $\nu_1 \convstar{k} \nu_2 = {1 \over \zeta_k} \bigl((\zeta_k \ccdot \nu_1) \conv{k} (\zeta_k \ccdot \nu_2)\bigr)$ or, equivalently, by $(\nu_1 \convstar{k} \nu_2)(\bm{\cdot}) = \int_M \int_M \bm{\gamma}_{k,\bm{\xi_1}, \bm{\xi_2}}(\bm{\cdot}) \, \nu_1(d\bm{\xi_1}) \, \nu_2(d\bm{\xi_2})$, with $\bm{\gamma}_{k,\bm{\xi_1}, \bm{\xi_2}}$ given as in \eqref{eq:coneconv_gammaconvmeas_def};
\item $\widetilde{\mskip 0.5\thinmuskip\mu\,}\phantom{\!}_t^k := {\mu_t^k \over \zeta_k}$ (so that $\widetilde{\mskip 0.5\thinmuskip\mu\,}\phantom{\!}_t^k$ satisfies the convolution semigroup property with respect to $\convstar{k}$).
\end{itemize}

\begin{corollary} \label{cor:coneconv_rbm_convsemigroup}
Let $\mathrm{m}_0 \in \mathcal{M}_\mathbb{C}(M)$ be an absolutely continuous measure with respect to $\omega$ whose density function $q_{\mathrm{m}_0}$ belongs to $L^2(M) \cap \bigl(\bigcap_{k=0}^\infty L^1(M, \zeta_k \ccdot \omega)\bigr)$. Then there exist Gaussian $\Delta_k$-convolution semigroups $\{\mu_t^k\}_{t \geq 0}$ such that
\[
\int_M (e^{t\Delta_N} h)(\bm{\xi})\, \mathrm{m}_0(d\bm{\xi}) = \sum_{k \in \mathbb{Z}} \int_M e^{i2k\pi\theta} {\widehat{h}_k(x) \over \zeta_k(x)} \, \bigl(\mu_t^k \conv{k} (\zeta_k \ccdot \bm{\mathrm{m}}_{0,-k})\bigr)(d\bm{\xi}) \qquad \bigl(h \in L^2(M), \; t \geq 0\bigr)
\]
where $\bm{\mathrm{m}}_{0,k} = (\widehat{\mathrm{m}_0})_k \otimes \phi_k$ and $\widehat{h}_k$ is given as in \eqref{eq:coneconv_Feigenexp_pf1}.
\end{corollary}

\begin{proof}
We have
\[
\int_M e^{t\Delta_N} \biggl(\sum_{k \in \mathbb{Z}} e^{i2k\pi\theta\,} \widehat{h}_k(x)\biggr)\, \mathrm{m}_0(d\bm{\xi}) = \sum_{k \in \mathbb{Z}} \int_M e^{t\Delta_N} \bigl(e^{i2k\pi\theta\,} \widehat{h}_k(x)\bigr)\, \bigl((\widehat{\mathrm{m}_0})_{-k} \otimes \phi_{-k}\bigr)(d\bm{\xi}).
\]
Since each measure $(\widehat{\mathrm{m}_0})_{-k} \otimes \phi_{-k}$ satisfies $\smash{\bigl((\widehat{\mathrm{m}_0})_{-k} \otimes \phi_{-k}\bigr)\strut_{\!j}\!\!\scalebox{1.4}{$\widehat{\;}$} = 0}$ for $j \neq -k$, the corollary follows by applying Proposition \ref{prop:coneconv_rbm_convsemigroup} to each term in the right-hand side.
\end{proof}

We now extend the result of Proposition \ref{prop:coneconv_rbm_convsemigroup} to other Markovian semigroups whose generators are functions (in the functional calculus sense) of the Laplace-Beltrami operator.

\begin{proposition} \label{prop:coneconv_convsemigroup_1dexponent}
For $k \in \mathbb{Z}$, let $\mathrm{m}_0 \in \mathcal{M}_\mathbb{C}(M)$ be an absolutely continuous measure with respect to $\omega$ whose density function $q_{\mathrm{m}_0}$ belongs to $L^2(M) \cap L^1(M, \zeta_k \ccdot \omega)$, and such that $(\widehat{\mathrm{m}_0})_j = 0$ for each $j \neq k$. Let $\psi_k$ be a function of the form
\begin{equation} \label{eq:coneconv_convsemigroup_1dexponent_asmp}
\psi_k(\lambda) = c\lambda + \int_{\mathbb{R}^+} (1-\widetilde{v}_{k,\lambda}(x)) \, \tau(dx) \qquad (\lambda \geq 0)
\end{equation}
where $c \geq 0$ and $\tau$ is a $\sigma$-finite measure on $\mathbb{R}^+$ which is finite on the complement of any neighbourhood of $0$ and such that $\int_{\mathbb{R}^+} (1-\widetilde{v}_{k,\lambda}(x)) \, \tau(dx) < \infty$ for $\lambda \geq 0$. Assume also that $e^{-t\psi_k(\bm{\cdot})} \in L^2(\bm{\rho_k})$ for all $t > 0$. Then there exists a $\Delta_k$-convolution semigroup $\{\mu_t^{\psi_k}\}_{t \geq 0}$ such that
\begin{equation} \label{eq:coneconv_convsemigroup_1dexponent}
\int_M (e^{-t\psi_k(-\Delta_N)} h)(\bm{\xi})\, \mathrm{m}_0(d\bm{\xi}) = \int_M {h(\bm{\xi}) \over \zeta_k(x)} \, \bigl(\mu_t^{\psi_k} \conv{k} (\zeta_k \ccdot \mathrm{m}_0)\bigr)(d\bm{\xi}) \qquad \bigl(h \in L^2(M), \; t \geq 0\bigr)
\end{equation}
where $e^{-t\psi_k(-\Delta_N)}$ is defined via the spectral theorem for the self-adjoint operator $(-\Delta_N, \mathcal{D}_N)$.
\end{proposition}

We observe that, since $e^{-t\lambda} \in L^2(\bm{\rho_k})$ for all $t > 0$, the assumption $e^{-t\psi_k(\bm{\cdot})} \in L^2(\bm{\rho_k})$ is automatically satisfied whenever $c > 0$ in the right hand side of \eqref{eq:coneconv_convsemigroup_1dexponent_asmp}.

\begin{proof}
By Proposition \ref{prop:SLappendix_convsemigr_props}(b), there exists a $\convdiam{k}$-convolution semigroup $\{\alpha_t^{\psi_k}\}_{t \geq 0}$ such that $(\mathcal{F}_{\ell_k} \, \alpha_t^{\psi_k})(\lambda) = e^{-t\psi_k(\lambda)}$. Using Proposition \ref{prop:SLappendix_eigenexp} and the assumption $e^{-t\psi_k(\bm{\cdot})} \in L^2(\bm{\rho_k})$, we deduce that $\alpha_t^{\psi_k}(dx) = (\mathcal{F}_{\ell_k}^{-1} e^{-t\psi_k(\bm{\cdot})})(x) B_k(x) dx$, where $\mathcal{F}_{\ell_k}^{-1} e^{-t\psi_k(\bm{\cdot})} \in L^1(\mathbb{R}^+,\, B_k(x)dx)$. The result can now be proved using the same argument as in Proposition \ref{prop:coneconv_rbm_convsemigroup} above.
\end{proof}

The sesquilinear form  $\mathcal{E}^{\psi_k} : \mathcal{D}(\mathcal{E}^{\psi_k}) \times \mathcal{D}(\mathcal{E}^{\psi_k}) \longrightarrow \mathbb C$ associated with the Markovian self-adjoint operator $-\psi_k(-\Delta_N)$, defined as
\[
\mathcal{D}(\mathcal{E}^{\psi_k}) = \mathcal{D}\bigl( \sqrt{\psi_k(-\Delta_N)} \mskip0.7\thinmuskip \bigr), \qquad \mathcal{E}^{\psi_k}(u,v) = \bigl\langle \sqrt{\psi_k(-\Delta_N)} \, u, \, \sqrt{\psi_k(-\Delta_N)} \, v \bigr\rangle_{L^2(M)},
\]
is a regular Dirichlet form on $L^2(M)$. (We can prove this claim using the upper bound \eqref{eq:SLaapendix_expon_ineq} and the proof of Proposition 3.1 of \cite{mcgillivray1997}.) Accordingly, the result stated before Corollary \ref{cor:coneconv_rbm_convsemigroup} ensures that there exists a Hunt process $\{X_t\}_{t \geq 0}$ with state space $M$ such that $(e^{-t\psi_k(-\Delta_N)} h)(\bm{\xi}) = \mathbb{E}_{\bm{\xi}}[h(X_t)]$, and therefore the convolution semigroup property \eqref{eq:coneconv_convsemigroup_1dexponent} translates into the Lévy-like representation
\[
\mathbb{E}_{\mathrm{m}_0}[h(X_t)] = \int_M h\, d(\widetilde{\mskip 0.5\thinmuskip\mu\,}\phantom{\!}_t^{\psi_k} \convstar{k} \mathrm{m}_0) \qquad \bigl(h \in L^2(M), \; t \geq 0\bigr)
\]
for the law of the process $\{X_t\}$. (Here $\mathrm{m}_0$ is any complex measure satisfying the assumptions in Proposition \ref{prop:coneconv_convsemigroup_1dexponent}.) The representation \eqref{eq:coneconv_gentransl_BMlevy_v2} for the law of reflected Brownian motion on $(M,g)$ is a particular case of this result.

\begin{remark}
In general we cannot state a counterpart of Corollary \ref{cor:coneconv_rbm_convsemigroup} for the semigroup $e^{-t\psi_k(-\Delta_N)}$. This would only be possible if $\psi_k(\lambda)$ did not depend on $k$, i.e.\ if a given function $\psi(\lambda)$ could be written, for each $k = 0,1,\ldots$, as $c_k\lambda + \int (1-\widetilde{v}_{k,\lambda}) \, d\tau_k$ with $c_k \geq 0$ and $\tau_k$ measures satisfying the conditions above, but there are no reasons to expect that this is possible other than in the trivial case $\psi(\lambda) = c \lambda$. (See \cite{zeuner1995}, where it is shown that the stable infinitely divisible measures for the convolution $\convdiam{k}$ are not the same for different values of $k$.)
\end{remark}

\section{Examples} \label{sec:examples}

We now review some special cases in which the theory of special functions provides further information on the eigenfunction expansion of $\Delta_N$ and the associated convolution structure. We start with an example where the solutions $w_{k,\lambda}$ can be expressed in terms of the Whittaker function of the second kind, and the Fourier decomposition gives rise to a family of Sturm-Liouville operators which are generators of drifted Bessel processes.

\begin{example} \label{exam:coneconv_exam_conehalf}
Consider the case $A(x) = x^{1/2}$, so that the Riemannian metric on $M = \mathbb{R}^+ \times \mathbb{T}$ is $g = dx^2 + x \, d\theta^2$, the volume form is $d\omega = \sqrt{x} \, dx d\theta$ and the Laplace-Beltrami operator on $(M, g)$ is $\Delta = \partial_x^2 + {1 \over 2x} \partial_x + {1 \over x} \partial_\theta^2$.
\begin{enumerate}
\item[\textbf{\itshape(i)}]
\emph{Let $\bm{M}_{\alpha, \nu}(z) := z^\nu M_{\alpha,\nu}(z)$, where $M_{\alpha,\nu}(z)$ denotes the Whittaker function of the first kind \cite[\S13.14]{dlmf}. The unique solution of the boundary value problem \eqref{eq:coneconv_wsol_bvp} is given by
\begin{equation} \label{eq:coneconv_exam_Mwhit_wsol}
w_{k,\lambda}(x,\theta) = e^{i2k\pi\theta} \bm{M}_{\!{2(k\pi)^2 i \over \sqrt{\lambda}}, -{1 \over 4}}(2ix\sqrt{\lambda}).
\end{equation}}%
\end{enumerate}
Indeed, the function $v_{k,\lambda}(x) = e^{-\pi i/8} (2x\sqrt{\lambda})^{-1/4} M_{\!{2(k\pi)^2 i \over 2\sqrt{\lambda}}, -{1 \over 4}}(2ix\sqrt{\lambda})$ is, according to \cite[Equation 2.1.2.108]{polyaninzaitsev2003}, a solution of $-\Delta_k \, v = \lambda v$, and we can use the results of \cite[\S13.14(iii) and \S 13.15(ii)]{dlmf} to check that $v_{k,\lambda}(0) = 1$ and $(A v_{k,\lambda}')(0) = 0$.
\begin{enumerate}
\item[\textbf{\itshape(ii)}] \emph{Let $\mb{r}_{\bm{\rho_k}}(\lambda) := 2^{-3/2} \pi^{-2} \lambda^{-1/4} \exp\bigl(-{2 k^2 \pi^3 \over \sqrt{\lambda}}\bigr) \bigl|{\Gamma\bigl({1 \over 4} - {2(k\pi)^2 i \over \sqrt{\lambda}}\bigr)}\bigr|^2$. The integral operator $\bm{\mathcal{F}}: L^2(M) \longrightarrow \bigoplus_{k \in \mathbb{Z}} L^2(\mathbb{R}^+, \mb{r}_{\bm{\rho_k}}(\lambda) d\lambda)$ defined by
\[
(\bm{\mathcal{F}} h)_k(\lambda) := \int_0^\infty \int_0^1 h(x,\theta) \, e^{-i2k\pi\theta} d\theta \, \bm{M}_{\!{2(k\pi)^2 i \over \sqrt{\lambda}}, -{1 \over 4}}(2ix\sqrt{\lambda}) \, x^{1/2} dx
\]
is a spectral representation of the Laplace-Beltrami operator (cf.\ Proposition \ref{prop:coneconv_Feigenexp}), and its inverse is given by
\[
(\bm{\mathcal{F}}^{-1} \{\varphi_k\})(\bm{\xi}) = \sum_{k \in \mathbb{Z}} \int_0^\infty \varphi_k(\lambda) \, e^{i2k\pi\theta} \bm{M}_{\!{2(k\pi)^2 i \over \sqrt{\lambda}}, -{1 \over 4}}(2ix\sqrt{\lambda}) \, \mb{r}_{\bm{\rho_k}}(\lambda) d\lambda.
\]}%
\end{enumerate}
By Proposition \ref{prop:coneconv_Feigenexp}, to prove \textbf{\emph{(ii)}} we only need to show that the spectral measure of the self-adjoint realization of $\Delta_k$ determined by the boundary condition $(Av')(0)= 0$ is given by $\bm{\rho_k}(d\lambda) = \mb{r}_{\bm{\rho_k}}(\lambda) d\lambda$. But this fact is a consequence of the general results of \cite{linetsky2004b} on spectral representations associated with the Sturm-Liouville operator $\bm{\ell}_{\nu,\mu} = - {d^2 \over dx^2} - ({2\nu + 1 \over x} + \mu){d \over dx}$. Indeed, it follows from \cite[Proposition 1]{linetsky2004b} (see also \cite{titchmarsh1962}) that the integral transforms
\begin{align*}
(\mathcal{Q}f)(\tau) & = (2\tau)^{-1/4} \int_0^\infty f(x) \, x^{1/4} \, e^{8(k\pi)^2x - \pi i/8} M_{\!{2(k\pi)^2 i \over \tau}, -{1 \over 4}}(2ix\tau) \, dx \\
(\mathcal{Q}^{-1} \eta)(x) & = {x^{1/4} \, e^{-8(k\pi)^2x - \pi i/8}  \over 2^{3/4} \, \pi^2} \int_0^\infty \eta(\tau) \, M_{\!{2(k\pi)^2 i \over \tau}, -{1 \over 4}}(2ix\tau) \, \tau^{1/4} \exp\Bigl(-{2k^2\pi^3 \over \tau}\Bigr) \, \biggl|\Gamma\biggl({1 \over 4} - {2(k\pi)^2 i \over \tau}\biggr)\biggr|^2 d\tau
\end{align*}
define an isometric isomorphism $L^2(\mathbb{R}^+,x^{1/2}e^{(4k\pi)^2 x}) \longrightarrow L^2\bigl(\mathbb{R}^+, ({\tau \over 2})^{1/2}\, \pi^{-2}\exp(-{2k^2\pi^3 \over \tau}) \, \bigl|\Gamma({1 \over 4} - {2(k\pi)^2 i \over \tau})\bigr|^2 d\tau \bigr)$ satisfying $\mathcal{Q}(\ell_{-{1 \over 4}, 8(k\pi)^2} f)(\lambda) = \lambda \ccdot (\mathcal{Q}f)(\lambda)$. Since the operators $\ell_{-{1 \over 4}, 8(k\pi)^2}$ and $\Delta_k$ are related via an elementary change of variables, we easily conclude that $\bm{\rho_k}(d\lambda) = \mb{r}_{\bm{\rho_k}}(\lambda) d\lambda$.
\begin{enumerate}
\item[\textbf{\itshape(iii)}] \emph{For each $k \in \mathbb{N}_0$ and $\bm{\xi_j} = (x_j,\theta_j) \in M$ ($j=1,2$) there exists a positive measure $\bm{\gamma}_{k,\bm{\xi_1}, \bm{\xi_2}}$ on $M$ such that for all $\tau \in \mathbb{C}$ the generalized eigenfunctions \eqref{eq:coneconv_exam_Mwhit_wsol} satisfy
\begin{equation} \label{eq:coneconv_exam_Mwhit_prodform}
e^{i2k\pi(\theta_1+\theta_2)} \bm{M}_{\!{2(k\pi)^2 i \over \tau}, -{1 \over 4}}(2ix_1\tau)\, \bm{M}_{\!{2(k\pi)^2 i \over \tau}, -{1 \over 4}}(2ix_2\tau) = \int_M e^{i2k\pi\theta_3} \bm{M}_{\!{2(k\pi)^2 i \over \tau}, -{1 \over 4}}(2ix_3\tau)\, \bm{\gamma}_{k,\bm{\xi_1},\bm{\xi_2}}(d\bm{\xi_3}).
\end{equation}
The support of measure $\bm{\gamma}_{k,\bm{\xi_1},\bm{\xi_2}}$ is the set $[|x_1-x_2|,x_1+x_2] \times \{\theta_1 + \theta_2\}$.
}

\item[\textbf{\itshape(iii')}] \emph{When $k = 0$, the product formula \eqref{eq:coneconv_exam_Mwhit_prodform} reduces to
\begin{equation} \label{eq:coneconv_exam_Mwhit_prodform_0}
\bm{J}_{-{1 \over 4}}(x_1\tau) \, \bm{J}_{-{1 \over 4}}(x_2\tau) = \int_M \bm{J}_{-{1 \over 4}}(x_3\tau) \, \bm{\gamma}_{0,\bm{\xi_1},\bm{\xi_2}}(d\bm{\xi_3})
\end{equation}
where $\bm{J}_\alpha(\tau x) := 2^\alpha \Gamma(\alpha+1) (\tau x)^{-\alpha} J_\alpha(\tau x)$ and $J_\alpha$ is the Bessel function of the first kind. The measures $\bm{\gamma}_{0,\bm{\xi_1},\bm{\xi_2}}$ in \eqref{eq:coneconv_exam_Mwhit_prodform_0} are explicitly given by 
\begin{align*}
\bm{\gamma}_{0,\bm{\xi_1},\bm{\xi_2}}(d\bm{\xi_3}) & = {2^{3/2} \Gamma({3 \over 4}) \over \sqrt{\pi} \, \Gamma({1 \over 4})} \bigl[ (x_3^2 - (x_1-x_2)^2) ((x_1+x_2)^2 - x_3^2) \bigr]^{-3/4} \\
& \qquad\qquad \times \sqrt{x_1 x_2} \, x_3 \mathds{1}_{[|x_1-x_2|,x_1+x_2]}(x_3) \, dx_3 \, \delta_{\theta_1+\theta_2}(d\theta_3).
\end{align*}
}
\end{enumerate}
Property \emph{\textbf{(iii)}} is a particular case of Proposition \ref{prop:coneconv_prodform}. The identity $\bm{M}_{0,-{1 \over 4}}(2ix\tau) = \bm{J}_{-{1 \over 4}}(x\tau)$ (cf.\ \cite[\S10.27 and \S13.18(iii)]{dlmf}) leads to  \eqref{eq:coneconv_exam_Mwhit_prodform_0}. The closed-form expression for the measures $\bm{\gamma}_{0,\bm{\xi_1},\bm{\xi_2}}(d\bm{\xi_3})$ follows from the well-known product formula for the Bessel function of the first kind \cite{hirschman1960,watson1944}. We mention that the convolution $(\mu \convstar{0} \nu)(\bm{\cdot}) = \int_M \int_M \bm{\gamma}_{0,\bm{\xi_1}, \bm{\xi_2}}(\bm{\cdot}) \, \mu(d\bm{\xi_1}) \, \nu(d\bm{\xi_2})$ is (modulo the product with the trivial convolution on the torus) a particular case of the Bessel-Kingman convolution \cite{kingman1963}, which is one of most notable examples of Sturm-Liouville hypergroups \cite{bloomheyer1994,rentzschvoit2000}.

It is natural to conjecture that the measures $\bm{\gamma}_{k,\bm{\xi_1},\bm{\xi_2}}$ ($k = 1,2,\ldots$) can also be written in closed form in terms of classical special functions. If this is true, determining such a closed form expression is likely to require a detailed and nontrivial analysis of the properties of the Whittaker function of the first kind. (Compare with e.g.\ \cite{koornwinderschwartz1997,laine1980,sousaetal2019a}, where nontrivial product formulas have been determined for other special functions.)

According to \cite{boscainprandi2016}, one can formally interpret the manifold $(M,g)$ as a cone-like surface of revolution $\mathcal{S} = \{(t,r(t)\cos\theta,r(t)\sin\theta) \mid t > 0, \, \theta \in \mathbb{T}\}$ with profile $r(t) \sim \sqrt{t}$ as $t \downarrow 0$. The properties of self-adjoint extensions of the Laplace-Beltrami operator (and the corresponding Markovian semigroups) on such cone-like manifolds have been widely studied, see \cite{boscainprandi2016} and references therein. As a particular case of Corollary \ref{cor:coneconv_rbm_convsemigroup}, we obtain the following convolution semigroup property for the heat semigroup generated by the Neumann realization of the Laplace-Beltrami operator on $(M,g)$:
\begin{enumerate}
\item[\textbf{\itshape(iv)}] \emph{If $\mathrm{m}_0 \in \mathcal{M}_\mathbb{C}(M)$ satisfies the absolute continuity assumption of Corollary \ref{cor:coneconv_rbm_convsemigroup}, then the transition probabilities of the reflected Brownian motion $\{W_t\}$ on the manifold $(M,g)$ with initial distribution $\mathrm{m}_0$ can be written as
\begin{equation} \label{eq:coneconv_exam_Mwhit_rbm_convsemigroup}
\begin{aligned}
\mathbb{E}_{\mathrm{m}_0}[h(W_t)] & \equiv \int_M (e^{t\Delta_N} h)(\bm{\xi})\, \mathrm{m}_0(d\bm{\xi}) \\
& = \sum_{k \in \mathbb{Z}} \int_M e^{i2k\pi\theta\,} \widehat{h}_k(x) \, \bigl(\widetilde{\mu}_t^k \convstar{k} \bm{\mathrm{m}}_{0,-k}\bigr)(d\bm{\xi}) \qquad \bigl(h \in L^2(M), \; t \geq 0\bigr)
\end{aligned}
\end{equation}
where $\{\widetilde{\mu}_t^k\}_{t \geq 0}$ is a convolution semigroup with respect to the convolution $\convstar{k}$ defined by $(\mu \convstar{k} \nu)(\bm{\cdot}) = \int_M \int_M \bm{\gamma}_{k,\bm{\xi_1}, \bm{\xi_2}}(\bm{\cdot}) \, \mu(d\bm{\xi_1}) \, \nu(d\bm{\xi_2})$,\, $\widehat{h}_k(x) := \int_0^1 e^{-i2k\pi\vartheta} h(x,\vartheta) d\vartheta$ and the measures $\bm{\mathrm{m}}_{0,-k}$ are defined as in Corollary \ref{cor:coneconv_rbm_convsemigroup}.
}
\end{enumerate}
As we saw earlier, the convolution semigroups $\{\widetilde{\mu}_t^k\}$ are of the form $({\alpha_t^k \over \zeta_k}) \otimes \delta_0$, where $\{\alpha_t^k\}$ is the law of the one-dimensional diffusion process (started at $x=0$) generated by the Sturm-Liouville operator $\ell_k$ defined in \eqref{eq:coneconv_wsol_pf3}. The differential equation $\ell_k u = \lambda u$ can be transformed, by the change of dependent variable $U(x) = \zeta_k(x) e^{-2k^2 x} u(x)$, into an equation of the form $\mathfrak{L}_k U = \Lambda U$ ($\Lambda \in \mathbb{C}$), where $\mathfrak{L}_k = -{d^2 \over dx^2} - \bigl({1 \over 2x} + (4k\pi)^2\bigr){d \over dx}$, i.e.\ $\mathfrak{L}_k$ is the infinitesimal generator of a Bessel process with constant drift $(4k\pi)^2$ \cite{linetsky2004b}. Therefore, the identity \eqref{eq:coneconv_exam_Mwhit_rbm_convsemigroup} shows that the transition probabilities of $\{W_t\}$ can be decomposed in terms of transition probabilities of (one-dimensional) drifted Bessel processes.

Finally, we call attention to the following convolution semigroup representation for Markovian semigroups generated by fractional powers of the Laplace-Beltrami operator, which can be deduced from Proposition \ref{prop:coneconv_convsemigroup_1dexponent}:
\begin{enumerate}
\item[\textbf{\itshape(v)}] \emph{Let $\mathrm{m}_0 \in \mathcal{M}_\mathbb{C}(M)$ satisfy the assumptions of Proposition \ref{prop:coneconv_convsemigroup_1dexponent} and $(\widehat{\mathrm{m}_0})_j = 0$ for each $j \neq 0$. Let $0 < q < 1$. Then the Markovian semigroup generated by the operator $-(-\Delta_N)^q$ is such that
\[
\int_M (e^{-t(-\Delta_N)^q} h)(\bm{\xi})\, \mathrm{m}_0(d\bm{\xi}) = \int_M h(\xi) \, \bigl(\nu_{q,t} \convstar{0} \mathrm{m}_0\bigr)(d\bm{\xi}) \qquad \bigl(h \in L^2(M), \; t \geq 0\bigr)
\]
where $\{\nu_{q,t}\}_{t \geq 0}$ is a $\convstar{0}$-convolution semigroup.
}
\end{enumerate}
(To prove \emph{\textbf{(v)}}, we also need to recall the following special property of the Bessel-Kingman convolution \cite[Theorem 2]{urbanik1988}: for each $0 < q < 1$ there exists a measure $\sigma_q \in \mathcal{P}(\mathbb{R}_0^+)$ which is infinitely divisible with respect to the Bessel-Kingman convolution $\convstar{0}$ and such that $\int_{\mathbb{R}_0^+} \bm{J}_{-{1 \over 4}}(x\sqrt{\lambda}) \, \sigma_q(dx) = e^{-\lambda^q}$. It is easy to check that $e^{-t\lambda^q} \in L^2(\bm{\rho_0})$ for all $t > 0$.)
\end{example}

\begin{example} \label{exam:coneconv_exam_conebeta}
Consider now the more general case $A(x) = x^\beta$ with $0 < \beta < 1$. The corresponding Riemannian metric, $g = dx^2 + x^{2\beta} d\theta^2$, endows the space $M = \mathbb{R}^+ \times \mathbb{T}$ with a metric structure which, like in the previous example, can be formally interpreted as that of a surface of revolution with profile $r(t) \sim t^{\beta}$ as $t \downarrow 0$.

If $\beta \neq {1 \over 2}$, the solution of the boundary value problem \eqref{eq:coneconv_wsol_bvp} and the spectral measures $\bm{\rho_k}$ ($k=1,2,\ldots$) can no longer be written in closed form. Nevertheless, the convolution semigroup property of the Laplace-Beltrami operator $\Delta = \partial_x^2 + {\beta \over x} \partial_x + {1 \over x^{2\beta}} \partial_\theta^2$ on the cone-like manifold $(M,g)$, stated in property \emph{\textbf{(iv)}} of the previous example, continues to hold in this more general setting.

For $k=0$, the solutions of $-\Delta_0 v(x) \equiv -v''(x) - {\beta \over x} v'(x) = \lambda v(x)$ are the normalized Bessel functions with parameter ${\beta-1 \over 2}$. Therefore, we have the following extension of property \emph{\textbf{(iii')}} of the preceding example: \emph{the product formula
\[
\bm{J}_{\beta - 1 \over 2}(x_1\tau) \, \bm{J}_{\beta - 1 \over 2}(x_2\tau) = \int_M \bm{J}_{\beta - 1 \over 2}(x_3\tau) \, \bm{\gamma}_{0,\bm{\xi_1},\bm{\xi_2}}(d\bm{\xi_3}) \qquad (\tau \in \mathbb{C})
\]
holds for all $\bm{\xi_1}, \bm{\xi_2}\in M$, where the measures $\bm{\gamma}_{0,\bm{\xi_1},\bm{\xi_2}}$ are given by
\begin{align*}
\bm{\gamma}_{0,\bm{\xi_1},\bm{\xi_2}}(d\bm{\xi_3}) & = {2^{2-\beta} \Gamma({\beta + 1 \over 2}) \over \sqrt{\pi} \, \Gamma({\beta \over 2})} \bigl[ (x_3^2 - (x_1-x_2)^2) ((x_1+x_2)^2 - x_3^2) \bigr]^{\beta/2 - 1} \\
& \qquad\qquad \times (x_1 x_2)^{1-\beta} \, x_3 \mathds{1}_{[|x_1-x_2|,x_1+x_2]}(x_3) \, dx_3 \, \delta_{\theta_1+\theta_2}(d\theta_3).
\end{align*}}%
The convolution semigroup representation for $\{e^{-t(-\Delta_N)^q}\}_{t \geq 0}$,  formulated in property \emph{\textbf{(v)}} of the previous example, also extends to the case $0 < \beta < 1$ without any essential change.
\end{example}

In the latter example, the limiting case $\beta = 0$ corresponds to the standard metric structure on the cylinder $\mathbb{R}^+ \times \mathbb{T}$, which is a trivial case because a convolution associated with the Laplace-Beltrami operator $\Delta = \partial_x^2 + \partial_\theta^2$ can be introduced in a trivial way (namely, by taking the product of one-dimensional convolutions). In this trivial case, the convolutions introduced in the previous sections have a particularly simple structure:

\begin{example} \label{exam:coneconv_exam_conetriv}
If $A \equiv 1$, the Fourier decomposition \eqref{Eq Fourier decomposition DN Laplacian} yields the Sturm-Liouville operators $\Delta_k = {\partial_x^2} - (2k\pi)^2$. The eigenfunction expansion \eqref{eq:coneconv_Feigenexp_dir}--\eqref{eq:coneconv_Feigenexp_inv} is simply a composition of a Fourier series in the variable $\theta$ and a cosine Fourier transform in the variable $x$,
\begin{align*}
(\bm{\mathcal{F}}h)_k(\lambda) & = \int_0^\infty \! \int_0^1 h(\bm{\xi}) e^{-i2k\pi\theta} d\theta\, \cos(x\widetilde{\lambda}_k) \, dx \qquad (\widetilde{\lambda}_k = \sqrt{\lambda - (2k\pi)^2}) \\
(\bm{\mathcal{F}}^{-1} \{\varphi_k\})(\bm{\xi}) & = {1 \over \pi} \sum_{k \in \mathbb{Z}} \int_{(2k\pi)^2}^\infty \varphi_k(\lambda) e^{i2k\pi\theta} \cos(x\widetilde{\lambda}_k) \, \widetilde{\lambda}_k^{-1} d\lambda,
\end{align*}
and the product formula $w_{k,\lambda}(\bm{\xi_1}) \, w_{k,\lambda}(\bm{\xi_2}) = \int_M w_{k,\lambda}\, d\bm{\gamma}_{k,\bm{\xi_1},\bm{\xi_2}}$ (where $w_{k,\lambda}(\bm{\xi_1}) = e^{-i2k\pi\theta} \cos(x\widetilde{\lambda}_k)$) holds for the measures $\bm{\gamma}_{k,\bm{\xi_1},\bm{\xi_2}} = {1 \over 2}(\delta_{|x_1-x_2|} + \delta_{x_1+x_2}) \otimes \delta_{\theta_1+\theta_2}$, which do not depend on $k$. The convolution $\bm{\star} \equiv \convstar{k}$ is such that 
\[
\delta_{\bm{\xi_1}} \bm{\star} \delta_{\bm{\xi_2}} = (\delta_{x_1} \convdiam{\!\mathrm{sym}\!} \delta_{x_2}) \otimes (\delta_{\theta_1} \convdiam{\mathbb{T}} \delta_{\theta_2}),
\]
i.e.\ it can be interpreted as a product of the convolution $\convdiam{\!\mathrm{sym}\!}$ of the symmetric hypergroup on $\mathbb{R}_0^+$ and the ordinary convolution $\convdiam{\mathbb{T}}$ on $\mathbb{T}$ (this is a product hypergroup structure, cf.\ \cite[Definition 1.5.29 and Example 3.5.73]{bloomheyer1994}). In turn, the convolution $\conv{k}$ of Definition \ref{def:coneconv_convk} is such that
\[
\delta_{\bm{\xi_1}} \conv{k} \delta_{\bm{\xi_2}} = (\delta_{x_1} \convdiam{\!\mathrm{ch}_k\!} \delta_{x_2}) \otimes (\delta_{\theta_1} \convdiam{\mathbb{T}} \delta_{\theta_2}),
\]
where $\delta_{x_1} \convdiam{\!\mathrm{ch}_k\!} \delta_{x_2} = {\cosh(2k\pi|x_1-x_2|) \over 2\cosh(2k\pi x_1)\cosh(2k\pi x_2)} \delta_{|x_1-x_2|} + {\cosh(2k\pi(x_1+x_2)) \over 2\cosh(2k\pi x_1)\cosh(2k\pi x_2)} \delta_{x_1+x_2}$ is the convolution of a cosh hypergroup \cite[Example 3.5.71]{bloomheyer1994}. It is interesting to note that the convolution $\conv{k}$ is a modification of the convolution $\bm{\star}$, as defined in the theory of hypergroups \cite[Section 2.3]{bloomheyer1994}. We also observe that, since $\bm{\star}$ is the convolution of the symmetric hypergroup, Corollary \ref{cor:coneconv_rbm_convsemigroup} specializes to the following fact: \emph{under the assumption on $\mathrm{m}_0 \in \mathcal{M}_\mathbb{C}(M)$ given in Corollary \ref{cor:coneconv_rbm_convsemigroup}, the transition probabilities of the reflected Brownian motion $\{W_t\}$ on $(M,g)$ are such that
\[
\mathbb{E}_{\mathrm{m}_0}[h(W_t)] = \sum_{k \in \mathbb{Z}} e^{-(2k\pi)^2 t}\int_{\mathbb{R}_0^+} \widehat{h}_k(x) \, \bigl(p\rule[-.3\baselineskip]{0pt}{0.8\baselineskip}_{\mathbb{R}_0^+,t}\, \convdiam{\!\mathrm{sym}\!} (\widehat{\mathrm{m}_0})_{-k}\bigr)(dx) \qquad \bigl(h \in L^2(M), \; t \geq 0\bigr)
\]
where $\widehat{h}_k(x) := \int_0^1 e^{-i2k\pi\vartheta} h(x,\vartheta) d\vartheta$, and $p_{\mathbb{R}_0^+,t}$ is the law of a reflected Brownian motion on $\mathbb{R}_0^+$ started at $0$, which satisfies $p_{\mathbb{R}_0^+,t+s} = p_{\mathbb{R}_0^+,t} \convdiam{\mathrm{sym}} p_{\mathbb{R}_0^+,s}$.
}
\end{example}

Next we give two other particular cases of metrics $g$ where, as in Example \ref{exam:coneconv_exam_conebeta}, the corresponding families of convolutions $\conv{k}$ are related with well-known one-dimensional generalized convolutions:

\begin{example}
Consider $A(x) = (\sinh x)^{2\alpha + 1} (\cosh x)^{2\beta + 1}$, which satisfies condition \eqref{eq:coneconv_Ahyp} provided that $-{1 \over 2} \leq \beta \leq \alpha < 0$ with $\alpha \neq -{1 \over 2}$. The action of the Laplace-Beltrami operator \eqref{eq:coneconv_laplbeltr} on functions which do not depend on $\theta$ is the same as that of the Sturm-Liouville operator $\Delta_0 = \partial_x^2 + [(2\alpha + 1)\coth(x) + (2\beta + 1)\tanh(x)] \partial_x$. The solution of the boundary value problem $-\Delta_0 v = \lambda v$, $v(0) = 1$, $(Av')(0) = 0$ is the \emph{Jacobi function}
\[
v_{0,\lambda}(x) = {}_2F_1\Bigl(\tfrac{1}{2}(\sigma - i\tau), \tfrac{1}{2}(\sigma + i\tau); \alpha + 1; -(\sinh x)^2\Bigr) \qquad (\sigma = \alpha + \beta + 1,\; \lambda = \tau^2 + \sigma^2)
\]
for which the product formula $v_{0,\lambda}(x_1) \, v_{0,\lambda}(x_2) = \int_0^\infty v_{0,\lambda}\, d\pi_{x_1,x_2}^{[0]}$ holds with
\begin{align*}
\pi_{x_1,x_2}^{[0]}(dx_3) = \, &  {2^{-2\sigma} \Gamma(\alpha+1) (\cosh x_1 \, \cosh x_2 \, \cosh x_3)^{\alpha - \beta - 1} \over \sqrt{\pi}\, \Gamma(\alpha + {1 \over 2}) (\sinh x_1 \, \sinh x_2 \, \sinh x_3)^{2\alpha}} \times \\
& \times (1-Z^2)^{\alpha - 1/2} {\,}_2F_1\Bigl( \alpha + \beta, \alpha - \beta; \alpha + \tfrac{1}{2}; \tfrac{1}{2}(1-Z) \Bigr) \mathds{1}_{[|x_1-x_2|,x_1+x_2]}(x_3) A(x_3) dx_3
\end{align*}
where $Z := {(\cosh x_1)^2 + (\cosh x_2)^2 + (\cosh x_3)^2 - 1 \over 2\cosh x_1 \, \cosh x_2 \, \cosh x_3}$, see \cite{koornwinder1984}. The convolution determined by this product formula is the so-called \emph{Jacobi convolution}, which also gives rise to a notable example of a Sturm-Liouville hypergroup \cite[Example 3.5.64]{bloomheyer1994}.

It follows from Proposition \ref{prop:coneconv_prodform} that the generalized eigenfunctions of the Sturm-Liouville operator $\Delta_k = \Delta_0 - (2k\pi)^2 (\sinh x)^{-4\alpha - 2} (\cosh x)^{-4\beta - 2}$ also admit a similar product formula, whose measures are also supported on $[|x_1-x_2|,x_1+x_2]$. The results of the previous sections therefore show that the Jacobi convolution naturally extends into a convolution structure associated with the Laplace-Beltrami operator on $(M,g)$.
\end{example}

\begin{example}
Consider $A(x) = (1+x)^2$. The first Fourier component of the Laplace-Beltrami operator \eqref{eq:coneconv_laplbeltr} is $\Delta_0 = \partial_x^2 + {2 \over 1+x} \partial_x$. The solution of the boundary value problem $-\Delta_0 v = \lambda v$, $v(0) = 1$, $(Av')(0) = 0$ is the function
\[
v_{0,\lambda}(x) =
\begin{cases}
{1 \over 1+x} [\cos(\tau x) + {1 \over \tau} \sin(\tau x)], & \tau > 0 \\
1, & \tau = 0
\end{cases}
\qquad (\lambda = \tau^2)
\]
and the product formula for this function is $v_{0,\lambda}(x_1) \, v_{0,\lambda}(x_2) = \int_0^\infty v_{0,\lambda}\, d\pi_{x_1,x_2}^{[0]}$, where the measures are given by \cite[Example 4.10]{zeuner1992}
\begin{align*}
\pi_{x_1,x_2}^{[0]} & = {1 \over 2(1+x_1)(1+x_2)}\bigl[(1+|x_1-x_2|)\delta_{x_1-x_2}(dx_3) \\
& \qquad\quad + (1+x_1+x_2)\delta_{x_1+x_2}(dx_3) + (1+x_3)\mathds{1}_{[|x_1-x_2|,x_1+x_2]}(x_3) dx_3 \bigr].
\end{align*}
The convolution determined by this product formula gives rise to the so-called \emph{square hypergroup}. It follows from \cite[Theorem 3.14]{zeuner1992} that all the product formula measures $\pi_{x_1,x_2}^{[k]}$ associated  with the Fourier components $\Delta_k$ (cf.\ proof of Proposition \ref{prop:coneconv_prodform}) share with $\pi_{x_1,x_2}^{[0]}$ the property of having both a discrete component supported on $\{|x_1-x_2|,x_1+x_2\}$ and an absolutely continuous component supported on $[|x_1-x_2|,x_1+x_2]$. 

The Riemannian metric $g = dx^2 + (1+x)^4 d\theta^2$ satisfies the additional assumption in Remark \ref{rmk:coneconv_heatkern_existsuffcond}; therefore, there exists a heat kernel for the Laplace-Beltrami operator $\Delta = \partial_x^2 + {2 \over 1+x} \partial_x + {1 \over (1+x)^4} \partial_\theta^2$ and, as we saw in Proposition \ref{prop:coneconv_prodform_extended}, our results on the existence of product formulas not depending on $\lambda$ extend to all the functions $e^{t\Delta_N} w_{k,\lambda} = e^{-t\lambda} w_{k,\lambda}$\, ($t \geq 0$).
\end{example}

In all the examples above, the support of the convolution $\delta_{\bm{\xi_1}} \convstar{k} \delta_{\bm{\xi_2}} = \bm{\gamma}_{k,\bm{\xi_1}, \bm{\xi_2}}$ does not depend on the parameter $k$. Our final example shows that this is not always the case:

\begin{example}
Let $\phi \in \mathrm{C}_\mathrm{c}^\infty(\mathbb{R}_0^+)$ be a nonnegative decreasing function with $\supp(\phi) = [0,S]$ and let $A(x) = \exp(\int_0^x \phi(y)dy)$. We know that $\pi_{x_1,x_2}^{[k]} = \delta_{x_1} \convdiam{k} \delta_{x_2}$, where $\convdiam{k}$ is the convolution associated with the Sturm-Liouville operator $\ell_k = -{1 \over B_k}{d \over dx}(B_k {d \over dx})$. According to the general result of \cite[Proposition 5.7]{sousaetal2019b} on the support of convolutions associated with Sturm-Liouville operators, we have
\[
\supp(\pi_{x_1,x_2}^{[0]}) = \begin{cases}
[|x_1-x_2|, x_1+x_2], & \min\{x_1,x_2\} \leq 2S \\
[|x_1-x_2|, 2S + |x_1-x_2|] \cup [x_1+x_2-2S, x_1+x_2] & \min\{x_1,x_2\} > 2S
\end{cases}
\]
and
\[
\supp(\pi_{x_1,x_2}^{[k]}) = [|x_1-x_2|, x_1+x_2], \qquad k \geq 1.
\]
\end{example}

\section{Product formulas and convolutions associated with elliptic operators on subsets of $\mathbb{R}^2$} \label{sec:ellipticR2}

In the remainder of this paper, we will show that the techniques used above can also be used to construct families of generalized convolution operators associated with elliptic differential operators on $\mathbb{R}^+ \times I \subset \mathbb{R}^2$ of the general form 
\[
\bm{\mathcal{G}}^\wp = \partial_x^2 + {A'(x) \over A(x)} \partial_x + {1 \over A(x)^2} \wp_z, \qquad \bigl(x \in \mathbb{R}^+,\, z \in (a,b)\bigr)
\]
where $\wp_z = {1 \over \mathfrak{r}(z)} \bigl( \mathfrak{p}(z) \partial_z^2 + \mathfrak{p}'(z) \partial_z \bigr)$ belongs to a class of Sturm-Liouville differential operators on the interval $(a,b)$, $-\infty \leq a < b \leq \infty$. It will be assumed that the coefficients $\mathfrak{p}, \mathfrak{r}$ are such that $\mathfrak{p}, \mathfrak{r} > 0$,\, $\mathfrak{p}, \mathfrak{r}$ are locally absolutely continuous on the interior of $I$ and the endpoint $a$ is regular or entrance (cf.\ Equation \eqref{eq:bkg_diffusion_fellerBC} of the Appendix). As in the previous sections, the coefficient $A(x)$ is assumed to satisfy the conditions \eqref{eq:coneconv_Ahyp} and $\lim_{x \to \infty} A(x) = \infty$.

In what follows we set $I = [a,b)$ if $b$ is an exit or natural endpoint of $\wp$ and $I = [a,b]$ if the endpoint $b$ is regular or entrance. We shall write $M = \mathbb{R}_0^+ \times I$ and $\omega^\wp\bigl(d(x,z)\bigr) = A(x) dx\, \mathfrak{r}(z) dz$.

Let $\psi_\eta$ be the unique solution of the boundary value problem $-\wp_z(u) = \eta \mskip0.5\thinmuskip u$, $u(a) = 1$, $(\mathfrak{p}u')(a) = 0$ (cf.\ Lemma \ref{lem:SLappendix_boundaryvaluepb}). The eigenfunction expansion of the operator $\wp$ with Neumann boundary conditions (cf.\ Proposition \ref{prop:SLappendix_eigenexp}) yields the integral transform
\[
(\mathcal{J}_{\wp} g)(\eta) := \int_a^b g(z) \, \psi_\eta(z) \, \mathfrak{r}(z) dz, \qquad (\mathcal{J}_{\wp}^{-1} \varphi)(z) := \int_{\mathbb{R}_0^+} \varphi(\eta) \, \psi_\eta(z) \, \bm{\sigma}(d\eta)
\]
which is an isometric isomorphism between $L^2(\mathfrak{r}) \equiv L^2((a,b), \mathfrak{r}(z)dz)$ and $L^2(\mathbb{R}_0^+, \bm{\sigma})$.

If the endpoint $b$ is regular, entrance or exit (cf.\ Proposition \ref{prop:SLappendix_eigenexp}), then the spectral measure $\bm{\sigma}$ is discrete and the inverse integral transform is written as $(\mathcal{J}_{\wp}^{-1} \varphi)(z) = \sum_{k=1}^\infty {1 \over \|\psi_{\eta_k}\|^2}\mskip0.5\thinmuskip \varphi(\eta_k) \, \psi_{\eta_k}(z)$, where the $\eta_k$ are eigenvalues of $\wp$. In these conditions, the application of the eigenfunction expansion to functions $h(x,z) \in L^2(M,\omega^\wp)$ yields the decomposition
\[
L^2(M,\omega^\wp) = \bigoplus_{k = 1}^\infty H_{\eta_k}^\wp, \qquad H_\eta^\wp := \{\psi_\eta(z) v(x) \mid v \in L^2(A)\}.
\]
A similar expansion also holds if $b$ is natural, with the direct sums replaced by direct integrals \cite[\S7.4]{folland2016}. Note also that if $u \in H_\eta^\wp \cap \mathrm{C}_\mathrm{c}^\infty(M)$ then $\bm{\mathcal{G}}^\wp u = \mathcal{G}_{\eta_k} u$, where $\mathcal{G}_\eta := \partial_x^2 + {A'(x) \over A(x)} \partial_x - {\eta \over A(x)^2}$.

The following result is a counterpart of Proposition \ref{prop:coneconv_Feigenexp}:

\begin{proposition}
For each $(\lambda,\eta) \in \mathbb{C} \times \mathbb{R}_0^+$, there exists a unique solution $w_{\lambda,\eta}^\wp \in \{\psi_\eta(z) v(x) \mid v \in \mathrm{C}(\mathbb{R}_0^+)\}$ of the boundary value problem
\[
-\bm{\mathcal{G}}^\wp u = \lambda u, \qquad\;\; u(0,z) = \psi_\eta(z), \qquad\;\; u^{[1]}(0,z) = 0.
\]
There exists a locally finite positive Borel measure $\bm{\rho}^\wp$ on $(\mathbb{R}_0^+)^2$ such that the map $h \mapsto \bm{\mathcal{F}}_{\!\wp} h$, where
\begin{equation} \label{eq:coneconvSLcompact_Feigenexp_dir}
(\bm{\mathcal{F}}_{\!\wp} h)(\lambda,\eta) := \int_M h(x,z) \, w_{\lambda,\eta}^\wp(x,z) \, \omega^\wp(d(x,z)) \qquad \bigl(\lambda,\eta \geq 0\bigr),
\end{equation}
is an isometric isomorphism $\bm{\mathcal{F}}_{\!\wp}: L^2(M,\omega^\wp) \longrightarrow L^2((\mathbb{R}_0^+)^2,\bm{\rho}^\wp)$ whose inverse is given by
\begin{equation} \label{eq:coneconvSLcompact_Feigenexp_inv}
(\bm{\mathcal{F}}_{\!\wp}^{-1} \Phi)(x,z) = \int_{(\mathbb{R}_0^+)^2} \Phi(\lambda,\eta) \, w_{\lambda,\eta}^\wp(x,z) \, \bm{\rho}^\wp(d(\lambda,\eta)).
\end{equation}
The convergence of the integral in \eqref{eq:coneconvSLcompact_Feigenexp_dir} is understood with respect to the norm of $L^2((\mathbb{R}_0^+)^2,\bm{\rho}^\wp)$ and the convergence of the integral in \eqref{eq:coneconvSLcompact_Feigenexp_inv} is understood with respect to the norm of $L^2(M,\omega^\wp)$.

If $b$ is regular, entrance or exit, then $\bm{\rho}^\wp(\Lambda_1 \times \Lambda_2) = \sum_{\eta_k \in \Lambda_2} {1 \over \|\psi_{\eta_k}\|^2} \bm{\rho}_{\eta_k}^\wp(\Lambda_1)$, where $\bm{\rho}_\eta^\wp$ is the spectral measure of (the Neumann self-adjoint extension of) the Sturm-Liouville operator $\smash{\mathcal{G}_\eta}\rule[-0.05\baselineskip]{0pt}{0.78\baselineskip}$, and the eigenfunction expansion \eqref{eq:coneconvSLcompact_Feigenexp_dir}--\eqref{eq:coneconvSLcompact_Feigenexp_inv} reduces to
\begin{gather*}
\bm{\mathcal{F}}_{\!\wp}: L^2(M,\omega^\wp) \longrightarrow \bigoplus_{k=1}^\infty L^2(\mathbb{R}_0^+, \bm{\rho}_{\eta_k}^\wp), \qquad \bm{\mathcal{F}}_{\!\wp} h \equiv \bigl( (\bm{\mathcal{F}}_{\!\wp} h)(\bm{\cdot},\eta_1), (\bm{\mathcal{F}}_{\!\wp} h)(\bm{\cdot},\eta_2), \ldots \bigr) \\
(\bm{\mathcal{F}}_{\!\wp}^{-1} \{\varphi_k\})(x,z) = \sum_{k=1}^\infty {1 \over \|\psi_{\eta_k}\|^2} \int_{\mathbb{R}_0^+} \varphi_k(\lambda) \, w_{\lambda,\eta_k}^\wp(x,z) \, \bm{\rho}_{\eta_k}^\wp(d\lambda).
\end{gather*}
\end{proposition}

\begin{proof}
The result for $b$ regular, entrance or exit can be proved in a direct way using the same method as in Proposition \ref{prop:coneconv_Feigenexp}.

If $b$ is natural, start by considering the operator $\bm{\mathcal{G}}^\wp$ on the restricted domain $M_N = [0,N] \times [a,N]$, where $\max\{0,a\} < N < \infty$. Applying first the eigenfunction expansion of the Sturm-Liouville operator $\wp$ on the interval $[a,N]$ (with boundary condition $u'(N) = 0$) and then the eigenfunction expansion of the Sturm-Liouville operators $\mathcal{G}_\eta$ on $[0,N]$ (also with $u'(N) = 0$), we obtain a discrete eigenfunction expansion of the form
\begin{align*}
(\bm{\mathcal{F}}_{\wp,N\,} h)(\lambda_{\mathsmaller{k,N}},\eta_{\mathsmaller{k,N}}) & = \int_{M_N} h(x,z) \, w_{\lambda_{\mathsmaller{k,N}},\eta_{\mathsmaller{k,N}}}^\wp(x,z) \, \omega^\wp(d(x,z)) \\
(\bm{\mathcal{F}}_{\wp,N}^{-1} \{c_k\})(x,z) & = \sum_{k = 1}^\infty {c_k \over \| w_{\lambda_{\mathsmaller{k,N}},\eta_{\mathsmaller{k,N}}} \|^2} \, w_{\lambda_{\mathsmaller{k,N}},\eta_{\mathsmaller{k,N}}}(x,z).
\end{align*}
Using the techniques of \cite{browne1972}, one can show that in the limit $N \to \infty$ this discrete expansion gives rise to an eigenfunction expansion of the general form \eqref{eq:coneconvSLcompact_Feigenexp_dir}--\eqref{eq:coneconvSLcompact_Feigenexp_inv}, where $\bm{\rho}^\wp$ is the limiting spectral measure.
\end{proof}

As before, we will use the shorthand notation $\bm{\xi} = (x,z)$, $\bm{\xi_1} = (x_1, z_1)$, etc.\ for points of $M$.

\begin{proposition}[Product formula for $w_{\lambda,\eta}^\wp$] Suppose that there exists a family of measures $\{\pi_{z_1,z_2}^\wp\}_{z_1,z_2 \in I} \subset \mathcal{P}(I)$ such that
\begin{equation} \label{eq:coneconvSLcompact_prodform_1Dhyp}
\psi_\eta(z_1) \, \psi_\eta(z_2) = \int_I \psi_\eta \, d\pi_{z_1,z_2}^\wp \qquad (z_1, z_2 \in I, \; \eta \in \supp(\bm{\sigma})).
\end{equation}
Then for each $\eta \in \supp(\bm{\sigma})$ and $\bm{\xi_1}, \bm{\xi_2} \in M$ there exists a positive measure $\bm{\gamma}_{\eta,\bm{\xi_1},\bm{\xi_2}}^\wp$ on $M$ such that the product $w_{\lambda,\eta}^\wp(\bm{\xi_1}) \, w_{\lambda,\eta}^\wp(\bm{\xi_2})$ admits the integral representation
\begin{equation} \label{eq:coneconvSLcompact_prodform}
w_{\lambda,\eta}^\wp(\bm{\xi_1}) \, w_{\lambda,\eta}^\wp(\bm{\xi_2}) = \int_M w_{\lambda,\eta}^\wp(\bm{\xi_3})\, \bm{\gamma}_{\eta,\bm{\xi_1},\bm{\xi_2}}^\wp(d\bm{\xi_3}), \qquad \bm{\xi_1},\bm{\xi_2} \in M, \; \lambda \in \mathbb{C},\; \eta \in \supp(\bm{\sigma}).
\end{equation}
\end{proposition}

\begin{proof}
It is straightforward that $w_{\lambda,\eta}^\wp(x,z) = \psi_\eta(z) \, \zeta_\eta(x) \, \widetilde{v}_{\!\lambda,\eta}(x)$, where $\zeta_\eta(x) := \cosh\bigl(\sqrt{\eta} \int_0^x \! {dy \over A(y)}\bigr)$ and $\widetilde{v}_{\!\lambda,\eta}$ is the solution of 
\[
-{1 \over B_\eta}(B_\eta \widetilde{v}')' = \lambda \widetilde{v}, \qquad \widetilde{v}(0) = 1, \qquad (B_\eta \widetilde{v}')(0) = 0
\]
where $B_\eta(x) = A(x) \zeta_\eta(x)^2$. Arguing as in the proof of Proposition \ref{prop:coneconv_prodform}, we deduce that the product formula \eqref{eq:coneconvSLcompact_prodform} holds for the positive measures
\[
\bm{\gamma}_{\eta,\bm{\xi_1}, \bm{\xi_2}}^\wp(d\bm{\xi_3}) =  {\zeta_\eta(x_1) \zeta_\eta(x_2) \over \zeta_\eta(x_3)} \bm{\nu}_{\eta,\bm{\xi_1}, \bm{\xi_2}}^\wp(d\bm{\xi_3})
\]
where $\bm{\nu}_{\eta,\bm{\xi_1}, \bm{\xi_2}}^\wp := \pi_{x_1,x_2}^{[\eta]} \otimes \pi_{z_1,z_2}^\wp$ and $\pi_{x_1,x_2}^{[\eta]}$ is the measure of the product formula for $\widetilde{v}_{\lambda,\eta}$.
\end{proof}

Sufficient conditions for the existence of $\{\pi_{z_1,z_2}^\wp\}_{z_1,z_2 \in I} \subset \mathcal{P}(I)$ such that \eqref{eq:coneconvSLcompact_prodform_1Dhyp} holds have been determined by various authors \cite[pp.\ 311--314]{berezansky1998}, \cite[pp.\ 234--245]{bloomheyer1994}, \cite{sousaetal2019b}. We highlight the following examples:

\begin{example}
\textbf{(a)} If $I = [0,\pi/2]$ and the coefficients $\mathfrak{p} = \mathfrak{r}$ satisfy
\begin{equation} \label{eq:coneconvSLcompact_SLzexample1}
\mathfrak{p}\Bigl({\pi \over 2} - z\Bigr) = \mathfrak{p}(z), \qquad {\mathfrak{p}' \over \mathfrak{p}} \text{ decreasing on } \Bigl[0, {\pi \over 4}\Bigr], \qquad {\mathfrak{p}' \over \mathfrak{p}}(z) =  2\alpha_0 \cot(2z) + \alpha_1(z)
\end{equation}
where $\alpha_0 > 0$ and $\alpha_1 \in \mathrm{C}^\infty[0,\pi/2]$ satisfies $\alpha_1(0) = 0$, then there exists $\{\pi_{z_1,z_2}^\wp\}_{z_1,z_2 \in I} \subset \mathcal{P}(I)$ satisfying \eqref{eq:coneconvSLcompact_prodform_1Dhyp}. \\[-8pt]

\textbf{(b)} With $I$, $\alpha_0$ and $\alpha_1$ as above, the same is true if we replace \eqref{eq:coneconvSLcompact_SLzexample1} by
\[
\mathfrak{p}^{(2j+1)}\Bigl({\pi \over 2}\Bigr) = 0 \text{ for } j = 0,1,2,\ldots, \qquad {\mathfrak{p}' \over \mathfrak{p}} \text{ decreasing on } \Bigl[0, {\pi \over 2}\Bigr], \qquad {\mathfrak{p}' \over \mathfrak{p}}(z) =  \alpha_0 \cot(z) + \alpha_1(z). \vspace{2pt}
\]

\textbf{(c)} For $I = [a,\infty)$, let 
\[
\alpha(z) = \int_c^z \sqrt{\mathfrak{r}(\zeta) \over \mathfrak{p}(\zeta)} d\zeta, \qquad \alpha^{-1} \text{ its inverse function }, \qquad R(y) = \sqrt{\mathfrak{p}(\alpha^{-1}(y)) \, \mathfrak{r}(\alpha^{-1}(y))}
\]
where $a < c < \infty$ is a fixed point. If $\alpha(\infty) = \infty$ and the function $R' \over R$ is decreasing and nonnegative, then there exists $\{\pi_{z_1,z_2}^\wp\}_{z_1,z_2 \in I} \subset \mathcal{P}(I)$ satisfying \eqref{eq:coneconvSLcompact_prodform_1Dhyp}.
\end{example}

\begin{definition}
Suppose that there exists $\{\pi_{z_1,z_2}^\wp\}_{z_1,z_2 \in I} \subset \mathcal{P}(I)$ such that \eqref{eq:coneconvSLcompact_prodform_1Dhyp} holds, and let $\eta \in \supp(\bm{\sigma})$,\, $\lambda \geq 0$ and $\mu, \nu \in \mathcal{M}_{\mathbb{C}}(M)$. The measure
\[
(\mu \conv{\!\eta,\wp\!} \nu)(\bm{\cdot}) = \int_M \int_M \bm{\nu}_{\eta,\bm{\xi_1}, \bm{\xi_2}}^\wp(\bm{\cdot}) \, \mu(d\bm{\xi_1}) \, \nu(d\bm{\xi_2})
\]
is called the \emph{$\mathcal{G}_\eta$-convolution} of the measures $\mu$ and $\nu$. The functions
\[
(\bm{\mathcal{F}}_{\!\wp\,} \mu)(\lambda,\eta) = \int_M {w_{\lambda,\eta}^\wp(\bm{\xi}) \over \zeta_\eta(x)}\, \mu(d\bm{\xi}) \qquad \text{and} \qquad 
(\bm{\mathcal{T}}_{\!\eta,\wp}^\mu h)(\bm{\xi}) := \int_M h \, d(\delta_{\bm{\xi}} \conv{\eta,\wp} \mu)
\]
are, respectively, the \emph{$\bm{\mathcal{G}}^\wp$-Fourier transform} of the measure $\mu$ and the \emph{$\mathcal{G}_\eta$-translation} of a function $h$ by $\mu$.
\end{definition}

Unsurprisingly, the \emph{$\mathcal{G}_\eta$-convolution} shares many properties with the $\Delta_k$-convolution studied in the previous sections, among which the following:

\begin{proposition}
Assume that there exists $\{\pi_{z_1,z_2}^\wp\}_{z_1,z_2 \in I} \subset \mathcal{P}(I)$ such that \eqref{eq:coneconvSLcompact_prodform_1Dhyp} holds. Assume also that $e^{-t\mskip0.5\thinmuskip \bm{\pmb\cdot}} \! \in L^2(\mathbb{R}_0^+,\bm{\sigma})$ for all $t > 0$.
\begin{enumerate}[itemsep=0pt,topsep=4pt]
\item[\textbf{(a)}] For each $\eta \in \supp(\bm{\sigma})$, the space $(\mathcal{M}_\mathbb{C}(M),\conv{\!\eta,\wp\!})$, equipped with the total variation norm, is a commutative Banach algebra over $\mathbb{C}$ whose identity element is the Dirac measure $\delta_{(0,a)}$. Moreover, the subset $\mathcal{P}(M)$ is closed under the $\mathcal{G}_\eta$-convolution.

\item[\textbf{(b)}] $\bigl(\bm{\mathcal{F}}_{\!\wp}(\mu \conv{\!\eta,\wp\!} \nu)\bigr)(\lambda,\eta) = (\bm{\mathcal{F}}_{\!\wp\,}\mu)(\lambda,\eta) \ccdot (\bm{\mathcal{F}}_{\!\wp\,}\nu)(\lambda,\eta)$\: for all $\lambda \geq 0$ and $\eta \in \supp(\bm{\sigma})$.

\item[\textbf{(c)}] Each measure $\mu \in \mathcal{M}_\mathbb{C}(M)$ is uniquely determined by $\bm{\mathcal{F}}_{\!\wp\,} \mu$.
\end{enumerate}
Set $\Sigma := \supp(\bm{\sigma})$ if $I = [a,b]\,$ and set $\Sigma := \mathbb{R}_0^+$ if $I = [a,b)$. In the latter case, assume also that $\lim_{z \uparrow b} \psi_\eta(z) = 0$ for all $\eta > 0$. Then:
\begin{enumerate}[itemsep=0pt,topsep=4pt]
\item[\textbf{(d)}] Let $\{\mu_n\}$ be a sequence of measures belonging to $\mathcal{M}_+(M)$ whose $\bm{\mathcal{G}}^\wp$-Fourier transforms are such that
\[
(\bm{\mathcal{F}}_{\!\wp\,} \mu_n)(\lambda,\eta) \xrightarrow[\,n \to \infty\,]{} f(\lambda,\eta) \qquad \text{pointwise in } (\lambda,\eta) \in \mathbb{R}_0^+ \times \Sigma
\]
for some real-valued function $f$ such that
\[
\begin{cases}
f(\bm{\cdot},0) \text{ is continuous at a neighbourhood of zero} & \text{ if } b \text{ is regular or entrance} \\
f \text{ is continuous at a neighbourhood of } (0,0) & \text{ if } b \text{ is exit or natural}.
\end{cases}
\]
Then $\mu_n \warrow \mu$ for some measure $\mu \in \mathcal{M}_+(M)$ such that $\bm{\mathcal{F}}_{\!\wp\,} \mu \equiv f$.

\item[\textbf{(e)}] For each $\eta \in \Sigma$ the mapping $(\mu,\nu) \mapsto \mu \conv{\!\eta,\wp\!} \nu$ is continuous in the weak topology.

\item[\textbf{(f)}] If $h \in \mathrm{C}_\mathrm{b}(M)$ (respectively $\mathrm{C}_0(M)$) then $\bm{\mathcal{T}}_{\!\eta,\wp}^\mu h \in \mathrm{C}_\mathrm{b}(M)$ (resp.\ $\mathrm{C}_0(M)$) for all $\mu  \in \mathcal{M}_\mathbb{C}(M)$.

\item[\textbf{(g)}] Let $1 \leq p \leq \infty$, $\mu \in \mathcal{M}_+(M)$ and $h \in L_{\eta,\wp}^p := L^p\bigl(M, B_\eta(x) dx \, \mathfrak{r}(z)dz\bigr)$. The $\mathcal{G}_\eta$-translation $(\bm{\mathcal{T}}_{\!\eta,\wp}^\mu h)(x)$ is a Borel measurable function of $x \in M$, and we have
\[
\|\bm{\mathcal{T}}_{\!\eta,\wp}^\mu h\|_{L_{\eta,\wp}^p} \leq \|\mu\| \ccdot \|h\|_{L_{\eta,\wp}^p}.
\]
\end{enumerate}
\end{proposition}

\begin{proof}
We will only prove \textbf{\emph{(c)}}, \textbf{\emph{(d)}} and \textbf{\emph{(g)}} because the proof of parts \textbf{\emph{(a)}}--\textbf{\emph{(b)}} and \textbf{\emph{(e)}}--\textbf{\emph{(f)}} are analogous to those of the corresponding results for the $\Delta_k$-convolution and translation. \\[-8pt]

\textbf{\emph{(c)}} Let $\mu$ be such that $(\bm{\mathcal{F}}_{\!\wp\,}\mu)(\lambda,\eta) = 0$ for $\lambda \geq 0$ and $\eta \in \supp(\bm{\sigma})$. For $f \in \mathrm{C}_\mathrm{c}(\mathbb{R}_0^+)$ and $g \in \mathrm{C}_\mathrm{c}(I)$ we have
\begin{align*}
\int_M f(x) \, g(z) \, \mu(d(x,z)) & = \lim_{t \downarrow 0} \int_M f(x) \int_{\mathbb{R}_0^+} e^{-t\eta} (\mathcal{J}_{\wp} g)(\eta) \, \psi_\eta(z) \, \bm{\sigma}(d\eta) \, \mu(d(x,z)) \\
& = \lim_{t \downarrow 0} \int_{\mathbb{R}_0^+} e^{-t\eta} \, (\mathcal{J}_{\wp} g)(\eta) \int_M f(x) \widehat{\mu}_\eta(dx) \, \bm{\sigma}(d\eta) \\
& = 0
\end{align*}
where the measure $\widehat{\mu}_\eta$ is defined as $\int_{\mathbb{R}_0^+} f \, d\widehat{\mu}_\eta = \int_M f(x) \psi_\eta(z) \mu\bigl(d(x,z)\bigr)$, and the argument in the proof of Proposition \ref{prop:coneconv_fourmeas_props}(ii) yields the last equality. The conclusion that $\mu = 0$ also follows in the same way. \\[-8pt]

\textbf{\emph{(d)}} For $I = [a,b]$, the proof is identical to that of Proposition \ref{prop:coneconv_fourmeas_props}(iv). For the case $I = [a,b)$, fix $\eps > 0$ and notice that the hypothesis (together with the integral mean value theorem) ensures that we can pick $\delta > 0$ such that
\[
\biggl| {2 \over \delta} \int_{V_\delta} \bigl( f(0,0) - f(\lambda,\eta) \bigr) \mskip 0.5\thinmuskip d(\lambda,\eta) \biggr| < \eps
\]
where $V_\delta := \{(\lambda,\eta) \in (\mathbb{R}_0^+)^2 \mid \lambda^2 + \eta^2 < \delta\}$. In addition, we know that $\lim_{z \to \infty} \psi_\eta(z) = 0$ and $\lim_{x \to \infty} \widetilde{v}_{\lambda,\eta}(x) = 0$ for all $\lambda,\eta > 0$ (Lemma \ref{lem:SLappendix_solprops}(b)), and thus there exist $0 < \beta_1 < \infty$ and $a < \beta_2 < b$ such that 
\[
\int_{V_\delta} \bigl( 1 - \psi_\eta(z) \widetilde{v}_{\lambda,\eta}(x)\bigr) \mskip 0.5\thinmuskip d(\lambda,\eta) \geq {\delta \over 2} \qquad \text{for all } (x,z) \in M \setminus [0,\beta_1] \times [a,\beta_2].
\]
Hence
\begin{align*}
\mu_n\bigl(M \setminus [0,\beta_1] \times [a,\beta_2]\bigr) & \leq {2 \over \delta} \int_{M \setminus [0,\beta_1] \times [a,\beta_2]} \int_{V_\delta} \bigl( 1 - \psi_\eta(z) \widetilde{v}_{\lambda,\eta}(x)\bigr) \mskip 0.5\thinmuskip d(\lambda,\eta) \, \mu_n(d\bm{\xi}) \\
& \leq {2 \over \delta} \int_{V_\delta} \bigl( (\bm{\mathcal{F}}_{\!\wp\,} \mu_n)(0,0) - (\bm{\mathcal{F}}_{\!\wp\,} \mu_n)(\lambda,\eta) \bigr) \mskip 0.5\thinmuskip d(\lambda,\eta)
\end{align*}
and consequently
\begin{align*}
\limsup_{n \to \infty} \mu_n\bigl(M \setminus [0,\beta_1] \times [a,\beta_2]\bigr) & \leq {2 \over \delta} \limsup_{n \to \infty}\! \int_{V_\delta} \bigl((\bm{\mathcal{F}}_{\!\wp\,} \mu_n)(0,0) - (\bm{\mathcal{F}}_{\!\wp\,} \mu_n)(\lambda,\eta)\bigr) \mskip 0.5\thinmuskip d(\lambda,\eta) \\
& = {2 \over \delta} \int_{V_\delta} \bigl(f(0,0) - f(\lambda,\eta)\bigr)\mskip 0.5\thinmuskip d(\lambda,\eta) \\
& < \eps.
\end{align*}
This shows that $\{\mu_n\}$ is a tight sequence of measures. Applying Prokhorov's theorem and the reasoning from the proof of Proposition \ref{prop:coneconv_fourmeas_props}(iv), the desired conclusion follows. \\[-8pt]

\textbf{\emph{(g)}} It suffices to prove that the generalized translation operator $(\mathcal{T}_\wp^{z_0} g)(z) := \int_I g \, d\pi_{z,z_0}^\wp$ is a contraction operator on $L^p(I,\mathfrak{r})$ (we then obtain the result by arguing as in Proposition \ref{prop:coneconv_transl_props}(c)). Let $\convdiam{\wp}$ be the convolution determined by the product formula \eqref{eq:coneconvSLcompact_prodform_1Dhyp} and let $\{\varkappa_t\}_{t \geq 0}$ the $\convdiam{\wp}$-Gaussian convolution semigroup generated by $\wp$. For $g \in \mathrm{C}_\mathrm{c}(I)$ we have
\begin{align*}
\int_I g(z) \bigl(\mu_t \convdiam{\wp} \delta_{z_1} \convdiam{\wp} \delta_{z_2}\bigr)(dz) & = \lim_{s \downarrow 0} \int_I (\mathcal{T}_\wp^{\varkappa_s} g)(z) \bigl(\mu_t \convdiam{\wp} \delta_{z_1} \convdiam{\wp} \delta_{z_2}\bigr)(dz) \\
& = \lim_{s \downarrow 0} \int_{\mathbb{R}_0^+} (\mathcal{J}_\wp g)(\eta) \, e^{-(t+s)\eta\, } \psi_\eta(z_1) \, \psi_\eta(z_2) \, \bm{\sigma}(d\eta) \\
& = \int_I g(z) \, q_t(z_1,z_2,z) \, \mathfrak{r}(z) dz
\end{align*}
where $q_t(z_1,z_2,z_3) := \int_{\mathbb{R}_0^+} e^{-t\eta} \, \psi_\eta(z_1) \, \psi_\eta(z_2) \, \psi_\eta(z_3) \, \bm{\sigma}(d\eta)$. Consequently, for $f,g \in \mathrm{C}_\mathrm{c}(I)$ we have
\[
\Bigl\langle \mathcal{T}_\wp^{\raisebox{3pt}{\scalebox{0.7}{$\mu_t \convdiam{\wp} \delta_{z_0}$}}} f,\, g \Bigr\rangle_{\!L^2(I,\mathfrak{r})\!} = \int_a^b \int_a^b g(z_1) q_t(z_0,z,z_1) \mathfrak{r}(z_1) dz_1 \, f(z) \mathfrak{r}(z) dz = \Bigl\langle f,\, \mathcal{T}_\wp^{\raisebox{3pt}{\scalebox{0.7}{$\mu_t \convdiam{\wp} \delta_{z_0}$}}} g \Bigr\rangle_{\!L^2(I,\mathfrak{r})\!}
\]
and by continuity it follows that $\bigl\langle \mathcal{T}_\wp^{z_0} f,\, g \bigr\rangle_{\!L^2(I,\mathfrak{r})\!} = \bigl\langle f,\, \mathcal{T}_\wp^{z_0} g \bigr\rangle_{\!L^2(I,\mathfrak{r})\!}$. Since $\mathcal{T}_\wp^{z_0}$ is clearly a contraction on $L^\infty(I,\mathfrak{r})$, by duality we find that it is also a contraction on $L^1(I,\mathfrak{r})$; hence, by interpolation, it is a contraction on $L^p(I,\mathfrak{r})$ for $1 \leq p \leq \infty$.
\end{proof}

Notions such as infinite divisibility and convolution semigroups with respect to the $\mathcal{G}_{\eta}$-convolution can also be defined like in the previous section, giving rise to a Lévy-Khintchine type representation and to Feller semigroups on $\mathrm{C}_0(M)$. The details are left to the reader.

\begin{remark} \label{rmk:coneconv_prodform_multiparam}
The above result on the existence of a product formula for the functions $w_{\lambda,\eta}^\wp$ can be interpreted in the context of the theory of multiparameter Sturm-Liouville spectral problems. 

First we recall some known results. Consider the system of Sturm-Liouville equations
\begin{equation} \label{eq:coneconv_prodform_multiparam_eq1}
- (p_m u_m')'(x_m) + (q_m u_m)(x_m) = \sum_{n=1}^N \lambda_n (r_{mn} u_m)(x_m) \qquad (m = 1,\ldots, N)
\end{equation}
where $N \in \mathbb{N}$ and $a_m \leq x_m \leq b_m$, together with boundary conditions at the endpoints $a_m$ and $b_m$ of the form 
\begin{equation} \label{eq:coneconv_prodform_multiparam_eq2}
u_m(a_m) \cos(\vartheta_m) = u_m'(a_m) \sin(\vartheta_m), \qquad u_m(b_m) \cos(\vartheta_m') = u_m'(b_m) \sin(\vartheta_m') \qquad (m = 1, \ldots, N).
\end{equation}
Let us assume that the intervals $I_m = [a_m,b_m]$ are bounded, the functions $p_m$, $q_m$, $r_{mn}$ are sufficiently well-behaved and $r(x) = \det\{r_{mn}(x_m)\} > 0$ for $x = (x_1,\ldots,x_N) \in \prod_{m=1}^N I_m$. If $\lambda = (\lambda_1, \ldots, \lambda_N)$ is chosen such that for each $m$ there exists a nontrivial solution $u_m(x_m;\lambda)$ of \eqref{eq:coneconv_prodform_multiparam_eq1}--\eqref{eq:coneconv_prodform_multiparam_eq2}, then the function $u(x;\lambda) = \prod_{i=1}^N u_m(x_m;\lambda)$ is said to be an eigenfunction of the system \eqref{eq:coneconv_prodform_multiparam_eq1}--\eqref{eq:coneconv_prodform_multiparam_eq2} corresponding to the eigenvalue $\lambda$. 

By the completeness theorem for multi-parameter eigenvalue problems \cite{faierman1969}, the following Fourier-like expansion holds:
\begin{equation} \label{eq:coneconv_rmkmultiparam_Finv}
h(x) = \sum_k (\bm{F}h)(\lambda^{(k)}) \, u(x;\lambda^{(k)}),
\end{equation}
where 
\[ (\bm{F}h)(\lambda^{(k)}) := \int_a^b h(x)\, u(x;\lambda^{(k)})\, r(x) dx,
\]
$\lambda^{(k)}$ are the eigenvalues of \eqref{eq:coneconv_prodform_multiparam_eq1}--\eqref{eq:coneconv_prodform_multiparam_eq2}, and $\int_a^b = \int_{a_1}^{b_1} \ldots \int_{a_N}^{b_N}$.

Similar results have been established for (singular) systems where some of the intervals $[a_m,b_m]$ are unbounded; in this case, the sum in \eqref{eq:coneconv_rmkmultiparam_Finv} is, in general, replaced by a Stieltjes integral with respect to a spectral function \cite{browne1972,browne1977} However, compared to one-dimensional Sturm-Liouville operators, much less is known regarding the spectral properties of such singular systems \cite{atkinsonmingarelli2011,sleeman2008}.

Returning to our eigenfunction expansion \eqref{eq:coneconvSLcompact_Feigenexp_dir}--\eqref{eq:coneconvSLcompact_Feigenexp_inv} for the elliptic operator $\bm{\mathcal{G}}^\wp$, we can now reinterpret it as a Fourier-like expansion for the system of differential equations \eqref{eq:coneconv_prodform_multiparam_eq1} with $N = 2$, $x_1 \in \mathbb{R}_0^+$, $x_2 \in [a,b]$, $\lambda_1 = \lambda$, $\lambda_2 = \eta$, $p_1 = r_{11} = A$, $p_2 = p$, $r_{22} = r$, $r_{12} = {1 \over A}$ and $q_1 = q_2 = r_{21} = 0$.

In the theory of product formulas and convolutions associated with one-dimensional Sturm-Liouville equations $- (p u')' + q u = \lambda r u$, a crucial requirement is that the measures of the product formula should not depend on the spectral parameter $\lambda$, cf.\ e.g.\ \cite{chebli1995,koornwinder1984,sousaetal2020}. Similarly, the measures of product formula \eqref{eq:coneconvSLcompact_prodform} for the generalized eigenfunctions $w_{\lambda,\eta}^\wp$ do not depend on one of the spectral parameters (the measures $\gamma_{\eta,\bm{\xi_1}, \bm{\xi_2}}$ are independent of $\lambda$); this is a fundamental property which (as we saw above) enables us to develop the theory of $\mathcal{G}_\eta$-convolutions. This suggests that the natural way to introduce the notion of a product formula for a general Sturm-Liouville system \eqref{eq:coneconv_prodform_multiparam_eq1} (regular or singular, with suitable boundary conditions) is as follows:\\[2pt]
\emph{Let $1 \leq s \leq N$. The system \eqref{eq:coneconv_prodform_multiparam_eq1} is said to admit a \emph{$(\lambda_1,\ldots, \lambda_s)$-product formula} if for each $x^{(1)}, x^{(2)} \in I := \prod_{m=1}^N I_m$ there exists a positive measure $\bm{\gamma}_{x^{(1)}, x^{(2)}}^{\lambda_{s+1}, \ldots, \lambda_N}$ on $I$ such that the product $u(x^{(1)};\lambda)\, u(x^{(2)};\lambda)$ admits the representation
\begin{equation} \label{eq:coneconv_multiparam_prodformdef}
u(x^{(1)};\lambda)\, u(x^{(2)};\lambda) = \int_I u(x;\lambda) \, \bm{\gamma}_{x^{(1)}, x^{(2)}}^{\lambda_{s+1}, \ldots, \lambda_N}(dx), \qquad \lambda_1, \ldots, \lambda_N \geq 0.
\end{equation}}%

As far as we are aware, this paper is the first to establish the existence of product formulas of the form \eqref{eq:coneconv_multiparam_prodformdef} for a general family of \emph{nontrivial} Sturm-Liouville systems of the form \eqref{eq:coneconv_prodform_multiparam_eq1}; here, the word `nontrivial' means that $r_{mn} \neq 0$ for some $m \neq n$. (The few previous results on product formulas of the form \eqref{eq:coneconv_multiparam_prodformdef} only apply to very special cases where the measure is independent of all spectral parameters, see \cite{nessibitrimeche1997,trimeche1995}.) Developing a general theory of product formulas for nontrivial systems of Sturm-Liouville equations is an interesting problem which is left open for further investigation.
\end{remark}

\appendix
\setcounter{section}{-1}
\section{Generalized convolution structures for Sturm-Liouville operators}  \label{sec:appendix_slconv}
\renewcommand{\thesection}{A}

One-dimensional convolutions associated with Sturm-Liouville operators have been extensively studied in recent work of the authors \cite{sousaetal2020,sousaetal2019b}. For convenience, in this appendix we summarize some fundamental results on this topic.

Let 
\[
\ell(u)(x) := {1 \over r(x)} \Bigl( -(pu')'(x) + q(x)u(x) \Bigr), \qquad x \in (a,b) \subset \mathbb{R}
\]
be a Sturm-Liouville expression whose coefficients are such that $p, r > 0$ on $(a,b)$, $p, r$ are locally absolutely continuous and $q$ is locally integrable on $(a,b)$.

Consider the integrals
\begin{align*}
& I_a = \int_a^c \int_a^y {dx \over p(x)} \bigl(r(y) + q(y)\bigr) dy, && \hspace{-.08\linewidth} J_a = \int_a^c \int_y^c {dx \over p(x)} \bigl(r(y) + q(y)\bigr) dy \\
& I_b = \int_c^b \int_y^b {dx \over p(x)} \bigl(r(y) + q(y)\bigr) dy, && \hspace{-.08\linewidth} J_b = \int_c^b \int_c^y {dx \over p(x)} \bigl(r(y) + q(y)\bigr) dy.
\end{align*}
The Feller boundary classification for $\ell$ is as follows: the endpoint $e \in \{a,b\}$ is said to be:
\begin{equation} \label{eq:bkg_diffusion_fellerBC}
\begin{array}{llll}
\text{\emph{regular}} & \text{if} & I_e < \infty, & \!\!\!J_e < \infty; \\
\text{\emph{exit}} & \text{if} & I_e < \infty, & \!\!\!J_e = \infty;
\end{array}
\qquad\quad
\begin{array}{llll}
\text{\emph{entrance}} & \text{if} & I_e = \infty, & \!\!\!J_e < \infty; \\
\text{\emph{natural}} & \text{if} & I_e = \infty, & \!\!\!J_e = \infty
\end{array}
\end{equation}
(the classification is independent of the choice of $c$). Throughout the appendix we assume that the endpoint $a$ is regular or entrance.

\begin{lemma} \label{lem:SLappendix_boundaryvaluepb}
The boundary value problem
\[
\ell(v) = \lambda v \quad (a < x < b, \; \lambda \in \mathbb{C}), \qquad\; v(a) = 1, \qquad\; (pv')(a) = 0
\]
has a unique solution $v_\lambda(\bm{\cdot})$. Moreover, $\lambda \mapsto v_\lambda(x)$ is, for fixed $x$, an entire function of exponential type.
\end{lemma}

\begin{proof}
See \cite[Lemma 2.1]{sousaetal2019b}. (The result of \cite{sousaetal2019b} is stated for the case $q \equiv 0$, but the general case can be proved in a similar way.)
\end{proof}

\begin{proposition} \label{prop:SLappendix_eigenexp}
The operator $\mathcal{L}:\mathcal{D}(\mathcal{L}) \longrightarrow L^2(r) \equiv L^2\bigl((a,b),r(z)dz\bigr)$, where
\begin{align*}
\mathcal{D}(\mathcal{L}) & = \begin{cases}
\bigl\{ u \in L^2(r) \bigm| u, u' \!\in \mathrm{AC}_{\mathrm{loc}}(a,b), \; \ell(u) \in L^2(r), \; (pu')(a) = 0 \bigr\} \hspace{-2mm} & \text{ if } b \text{ is exit or natural}\\[1pt]
\bigl\{ u \in L^2(r) \bigm| u, u' \!\in \mathrm{AC}_{\mathrm{loc}}(a,b), \; \ell(u) \in L^2(r), \; (pu')(a) = (pu')(b) = 0 \bigr\} \hspace{-2mm} & \text{ otherwise}\\
\end{cases} \\[2pt]
\mathcal{L} u & = \ell(u), \qquad u \in \mathcal{D}(\mathcal{L})
\end{align*}
is a positive self-adjoint operator. There exists a locally finite positive Borel measure $\bm{\rho}_\ell$ on $\mathbb{R}_0^+$ such that the integral operator $\mathcal{F}_\ell: L^2(r) \longrightarrow L^2(\mathbb{R}_0^+, \bm{\rho}_\ell)$ defined by
\begin{equation} \label{eq:SLappendix_eigenexp_dir}
(\mathcal{F}_{\ell\,} f)(\lambda) := \int_a^b f(x) \, v_\lambda(x) \, r(x) dz
\end{equation}
is an isometric isomorphism whose inverse is given by
\begin{equation} \label{eq:SLappendix_eigenexp_inv}
(\mathcal{F}_\ell^{-1} \varphi)(x) = \int_{\mathbb{R}_0^+} \varphi(\lambda) \, v_\lambda(x) \, \bm{\rho}_\ell(d\lambda).
\end{equation}
(The convergence of the integrals above is understood with respect to the norm of $L^2(\mathbb{R}_0^+, \bm{\rho}_\ell)$ and $L^2(r)$ respectively.) The operator $\mathcal{F}_\ell$ is a spectral representation of $\mathcal{L}$, i.e.\ we have
\begin{align*}
& \mathcal{D}(\mathcal{L}) = \biggl\{f \in L^2(r) \biggm| \int_{\mathbb{R}_0^+} \lambda^2 \bigl|(\mathcal{F}_\ell f)(\lambda)\bigr|^2 \bm{\rho}_\ell(d\lambda) < \infty \biggr\}\\[2pt]
& \bigl(\mathcal{F}_\ell (\mathcal{L} f)\bigr) (\lambda) = \lambda \ccdot (\mathcal{F}_\ell f)(\lambda), \qquad\quad f \in \mathcal{D}(\mathcal{L}).
\end{align*}
Moreover, if $f \in \mathcal{D}(\mathcal{L})$ then $f(x) = \int_{\mathbb{R}_0^+} (\mathcal{F}_{\ell\,}f)(\lambda) \, v_\lambda(x) \, \bm{\rho}_\ell(d\lambda)$ for all $x \in (a,b)$, where the integral converges absolutely and locally uniformly.

If $b$ is regular, entrance or exit, then $\mathcal{L}$ has purely discrete spectrum $0 = \lambda_1 < \lambda_2 \leq \ldots \to \infty$, and the isomorphism \eqref{eq:SLappendix_eigenexp_dir}--\eqref{eq:SLappendix_eigenexp_inv} reduces to 
\begin{align*}
\mathcal{F}_\ell: L^2(r) \longrightarrow \ell_2\bigl(\tfrac{1}{\| v_{\lambda_k} \|^2}\bigr), \qquad \mathcal{F}_{\ell\,} f \equiv \bigl( (\mathcal{F}_{\ell\,} f)(\lambda_1), (\mathcal{F}_{\ell\,} f)(\lambda_2), \ldots \bigr), \qquad (\mathcal{F}_\ell^{-1} \{c_k\})(x) = \sum_{k = 1}^\infty {c_k \, v_{\lambda_k}(x) \over \| v_{\lambda_k} \|^2}
\end{align*}
where $\ell_2\bigl(\frac{1}{\| v_{\lambda_k} \|^2}\bigr)$ denotes the weighted sequence space whose norm is $\|\{c_k\}\| = \bigl(\sum_{k=1}^\infty {|c_k|^2 \over \| v_{\lambda_k} \|^2}\bigr)^{1/2}$.
\end{proposition}

\begin{proof}
See \cite[Proposition 2.5 and Lemma 2.6]{sousaetal2019b}, \cite[Section 5]{linetsky2004a}.
\end{proof}

\begin{proposition} \label{prop:SLappendix_semigr_heatkern}
The self-adjoint operator $-\mathcal{L}$ is the generator of a Markovian semigroup $\{e^{-t\mathcal{L}}\}_{t \geq 0}$ on $L^2(r)$. For $t>0$, the operators $e^{-t\mathcal{L}}$ are given by
\begin{equation} \label{eq:SLappendix_semigr}
(e^{-t\mathcal{L}} h)(x) = \int_a^b h(y)\, p_\ell(t,x,\xi)\, r(\xi)d\xi \qquad \bigl(h \in L^2(r), \; x \in (a,b)\bigr)
\end{equation}
where the kernel is defined by the integral
\begin{equation} \label{eq:SLappendix_heatkern}
p_\ell(t,x_1,x_2) = \int_{\mathbb{R}_0^+} e^{-t\lambda} \, v_\lambda(x_1) \, v_\lambda(x_2)\, \bm{\rho}_\ell(d\lambda) \qquad \bigl(t > 0, \; x_1,x_2 \in (a,b)\bigr).
\end{equation}
The right hand side of \eqref{eq:SLappendix_heatkern} is (for fixed $t>0$) absolutely and uniformly convergent on compact squares of $(a,b) \times(a,b)$.

Suppose also that $a$ is regular and $b$ is exit or natural. Then the integral in \eqref{eq:SLappendix_heatkern} converges absolutely and uniformly on compact squares of $[a,b) \times [a,b)$. Moreover, the restriction of $e^{-t\mathcal{L}}$ to $L^2(r) \cap \mathrm{C}_0[a,b)$ extends into a strongly continuous contraction semigroup on $\mathrm{C}_0[a,b)$ which can be represented by the right hand side of \eqref{eq:SLappendix_semigr}, which is convergent for all $h \in \mathrm{C}_0[a,b)$ and $x \in [a,b)$.
\end{proposition}

\begin{proof}
See \cite[Proposition 2.7]{sousaetal2019b}, \cite{fukushima2014}.
\end{proof}

Next we restrict our attention to the case $q \equiv 0$ and state some further properties of the generalized eigenfunctions $v_\lambda(x)$.

\begin{lemma} \label{lem:SLappendix_solprops}
\textbf{(a)} If $q \equiv 0$ and $x \mapsto p(x) r(x)$ is an increasing function, then $|v_\lambda(x)| \leq 1$ for all $a \leq x < b$ and $\lambda \geq 0$. \\[-9pt]

\noindent\textbf{(b)} Let $S(\xi) := \sqrt{p(\gamma^{-1}(\xi)) \, r(\gamma^{-1}(\xi))}$, where $\gamma(x) = \int_c^x\! \smash{\sqrt{r(y) \over p(y)}} dy$ and $\gamma^{-1}$ is its inverse function. (Here $c \in (a,b)$ is a fixed point; if $\sqrt{r(y) \over p(y)}$ is integrable near $a$, then we may also take $c=a$.) Assume that \vspace{-2pt}
\begin{equation} \label{eq:SLappendix_mainassump}
\begin{minipage}{0.84\linewidth}
$q \equiv 0$,\, $\gamma(b) = \int_c^b\! \smash{\sqrt{r(y) \over p(y)}} dy = \infty$,\, and there exists $\eta \in \mathrm{C}^1(\gamma(a),\infty)$ such that $\eta \geq 0$, the functions $\bm{\phi}_\eta := {S' \over S} - \eta\rule[-0.05\baselineskip]{0pt}{\baselineskip}$, $\:\bm{\psi}_\eta := {1 \over 2} \eta' - {1 \over 4} \eta^2 + {S' \over 2S} \ccdot \eta$ are both decreasing on $(\gamma(a),\infty)$ and $\bm{\phi}_\eta$ satisfies $\rule[-0.05\baselineskip]{0pt}{0.85\baselineskip}\lim_{\xi \to \infty} \bm{\phi}_\eta(\xi) = 0$.
\end{minipage}
\end{equation}
Then the following assertions are equivalent:
\begin{itemize}[itemsep=-0.5pt,topsep=2pt,leftmargin=0.7cm]
\item $\lim_{x \uparrow b} p(x)r(x) = \infty$;
\item $\lim_{x \uparrow b} v_\lambda(x) = 0$ for all $\lambda > 0$.
\end{itemize}
\end{lemma}

\begin{proof}
See \cite[Lemma 2.3 and Proposition 3.6]{sousaetal2019b}.
\end{proof}

\begin{theorem}[Product formula for $v_\lambda$] \label{thm:SLappendix_prodform}
Assume that \eqref{eq:SLappendix_mainassump} holds. Then there exists a family of measures $\{\pi_{x_1,x_2}\}_{x_1,x_2 \in [a,b)} \subset \mathcal{P}[a,b)$ such that we have
\[
v_\lambda(x_1) \, v_\lambda(x_2) = \int_{[a,b)} v_\lambda(x_3)\, \pi_{x_1,x_2}(dx_3) \qquad \text{for all }\, x_1, x_2 \in [a,b), \; \lambda \in \mathbb{C}.
\]
If $p \equiv r$ and $a = \gamma(a) = 0$, then $\supp(\pi_{x_1,x_2}) \subset [|x_1-x_2|,x_1+x_2]$.
\end{theorem}

\begin{proof}
See \cite[Section 4 and Subsection 5.3]{sousaetal2019b}.
\end{proof}

\begin{proposition} \label{prop:SLappendix_fourmeas_props}
Assume that \eqref{eq:SLappendix_mainassump} holds. Let $\mathcal{F}_\ell$ be the \emph{$\ell$-Fourier transform of measures} defined by
\begin{equation} 	\label{eq:SLappendix_fourmeas_def}
(\mathcal{F}_{\ell\,} \mu)(\lambda) := \int_{[a,b)} v_\lambda(x) \, \mu(dx) \qquad (\mu \in \mathcal{P}[a,b),\, \lambda \geq 0).
\end{equation}
Then:
\begin{enumerate}[itemsep=0pt,topsep=4pt]
\item[\textbf{(i)}] $\mathcal{F}_{\ell\,} \mu$ is continuous on $\mathbb{R}_0^+$. Moreover, if the family $\{\mu_j\} \subset \mathcal{M}_\mathbb{C}[a,b)$ is tight and uniformly bounded, then $\{\mathcal{F}_{\ell\,} \mu_j\}$ is equicontinuous on $\mathbb{R}_0^+$.

\item[\textbf{(ii)}] Let $\mu_1, \mu_2 \in \mathcal{M}_\mathbb{C}[a,b)$. If $\mathcal{F}_{\ell\,} \mu_1 \equiv \mathcal{F}_{\ell\,} \mu_2$, then $\mu_1 = \mu_2$.

\item[\textbf{(iii)}]
Let $\{\mu_n\} \subset \mathcal{M}_+[a,b)$, $\mu \in \mathcal{M}_+[a,b)$, and suppose that $\mu_n \warrow \mu$. Then
\[
\mathcal{F}_{\ell\,} \mu_n \xrightarrow[\,n \to \infty\,]{} \mathcal{F}_{\ell\,} \mu \qquad \text{uniformly on compact sets.}
\]
\end{enumerate}
\end{proposition}

\begin{proof}
See \cite[Proposition 5.2]{sousaetal2019b}.
\end{proof}

The \emph{$\ell$-convolution} and the \emph{$\ell$-translation operator} are respectively defined by
\begin{align*}
(\mu \convdiam{\ell} \nu)(d\xi) & := \int_{[a,b)} \int_{[a,b)} \pi_{x,y}(d\xi) \, \mu(dx) \, \nu(dy), & \hspace{-10em} \mu, \nu \in \mathcal{M}_\mathbb{C}[a,b) \\
(\mathcal{T}_\ell^\mu f)(x) & := \int_{[a,b)} f \, d(\delta_x \convdiam{\ell} \mu), & \hspace{-10em} \mu \in \mathcal{M}_\mathbb{C}[a,b), \; x \in [a,b), \; f \in L^p(r)
\end{align*}
where $L^p(r) \equiv L^p\bigl((a,b),r(z)dz\bigr)$\, ($1 \leq p \leq \infty$).

\begin{proposition} \label{prop:SLappendix_convtransl_props}
Assume that \eqref{eq:SLappendix_mainassump} holds. \\[2pt]
\textbf{(a)} $\mu = \mu_1 \convdiam{\ell} \mu_2$ if and only if $(\mathcal{F}_{\ell\,}\mu)(\lambda) = (\mathcal{F}_{\ell\,}\mu_1)(\lambda) \ccdot (\mathcal{F}_{\ell\,}\mu_2)(\lambda)$ for all $\lambda \geq 0$ \quad $(\mu, \mu_1, \mu_2 \in \mathcal{M}_\mathbb{C}[a,b))$.
\\[-9pt]

\noindent\textbf{(b)} The $\convdiam{\ell}$ convolution is weakly continuous: if $\mu_n \warrow \mu$ and $\nu_n \warrow \nu$, then $\mu_n \convdiam{\ell} \nu_n \warrow \mu \convdiam{\ell} \nu$. \\[-9pt]

\noindent\textbf{(c)} If $f \in \mathrm{C}_\mathrm{c}^2[a,b)$ with $f' \in \mathrm{C}_\mathrm{c}(a,b)$, then 
\[
\int_{[a,b)} f \: d(\delta_x \convdiam{\ell} \mu) = \int_{\mathbb{R}_0^+} (\mathcal{F}_{\ell\,} f)(\lambda) \, (\mathcal{F}_{\ell\,} \mu)(\lambda) \, v_\lambda(x) \, \bm{\rho}_\ell(d\lambda) \qquad \text{for all } \mu \in \mathcal{M}_\mathbb{C}[a,b),\: x \in (a,b).
\]

\noindent\textbf{(d)} Suppose that $\lim_{x \uparrow b} p(x) r(x) = \infty$, and let $\mu \in \mathcal{M}_\mathbb{C}[a,b)$. Then $\delta_x \convdiam{\ell} \mu \varrow \bm{0}$ as $x \uparrow b$, where $\bm{0}$ is the zero measure. \\[-9pt]

\noindent\textbf{(e)} Let $1 \leq p \leq \infty$ and $\mu \in \mathcal{M}_+[a,b)$. The $\ell$-translation $\mathcal{T}_\ell^\mu$ is a bounded operator on $L^p(r)$ such that 
\[
\|\mathcal{T}_\ell^\mu f\|_{L^p(r)} \leq \|\mu\| \ccdot \|f\|_{L^p(r)} \qquad \text{for all } f \in L^p(r).
\]

\noindent\textbf{(f)} Let $1 \leq p_1, p_2 \leq \infty$ with ${1 \over p_1} + {1 \over p_2} \geq 1$, and let $f \in L^{p_1\!}(r)$, $g \in L^{p_2\!}(r)$. Then the $\ell$-convolution
\[
(f \convdiam{\ell} g)(x) := \int_a^b (\mathcal{T}_\ell^y f)(x)\, g(y)\, r(y) dy
\]
(where $\mathcal{T}_\ell^y \equiv \mathcal{T}_\ell^{\delta_y}$) is well-defined and satisfies
\[
\| f \convdiam{\ell} g \|_{L^s(r)} \leq \| f \|_{L^{p_1\!}(r)} \| g \|_{L^{p_2\!}(r)}, \qquad \text{where }\, s = {1 \over 1/p_1 + 1/p_2 - 1}.
\]
\end{proposition}

\begin{proof}
See \cite[Corollary 5.2 and Proposition 6.4]{sousaetal2020}, \cite[Proposition 5.4 and Lemma 5.5]{sousaetal2019b}.
\end{proof}

\begin{proposition} \label{prop:SLappendix_hypergroup}
If \eqref{eq:SLappendix_mainassump} holds with $p \equiv r$ and $a = \gamma(a) = 0$, then $\bigl(\mathbb{R}_0^+,\convdiam{\ell}\bigr)$ is a commutative hypergroup (in the sense of \cite{bloomheyer1994,jewett1975}) with identity element $\delta_0$ and trivial involution, i.e.\ the following axioms hold:
\begin{itemize}[itemsep=0pt,topsep=3pt,leftmargin=0.7cm]
\item $\bigl(\mathbb{R}_0^+,\convdiam{\ell}\bigr)$, equipped with the total variation norm, is a commutative Banach algebra over $\mathbb{C}$ whose identity element is the measure $\delta_0$;
\item If $\mu, \nu \in \mathcal{P}(\mathbb{R}_0^+)$, then $\mu \convdiam{\ell} \nu \in \mathcal{P}(\mathbb{R}_0^+)$;
\item $(\mu,\nu) \mapsto \mu \convdiam{\ell} \nu$ is continuous in the weak topology of measures;
\item $(x_1,x_2) \mapsto \supp(\delta_{x_1} \convdiam{\ell} \delta_{x_2})$ is continuous from $\mathbb{R}_0^+ \times \mathbb{R}_0^+$ into the space of compact subsets of $\mathbb{R}_0^+$, and we have $0 \in \supp(\delta_{x_1} \convdiam{\ell} \delta_{x_2})$ if and only if $x_1 = x_2$.
\end{itemize}
\end{proposition}

\begin{proof}
See \cite[Subsection 5.3]{sousaetal2019b}.
\end{proof}

The measures $\{\mu_t\}_{t \geq 0} \subset \mathcal{P}[a,b)$ are said to be an \emph{$\ell$-convolution semigroup} if
\[
\mu_s \convdiam{\ell} \mu_t = \mu_{s+t} \text{ for all } s, t \geq 0, \qquad \mu_0 = \delta_a \qquad \text{ and } \;\; \mu_t \warrow \delta_a \text{ as } t \downarrow 0.
\]

\begin{proposition} \label{prop:SLappendix_convsemigr_props}
Let $\ell$ be a Sturm-Liouville expression with $p \equiv r$, $q \equiv 0$ and $a = \gamma(a) = 0$. Suppose that \eqref{eq:SLappendix_mainassump} holds and that $\lim_{x \uparrow b} p(x) = \infty$. \\[2pt]
\noindent\textbf{(a)} For $t > 0$, let $\alpha_t^\ell$ be the measure defined by $\alpha_t^\ell(dx) := p_\ell(t,0,x) r(x) dx$, where $p_\ell(t,x_1,x_2)$ is the kernel \eqref{eq:SLappendix_heatkern}. Set $\alpha_0^\ell = \delta_0$. Then $\{\alpha_t^\ell\}_{t \geq 0}$ is an $\ell$-convolution semigroup such that 
\[
\lim_{t \downarrow 0} {1 \over t} \alpha_t^\ell\mskip0.5\thinmuskip[\eps, \infty) = 0 \qquad \text{for every }\, \eps > 0. 
\]

\noindent\textbf{(b)} Let $\psi(\lambda)$ be a function which can be written as
\[
\psi(\lambda) = c\lambda + \int_{\mathbb{R}^+} (1-v_\lambda(x)) \, \tau(dx) \qquad (\lambda \geq 0)
\]
for some $c \geq 0$ and some positive measure $\tau$ on $\mathbb{R}^+$ such that $\tau$ is finite on the complement of any neighbourhood of $0$ and satisfies $\int_{\mathbb{R}^+} (1-v_\lambda(x)) \, \tau(dx) < \infty$ for $\lambda \geq 0$. Then there exists an $\ell$-convolution semigroup $\{\mu_t\}_{t \geq 0}$ such that $(\mathcal{F}_{\ell\,} \mu_t)(\lambda) = e^{-t\psi(\lambda)}$. Moreover, there exists a constant $C > 0$ independent of $\lambda$ such that
\begin{equation} \label{eq:SLaapendix_expon_ineq}
\psi(\lambda) \leq C(1+\lambda) \qquad \text{for all } \lambda \geq 0.
\end{equation}
\end{proposition}

\begin{proof}
Part (a) and the first statement in part (b) follow from \cite[Theorem 7.3 and Propositions 7.11--7.12]{sousaetal2020}.

To prove the estimate \eqref{eq:SLaapendix_expon_ineq}, start by picking $\lambda_1 > 0$. We know that $\lim_{x \uparrow b} v_{\lambda_1}(x) = 0$ (Lemma \ref{lem:SLappendix_solprops}(b)), hence there exists $\beta \in (a,b)$ such that $|v_{\lambda_1}(x)| \leq {1 \over 2}$ for all $\beta \leq x < b$. Combining this with Lemma \ref{lem:SLappendix_solprops}(a), we deduce that for all $\lambda \geq 0$ we have
\begin{equation} \label{eq:shypPDE_infdiv_bound_pf1}
\begin{aligned}
n \int_{[\beta,b)\!} \bigl(1-v_\lambda(x)\bigr) \mu_{1/n}(dx) & \leq 2n \int_{[\beta,b)\!} \mu_{1/n}(dx) \\
& \leq 4n \int_{[\beta,b)\!}\bigl( 1-v_{\lambda_1}(x) \bigr) \mu_{1/n}(dx) \\
& \leq 4n \bigl(1-(\mathcal{F}_\ell \mu_{1/n})(\lambda_1)\bigr) \leq 4\psi(\lambda_1).
\end{aligned}
\end{equation}
Next, choose $\lambda_2 > 0$ such that $1 - v_\lambda(x) < {1 \over 2}$ for all $0 \leq \lambda \leq \lambda_2$ and all $a < x \leq \beta$. (This is possible because of the boundedness of the family of derivatives  $\{\partial_\lambda v_{(\cdot)}(x)\}_{x \in (a, \beta]}$, cf.\ \cite[pp.\ 5--6]{sousaetal2019b}.) Defining $\eta_1(x) := \int_a^x {1 \over p(y)} \int_a^y r(\xi) d\xi\, dy$, we obtain
\[
1 - v_{\lambda_2}(x) = \lambda_2 \int_a^x {1 \over p(y)} \int_a^y v_{\lambda_2}(\xi) \mskip0.6\thinmuskip r(\xi) d\xi\, dy \geq {\lambda_2 \over 2} \eta_1(x) \qquad \text{for all } a \leq x \leq \beta.
\]
On the other hand, by Lemma \ref{lem:SLappendix_solprops}(a) we have $1 - v_\lambda(x) \leq \lambda \int_a^x {1 \over p(y)} \int_a^y |v_\lambda(\xi)| \mskip0.6\thinmuskip r(\xi) d\xi\, dy \leq \lambda \mskip0.6\thinmuskip \eta_1(x)$ for all $x \in [a,b)$ and $\lambda \geq 0$. Consequently,
\begin{equation} \label{eq:shypPDE_infdiv_bound_pf2}
\begin{aligned}
n \int_{[a,\beta)\!} \bigl(1-v_\lambda(x)\bigr) \mu_{1/n}(dx) & \leq \lambda n \int_{[a,\beta)\!} \eta_1(x)\, \mu_{1/n}(dx) \\
& \leq {2\lambda n \over \lambda_2} \int_{[a,\beta)\!} \bigl(1 - v_{\lambda_2}(x)\bigr) \mu_{1/n}(dx) \\
& \leq {2\lambda n \over \lambda_2} \bigl(1-(\mathcal{F}_\ell \mu_{1/n})(\lambda_2)\bigr) \leq {2\lambda \over \lambda_2} \psi(\lambda_2).
\end{aligned}
\end{equation}
Combining \eqref{eq:shypPDE_infdiv_bound_pf1} and \eqref{eq:shypPDE_infdiv_bound_pf2} one sees that for all $n \in \mathbb{N}$ and $\lambda \geq 0$ we have $n(1-e^{-\psi(\lambda)/n}) \leq C (1+\lambda)$, where $C = \max\bigl\{4\psi_\mu(\lambda_1), {2 \over \lambda_2} \psi_\mu(\lambda_2)\bigr\}$. The conclusion follows by taking the limit as $n \to \infty$.
\end{proof}

\section*{Acknowledgements}

The first and third authors were partially supported by CMUP, which is financed by national funds through FCT – Fundação para a Ciência e a Tecnologia, I.P., under the project with reference\linebreak UIDB/00144/2020. The first author was also supported by the grant PD/BD/135281/2017, under the FCT PhD Programme UC|UP MATH PhD Program. The second author was partially supported by the project CEMAPRE/REM – UIDB/05069/2020 – financed by FCT through national funds.

\linespread{1.125}
\renewcommand{\bibname}{References} 
\begin{small}

\end{small}

\end{document}